\documentclass[11pt]{article}
\usepackage{graphicx}
\usepackage{amssymb,amsmath}

\usepackage{bm}
\newtheorem{remark}{Remark}[section]
\newtheorem{proposition}{Proposition}[section]

\newtheorem{lemma}[proposition]{Lemma}
\newtheorem{theorem}[proposition]{Theorem}
\newcommand{\END}{\hfill$\Box$}

\numberwithin{equation}{section}

\newcommand{\R}{\mathbb{R}}
\newcommand{\C}{\mathbb{C}}
\newcommand{\Z}{\mathbb{Z}}
\newcommand{\ds}{\displaystyle}

\newcommand{\supp}{\mathop{\rm supp}}
\renewcommand{\ds}{\displaystyle}
\renewcommand{\r}{{\bf r}}
\renewcommand{\H}{{H}}

\newenvironment{proof}{\noindent{\it Proof}. }{\hfill\END\\[0.5ex]}
\title{Convergence analysis of a high-order Nystr\"{o}m
  integral-equation method for surface scattering problems}
\date{\today} \author{Oscar P. Bruno\footnote{Applied and
    Computational Mathematics, California Institute of Technology,
    Pasadena, CA 91125, USA -- {\tt bruno@acm.caltech.edu}},
  V\'{\i}ctor Dom\'{\i}nguez\footnote{ Dep. Ingenier\'{\i}a
    Matem\'{a}tica e Inform\'{a}tica, E.T.S.I.I.T., Universidad
    P\'{u}blica de Navarra. 31500 - Tudela, Spain -- {\tt
      victor.dominguez@unavarra.es}}\, \& Francisco-Javier
  Sayas\footnote{ Department of Mathematical Sciences, University of
    Delaware, Newark, DE 19716, USA -- {\tt
      fjsayas@math.udel.edu}}} 
      
\begin{document}
\maketitle

\thispagestyle{empty}

\begin{abstract}
  In this paper we present a convergence analysis for the Nystr\"om
  method proposed in [Jour. Comput. Phys. 169 pp. 2921-2934, 2001] for
  the solution of the combined boundary integral equation formulations
  of sound-soft acoustic scattering problems in three-dimensional
  space. This fast and efficient scheme combines FFT techniques and a
  polar change of variables that cancels out the kernel singularity.
  We establish the stability of the algorithms in the $L^2$ norm and
  we derive convergence estimates in both the $L^2$ and $L^\infty$
  norms. In particular, our analysis establishes theoretically the
  previously observed super-algebraic convergence of the method in
  cases in which the right-hand side is smooth.
\end{abstract}

\section{Introduction}

In this paper we present a full convergence analysis of the high-order
Nystr\"{o}m method introduced by Bruno and Kunyansky in
\cite{377360,MR1875086} for boundary integral equations (BIE) related
to the scattering of time-harmonic acoustic waves by arbitrary
(smooth) surfaces in three-dimensional space.  In particular, our
analysis establishes theoretically the previously observed
super-algebraic convergence of the method for smooth right-hand sides.
To the best of the authors' knowledge, this proof constitutes one of
the few instances of analysis of a Nystr\"{o}m type method for a
boundary integral equation in three dimensions.

As is well known, Galerkin methods for BIE of the first kind have
enjoyed thorough theoretical analyses since their inception---on the
basis of ellipticity properties and discrete Fredholm
theory. Compactness arguments can also be used to establish
convergence of Galerkin methods for equations of the second kind.  Few
results exist, on the other hand, concerning convergence for
three-dimensional BIE collocation methods---in which finite-element
bases are used for approximation, but testing relies on point
sampling.  We refer to the excellent text-book~\cite[Chapter
9]{atkinson} for a brief introduction on this topic.

This state of affairs has led to the widespread perception that, being
even more ``discrete'' than collocation schemes, Nystr\"{o}m methods
for BIEs of the second kind with weakly singular kernels could not be
easily analyzed.  This paper will hopefully help dispell this view and
lend additional credibility to Nystr\"om methods---whose qualities can
be very attractive in practice~\cite{377360}.

One of the main difficulties in the design of three-dimensional
integral solvers concerns development of high-quality quadrature rules
for approximation of singular integral terms.  Wienert~\cite{Wienert}
constructed a singular integration rule on the basis of
spherical-harmonic transforms for surfaces for which a diffeomorphism
to the sphere can be constructed, (see also \cite[Section
3.6]{CoKr:1998}). A Galerkin version of this approach was introduced
and analyzed in~\cite{MR1922922} where, in particular, the
superalgebraic convergence of the method was
established. Reference~\cite{MR1618899} presents a collocation method,
with corresponding analysis, for the Laplace equation, which shares
the good convergence properties of the method by Wienert but which
again is limited to surfaces for which a smooth mapping from the
sphere is available. We emphasize that such limitations, which are
highly restrictive in practice, are not imposed by the Nystr\"om
method studied in this paper.

We thus consider the numerical method~\cite{377360} for the
second-kind combined field integral equation associated to the
problem of sound-soft scattering by a three dimensional obstacle
with smooth boundary $S$.  The method relies on a series of
geometric constructions: 1) Representation of the surface by means
of a set of $J$ overlapping smooth charts (parametrizations); 2) A
smooth partition of unity subordinated to these charts which
decomposes both the overall integral operators as well as the
solution of the BIE as the sum of contributions defined on the
parametrized patches; 3) A family of floating smooth cut-off
functions $\eta_\delta: S \times S \to \mathbb R$ (see
\eqref{eq:eta} below) that is used to isolate the singularity of the
kernel function, and thus produces a splitting of the integral
operator of the form
$K=K_{\mathrm{reg},\delta}+K_{\mathrm{sing},\delta}$---whereby the
regular operator is an integral operator with a smooth kernel, and
the singular part which enjoys a reduced domain of integration but
contains the weak singularity of the integral kernel.

The Nystr\"{o}m method under consideration is based on approximation
of the integral operator by a quadrature rule which treats the regular
and singular parts $K_{\mathrm{reg},\delta}$ and $
K_{\mathrm{sing},\delta}$ separately. The $J$ local charts mentioned
above are used to push forward $J$ uniform grids from the unit square
to the surface; these grids are used for both, approximation and
integration.  The quadrature rule used for the regular operator
$K_{\mathrm{reg},\delta}$ is based on application of the
two-dimensional trapezoidal rule in the parameter space $[0,1]\times
[0,1]$ for each one of the $J$ parametric domains: since the
corresponding integrands are smooth with compact support strictly
contained in the unit square, these trapezoidal quadratures give rise
to super-algebraically accurate approximations. For the singular part,
a change of variables to polar coordinates around each integration
point is applied. This procedure results in a smooth integrand to
which, upon necessary Fourier-based interpolations, the trapezoidal
rule is applied for radial and angular integration---once again
yielding super-algebraically accurate approximations.

As a result of these constructions we obtain an algorithm that
evaluates the action of an approximating discrete operator on a
continuous function, using only the values of the function on the
selected quadrature points. 

%The algorithm can then be accelerated by
%some of the available acceleration techniques: equivalent sources and
%FFT techniques (an approach that was already considered
%in~\cite{377360}), Fast Multipole Method, etc.; issues associated with
%such acceleration, however, will not be considered in the present
%paper.

Our theoretical treatment relies on use of both existing and new
analytical tools. In a first key step of the analysis the problem is
re-expressed as a system of integral equations {\em in a space of
  periodic functions}. This is accomplished by means of yet another
set of local cut-off functions, whose presence does not affect the
actual numerical scheme. The use of periodic Sobolev spaces allows us
to take full advantage of numerous results for approximation of
Fourier series and interpolation operators~\cite{SaVa:2002}, as well
as the theory of {\em collective compact operators} by
Anselone~\cite{MR0443383}.  Recasting the numerical scheme of
quadrature type as a discretization method in $L^2$-type (Sobolev)
spaces gives rise to a number of difficulties. In particular, for the
sake of the analysis we introduce Fourier projection operators,
bi-periodic trigonometric interpolation operators, and a discrete
operator~\cite{CeDoSa:2002} that produces a linear combination of {\em
  Dirac delta distributions} on a uniform grid from a smooth function
input.  The convergence analysis for the operators arising from the
regular part of the original boundary integral operator follows the
lines of the theory~\cite{CeDoSa:2002,DoRaSa:2008} on periodic
integral equations in one variable.  The final (rather technical)
element of our convergence proof is a detailed analysis of the
integration error arising from the numerical polar coordinate
integration algorithm for products of smooth functions, ``shrinking''
floating cut-off functions and bi-variate trigonometric polynomials,
in terms of Sobolev norms of the latter polynomials.

We point out that, for efficiency, a variety of
  acceleration techniques were used in~\cite{377360,MR1875086} in
  conjunction with the Nystr\"om algorithm we consider. Some of these
  algorithmic components have been taken into account in our
  analysis. In the formulation considered in the present paper, for
  example, the computation of the singular part requires the
  evaluation of bivariate trigonometric polynomials that approximate
  the unknown at points on a polar grid. A deliberate choice of the
  radial quadrature nodes for this integral makes it possible to
  reduce this process to 1--dimensional trigonometric interpolation
  problems (see section~\ref{disc_int_op} and Figure
  \ref{fig:points03} below) on the horizontal and vertical lines of
  the grid. In~\cite{377360}, such trigonometric polynomials are then
  approximated by means of piecewise Hermite interpolation on dyadic
  grids, which can be evaluated much more rapidly than the either of
  the underlying trigonometric polynomials. We have analyzed the
  effect of these additional approximations on the full convergence of
  the Nystr\"om method. The corresponding results can be found
  in Appendix B; briefly, upon correct parameter
  selections, the resulting (more efficient) method retains the
  super-algebraic convergence of the original approach.  Additional,
  more sophisticated acceleration techniques which, based on use of
  equivalent sources and Fast Fourier Transforms, provide a means to
  reduce the solution cost for high-frequency problems (but on which
  the Nystr\"om method itself does not depend, and whose use is not
  advantageous for problems of lower frequencies) were introduced
  in~\cite{377360}. The impact of such equivalent-source acceleration
  methods on convergence are not considered either in the present
  paper or in the 
  Appendix \ref{sec:B} of this paper.

This paper is structured as follows. In Section \ref{sec:2} we
describe, in a compact form, the Nystr\"om method under
consideration, and we state the main convergence results of this
paper. In Section \ref{sec:3} we then recast both the continuous and
discrete problems as systems of equations in spaces of biperiodic
functions, we derive bounds, on various norms, of the main
continuous integral operators in our periodic formulation, we
establish unique solvability of the continuous system of periodic
integral equations as well as the equivalence of this system to the
original BIE, and, finally, we state the main approximation results
in the biperiodic frame: norm convergence of the discrete operators
to the continuous ones together with corresponding error estimates. In Section
\ref{sec:4} we present key estimates on the singular
operators that appear in the biperiodic formulation, and in Section
\ref{sec:5}, in turn, we provide the proofs of the main results
stated in Section 3 and Section \ref{sec:2}---in that order. 
Appendix \ref{sec:A}  is devoted to the Sobolev error analysis mentioned above
for the polar-integration of products of smooth functions, cut-offs
and trigonometric polynomials. Finally, in Appendix \ref{sec:B} we describe
and
analyze a slight variant of the numerical method, where one of the interpolation
processes is substituted by polynomial interpolation.

% In Section \ref{sec:B} we describe and analyze a slight variant of the
% numerical method, where one of the interpolation processes is
% substituted by polynomial interpolation. Appendix \ref{sec:A},
% finally, is devoted to the Sobolev error analysis mentioned above
% for the polar-integration of products of smooth functions, cut-offs
% and trigonometric polynomials.

We conclude our introduction with two remarks concerning notation.
\begin{remark}\label{notation1}
  To make a clearly visible distinction between points on the surface
  $S$ and coordinates for their parametrization, we use boldface
  letters (e.g.  $\r$) for points on the scattering surface $S\subset
  \R^3$, and underlined letters (such as $\underline u$) for points in
  $\R^2$. The coordinates of such points will be denoted according to
  $\underline{u}=(u_1,u_2)$.
\end{remark}
\begin{remark}\label{notation2}
  Throughout this paper the letter $C$ denotes a positive constant
  independent of the parameters $h$ and $\delta$ and any other
  variable quantities appearing in the equation. When necessary a
  subscripted letter $C$ is used---either to avoid confusion or to
  explicitly signify dependence on parameters other than $h$ and
  $\delta$.
\end{remark}

\section{The Nystr\"om method}\label{sec:2}

\subsection{The Boundary Integral Equation}
We consider the problem of time-harmonic acoustic scattering by a
sound-soft obstacle with smooth boundary $S$ in three-dimensional
space:
\begin{equation}
\label{eq:2.1} \left[
\begin{array}{rcll}
 \Delta U+\kappa^2 U&=&0\quad&\mbox{in $\mathrm{ext}(S)$},\\[1.1ex]
 U&=&-U^{\rm inc}&  \mbox{on $S$},\\[1.1ex]
 \partial_r  U-{\rm i} \kappa  U&=& o(r^{-1})\qquad&
\mbox{as }r\to\infty.
\end{array}
\right.
\end{equation}
Here $\kappa>0$ is the wave number, $U^{\rm inc}$ is the incident
wave, $r := |{\bf x} |$ and $\partial_r $ denotes the radial
derivative. Letting $\Phi_\kappa$ denote the fundamental solution of
the Helmholtz equation,
\[
 \Phi_{\kappa} (\r,\r'):=\frac{\exp({\rm i}\kappa|\r-\r'|)}{4\pi\,|\r-\r'|},
\]
then the solution $U$ of \eqref{eq:2.1} can be expressed as the
combined (or Brakhage-Werner \cite{BrWe1965}) potential
\[
U(\r)=\int_S \frac{\partial \Phi_{\kappa}(\r,\r')}{\partial
\bm{\nu}(\r')} \psi(\r'){\rm d}\r'-{\rm i}\eta\int_S
\Phi_{\kappa}(\r,\r')\psi(\r'){\rm d}\r',
\]
where $\frac{\partial }{\partial \bm{\nu}}$ denotes the outward normal
derivative on $S$ and $\eta > 0$ is a coupling parameter. The density
$\psi$ is the unique solution of the integral equation
\begin{equation}
\label{eq:intEqn} {\textstyle\frac{1}{2}}\psi(\r)+A\psi(\r)=-U^{\rm
inc}(\r) \qquad \forall \r\in S,
\end{equation}
where
\begin{equation}
\label{eq:defA}
 A\psi(\r)=\int_S K(\r,\r')\psi(\r')\,{\rm d}\r':=\int_S
\Big(\frac{\partial \Phi_{\kappa}(\r,\r')}{\partial \bm{\nu}(\r')}-{\rm
i}{\eta} \Phi_{\kappa}(\r,\r')\Big)\psi(\r'){\rm d}\r'.
\end{equation}
Various choices of the coupling parameter $\eta$ have been proposed
for accuracy and numerical efficiency; see
e.g.~\cite{MR2272197,ChandlerMonk:2007,DoGrSm:2007,377360}.  Note that
the kernel $K$ can be expressed in the form $K = K_0+K_1$, where
\begin{equation}\label{K0}  K_0(\mathbf r,\mathbf r'):=\eta
\frac{\sin(\kappa|\mathbf r-\mathbf r'|)}{4\pi|\mathbf r-\mathbf
r'|}+\mathrm i\left(\kappa \cos(\kappa |\mathbf r-\mathbf r'|)
-\frac{\sin(\kappa|\mathbf r-\mathbf r'|)}{|\mathbf r-\mathbf
r'|}\right) \frac{(\mathbf r-\mathbf r')\cdot\boldsymbol\nu(\mathbf
r')}{4\pi|\mathbf r-\mathbf r'|^2}
\end{equation}
and
\begin{equation}
\label{K1} K_1(\mathbf r,\mathbf r'):=-\mathrm
i\eta\frac{\cos(\kappa|\mathbf r-\mathbf r'|)}{4\pi|\mathbf
r-\mathbf r'|}- \frac{\kappa |\mathbf r-\mathbf r'|\sin
(\kappa|\mathbf r-\mathbf r'|)+\cos(\kappa|\mathbf r-\mathbf
r'|)}{4\pi|\mathbf r-\mathbf r'|}\frac{(\mathbf r-\mathbf
r')\cdot\boldsymbol\nu(\mathbf r')}{|\mathbf r-\mathbf r'|^2}.
\end{equation}
\begin{remark}\label{kernel_reg}
It is easy to check that $K_0\in\mathcal {C}^\infty(S\times S)$. The
kernel $K_1$, on the other hand, while not  $\mathcal
{C}^\infty(S\times S)$, can be integrated
with super-algebraic accuracy by means of a combination of a polar
change of variables and the trapezoidal rule; see~\cite{377360} and
equation~\eqref{polar}.
\end{remark}

In this paper we focus on the Brakhage-Werner formulation presented
above; a similar analysis can be used to treat the closely related
Burton-Miller formulation~\cite{BurMil:1971}.

\subsection{Geometry\label{geometry}}

The numerical method studied in this work relies heavily on the use
of a system of local charts for description of the surface $S$. We
will thus use a set of a number $J$ of open overlapping coordinate
patches $\{{ S}^j\}_{j=1}^J$ that cover $S$,
\begin{equation}\label{cover}
S=\bigcup_{j=1}^J { S}^j,
\end{equation}
each one of which is the image of a  $\mathcal
{C}^\infty$ parametrization
\[
 \begin{array}{rccl}
   \r ^j:&D^j&\longrightarrow &S^j\\
   &\underline{u}:=(u_1,u_2)&\longmapsto&\r^j(u_1,u_2),
 \end{array}
\]
where $\overline{D^j}\subset I_2:=(0,1)\times(0,1)$. We assume that
$\r^j$ can be extended to a  $\mathcal
{C}^\infty$ bijective diffeomorphism between
$\overline{D^j}$ and $\overline{S^j}$ so that, in particular, the
Jacobians
\begin{equation}\label{Jac}
  a^j(\underline{u}):= \bigg|\frac{\partial \r ^j(u_1,u_2)}{\partial
    u_1}\times \frac{\partial \r ^j(u_1,u_2)}{\partial u_2}\bigg|>{
    0},\qquad j=1,\ldots,J
\end{equation}
($a^j:D^j \to \mathbb R$) are $\mathcal
{C}^\infty$ functions of $\underline{u}$.

The method requires explicit use of a smooth partition of unity
$\{\omega^j\}_{j=1}^J\subset{\cal C}^\infty(S)$ on $S$, subordinated
to the covering~\eqref{cover}, that is
\[
\omega^j \ge 0, \qquad \supp\,\omega^j \subset {S}^j\quad
\mbox{and}\quad \sum_{j=1}^J \omega^j \equiv 1.
\]
The hypotheses on availability of local charts and a partition of
unity is not restrictive in practice: such parameterizations can be
constructed for smooth arbitrary geometries (see
e.g.~\cite{1322676}). We will also assume that the boundary of $\supp \omega^j$
is the finite union of Lipschitz arcs, a restriction which, again, is easy to
accommodate~\cite{1322676}.

For any $\delta>0$ and $\r\in S$ we let
\[
B({\bf r},\delta):=\{ {\bf r}'\in S\: :\: |\r'-\r|<\delta\}
\]
and, selecting parameters $0< \epsilon_0<\epsilon_1\le 1$ that will
otherwise be fixed throughout this paper, we define
\begin{equation}
\label{eq:2.8} S^j_{\delta}:=\overline{\bigcup_{\r\in\supp\omega^j}
B(\r,2\epsilon_1\delta)}=\bigcup_{\r\in
\supp\omega^j}\overline{B(\r,2\epsilon_1\delta)}.
\end{equation}
Clearly there exists $\delta_0>0$ such that for all $j$
\begin{equation}
\label{eq:requirement_on_delta}
\overline{B(\r,\epsilon_1\delta_0)}\cap
S^j_{2\delta_0}=\emptyset\qquad \forall\r\in S\setminus S^j
\end{equation}
---as it can be checked by considering a pair of points that realize
the distance between the boundaries of $\supp\omega^j$ and $S^j$. In
particular, this implies that $S^j_{2\delta_0}\subset S^j$.  The final
element in our geometric constructions is a $\mathcal
{C}^\infty$ function
$\upsilon:S\times[0,\infty)\to [0,1]$ such that
\[
\begin{array}{l}
\upsilon(\r,\,\cdot\,) \equiv 1 \mbox{\quad in $
[0,\epsilon_0\delta_0]$} \\[1.5ex] \upsilon(\r,\,\cdot\,) \equiv 0
\mbox{\quad in $ [\epsilon_1\delta_0,\infty)$}\end{array} \qquad
\forall \r \in S.
\]

Given $\delta\in(0,\delta_0]$ we now define the functions
$\eta_\delta:S \times S \to [0,1]$
\begin{equation}
\label{eq:eta} \eta_\delta({\bf r},{\bf r}'):=\upsilon \Big({\bf
r},\frac{\delta_0|{\bf r}'-{\bf r}|}{\delta}\Big).
\end{equation}
Clearly
\begin{equation}
\label{eq:suppetadelta} \supp \eta_\delta(\r,\,\cdot\,) \subset
\overline{B(\r,\epsilon_1\delta)}, \qquad \forall \r \in S.
\end{equation}
In view of the definition of the sets $S_{\delta}^j$ and
\eqref{eq:requirement_on_delta} we also have
\begin{equation}
 \label{eq:requirement_on_delta2}
\supp\eta_{\delta}(\mathbf{r},\cdot)\cap
S^j_{3\delta/2}=\emptyset\qquad \forall \mathbf{r}\in S\setminus
S^j_{2\delta}.
\end{equation}

Given $\delta$ and considering the decomposition $K=K_0+K_1$ in term
of the kernel functions given in equations~\eqref{K0}
and~\eqref{K1}, we define the regular part of the kernel of the
integral operator \eqref{eq:defA} as
\begin{equation}
\label{eq:Kreg} K_{\rm reg,\delta}(\r,\r'):= K_0(\r,\r')+
\big(1-\eta_\delta(\r, \r')\big) \, K_1(\r,\r').
\end{equation}
Clearly $K_{\rm reg,\delta}\in \mathcal C^\infty(S\times S)$. The
remainder is the singular part of the kernel,
\begin{equation}
\label{eq:Ksing} K_{\rm sing,\delta}:= K - K_{\rm reg,\delta}=
\eta_\delta   \, K_1,
\end{equation}
which, like the kernel $K_1$, can be integrated accurately by means of
a polar change of variables; see Remark~\ref{kernel_reg}.
The parameter $\delta$, which controls the support of
  the kernel $K_{\rm sing,\delta}(\r,\,\cdot\,)$, plays an essential
  role in both, the performance of the algorithm and its theoretical
  analysis; see Remark \ref{remark:delta:clarified} for details.

We next introduce
\[
K^{ij}(\underline u,\underline v):= K(\r^i(\underline
u),\r^j(\underline v)) \,a^j(\underline v),
\]
(see equation~\eqref{Jac}) with corresponding definitions of
$K^{ij}_{\rm sing,\delta}$ and $K^{ij}_{\rm reg,\delta}$ (cf.
equations~\eqref{eq:Kreg} and~\eqref{eq:Ksing}). Noting that
$K^{ij}$, $K^{ij}_{\rm sing,\delta}$ and $K^{ij}_{\rm
  reg,\delta}$, are defined in $D^i \times D^j$, we extend these
functions by zero (possibly thereby introducing discontinuities on the
boundary of $D^j$) to the full product of squares $I_2\times
I_2$. Clearly,
\[
 \int_S K(\r^i(\underline u),\r')\xi(\r')\,{\rm d}\r'=\sum_{j=1}^J
\int_{D^j}  K^{ij}(\underline u,\underline
v)\omega^j(\r^j(\underline v))\xi (\r^j(\underline v))\,{\rm
d}\underline v \qquad \forall \underline u \in D^i.
\]
Therefore, if $\psi$ is the exact solution of \eqref{eq:intEqn},
then $\psi^i:=\psi\circ \r^i :D^i\to \C$ ($i=1,\ldots,J$), is a
solution of the system
\begin{eqnarray}\nonumber
\frac{1}{2}\psi^i(\underline u)+ \sum_{j=1}^J \int_{I_2} K^{ij}_{\rm
reg,\delta}(\underline u,\underline v) \omega^j(\r^j(\underline
v))\psi^j(\underline v)\,{\rm
d}\underline v\\
+\sum_{j=1}^J \int_{I_2}  K^{ij}_{\rm sing,\delta}(\underline
u,\underline v)  \omega^j(\r^j(\underline v)) \psi^j(\underline
v)\,{\rm d}\underline v&=&
- U^{\rm inc}(\r^i(\underline u))\nonumber\\
&& \forall \underline u \in D^i,\quad i=1,\ldots, J.\qquad
\label{eq:ExactSystem}
\end{eqnarray}
\begin{remark}\label{dom_def}
Note that the functions $\psi^j$ that appear in the integrals over $I_2$ in
equation~\eqref{eq:ExactSystem} are only defined in $D^j$. However, they are
multiplied by the cutoff function $\omega^j\circ \r^j$ which vanishes outside
$D^j$ and thus provides a natural extension for the product throughout $I_2$.
\end{remark}
The solution of this system is used in Section~\ref{sec:3} to
reconstruct the solution of the original problem~\eqref{eq:intEqn}.

We can clearly distinguish two different types of integral operators
in~\eqref{eq:ExactSystem}, namely, integral operators with smooth and
singular kernels. The discretizations of these operators are produced,
accordingly, by means of two different strategies---as discussed in
the following section.

\subsection{Discretization of the integral operators in
  equation~\eqref{eq:ExactSystem}\label{disc_int_op}}
For each patch $S^j$, we select a positive integer $N_j$, we let
$h_j:=1/N_j$, and we introduce the grid points
\[
\underline x^j_{\underline m} := h_j \underline m= h_j\, (m_1,m_2),
\qquad 0\le m_1,m_2 \le N_j-1.
\]
We assume that these grids are quasiuniform: letting
\[
N:=\min_{j=1,\ldots,J} N_j, \qquad h:=1/N=\max_{j=1,\ldots,J} h_j,
\]
we assume there exists $c>0$ such that
\begin{equation}
\label{eq:quasiuniform}
h<c h_j,\qquad j=1,\ldots, J.
\end{equation}
This is not a restrictive assumption in view of the assumed smoothness
of $S$ and $U^{\rm inc}$, and, therefore, of the solution $(\psi^j)$.

\begin{figure}[htp]
\centerline{\includegraphics[width=0.8\textwidth]{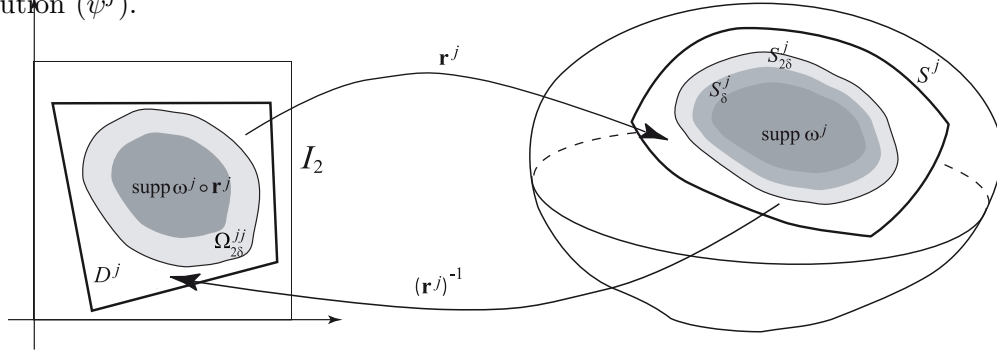}}
\caption{Sketch of the  geometric layout of a single chart and
its associated cut-off function.}\label{fig:geometry}
\end{figure}

The algorithm under consideration is a Nystr\"{o}m method which
produces point-wise values of the unknown $(\psi^j)$ at the
discretization points $\underline x_{\underline m}^j$. The quadrature
rules that ultimately define the method are described in what follows;
for notational simplicity our approximate quadrature formulae use the
set
\begin{equation}\label{eq:defZNj}
  \mathbb Z_{N_j}:=\big\{ \underline{m}:=(m_1,m_2)\in \Z^2
  :  0\le m_1,m_2\le N_j-1\big\}
\end{equation}
of two-dimensional summation indices $\underline m$.

The approximate integration method used by the algorithm to treat the
{\em regular} portion of the kernel is, simply, based on the
trapezoidal rule,
\begin{equation}\label{eq:regularpart}
\int_{I_2} K^{ij}_{{\rm reg},\delta}(\,\cdot\,,
 \underline v) \xi(\underline v){\rm d}\underline v
\approx h_j^2\sum_{\underline{m}\in \Z_{N_j}} \!
 K^{ij}_{{\rm reg},\delta}(\,\cdot\,,
\underline x^j_{\underline m}) \xi (\underline x^j_{\underline m}),
\end{equation}
with the convention that $K^{ij}_{{\rm reg},\delta}(\,\cdot\,,
\underline v)\equiv 0$ for $\underline v\not\in D^j$. Notice that for
a  $\mathcal
{C}^\infty$ function $\xi$ compactly supported in
\begin{equation}
\label{eq:defwj}
\Omega^j:=(\r ^{j})^{-1}(\supp \omega^j)\subset D^j,
\end{equation}
(such as $\xi=(\omega^j\psi)\circ \r^j$ with $\psi\in {\cal
  C}^\infty(S)$) the rule embodied in equation~\eqref{eq:regularpart}
gives rise to super-algebraic convergence.

\begin{figure}[htb]
\centerline{\includegraphics[width=0.8\textwidth]{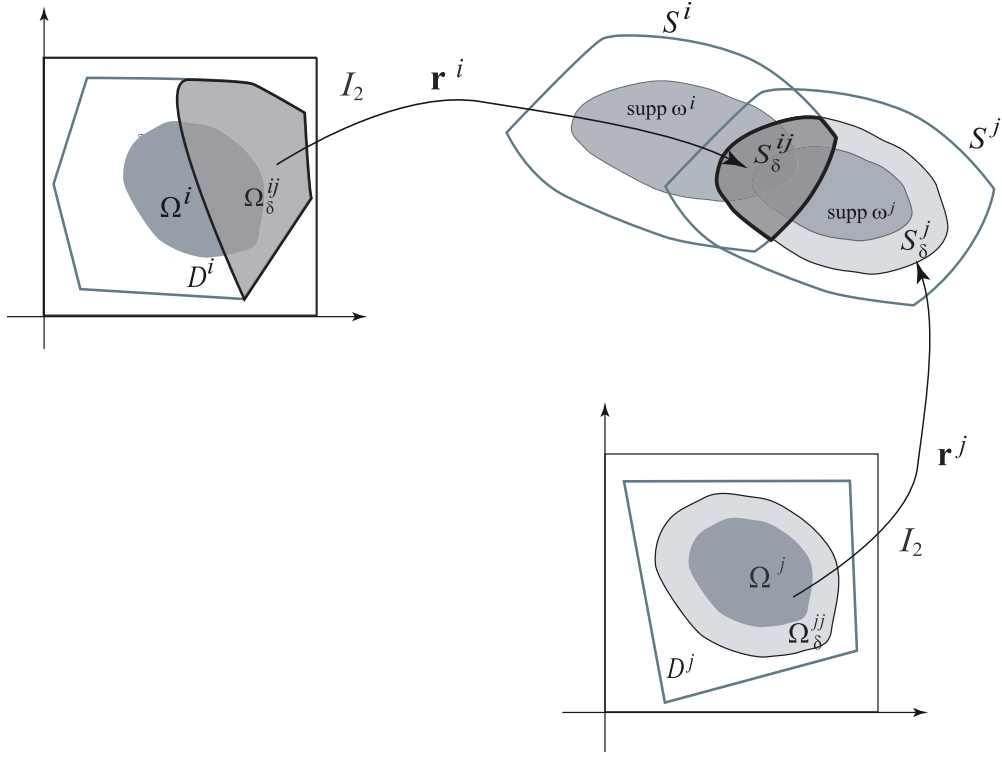}}
\caption{The overlapping of two charts and the associated
domains.}\label{fig:geometry2}
\end{figure}
To approximate integrals of the form
\begin{equation}
\label{eq:singularPart}
 \int_{I_2}  K^{ij}_{{\rm sing},\delta}(\underline u,
\underline v)  \xi(\underline v) \mathrm d \underline v, \qquad \xi
\in \mathcal C^\infty(I_2), \quad \supp \xi \subseteq \Omega^j,
\end{equation}
that include the {\em singular} part of the kernel, the algorithm uses
polar coordinates around the singularity. To properly account for
contributions arising from various regions delineated by local charts
and overlaps we let
\begin{equation}
\label{eq:defSij}
\begin{array}{c}
  \underline{r}^{ji}:=(\r ^{j})^{-1} \circ\r ^i:(\r ^i)^{-1}(S^i\cap
  S^j)\subset D^i\to D^j \\[1.5ex]
  S^{ij}_{\delta}:=S^i
  \cap S^j_{\delta} \subset S^i\cap S^j,\qquad
  \Omega^{ij}_{\delta}:=(\r^{i})^{-1}(S^i \cap
  S^j_{\delta})=(\r^i)^{-1} (S^{ij}_\delta);
\end{array}
\end{equation}
see Figure~\ref{fig:geometry2} and equation~\eqref{eq:2.8}. A close
inspection of the definition of the sets $S^j_\delta$ shows that
$\overline{B(\r,\epsilon_1\delta)}\cap \supp \omega^j=\emptyset$ if
$\r \not\in S^j_{\delta/2}$. Since, additionally,
\[
\supp K^{ij}_{{\rm sing},\delta}(\underline u, \,\cdot\,) \subset
\supp \eta_\delta(\r^i(\underline u), \r^j(\,\cdot\,))\subset
 (\r^{j})^{-1}\big( \overline{B(\r^i(\underline
u),\epsilon_1\delta)}\cap S^j\big),
\]
we conclude that the integral \eqref{eq:singularPart} with $\supp\xi
\subseteq \Omega^j$, vanishes for all $\underline u$ such that
$\r^i(\underline u)\not\in S^j_{\delta/2}$, i.e., for $\underline
u\not\in \Omega_{\delta/2}^{ij}$. The algorithm therefore only
produces approximations of~\eqref{eq:singularPart} for $\underline
u\in \Omega_{\delta/2}^{ij}$. To do this consider the relations
\begin{eqnarray}
\int_{\R^2} K^{ij}_{{\rm sing},\delta}( \underline u, \underline v)
\xi(\underline v)
 \mathrm d \underline v\!&&\nonumber\\
&&\hspace{-4cm}= \int_{\mathbb R^2}
 K_{{\rm sing},\delta}^{ij}\left( \underline u,
\underline r^{ji}(\underline u) + \underline w\right)   \xi
(\underline r^{ji}(\underline u)+\underline w)\, \mathrm d
\underline
w\nonumber\\
&&\hspace{-4cm}=
\frac{1}{2}\int_{0}^{2\pi}\bigg(\int_{-\infty}^\infty
|\rho|\,K_{{\rm sing},\delta}^{ij}\left( \underline u, \underline
r^{ji}(\underline u) + \rho \underline{e}(\theta)\right)
 \,\xi(\underline r^{ji}(\underline u)+ \rho
\underline{e}(\theta))\,
\mathrm d \rho\bigg) \mathrm d \theta\label{eq:singularPart2}
 \end{eqnarray}
 where $\underline e(\theta):=(\cos \theta,\sin\theta)$ and note that
 the function
\begin{equation}\label{polar}
(\rho,\theta) \mapsto  |\rho| K_{{\rm
sing},\delta}^{ij}\left(\underline{u}^i,\, \underline
 r^{ji}( \underline{u} ) + \rho \underline{e}(\theta)\right)
\end{equation}
is smooth (as shown in Section~\ref{sec:4},
cf. \cite{377360,MR1875086}), $2\pi$-periodic in $\theta$ and
compactly supported as a function of the variable $\rho$.  In what
follows we temporarily let $\underline z=(z_1,z_2):=\underline
r^{ji}(\underline u)$, so that the dependences on the patch indices
$(i,j)$ and parametric coordinate $\underline u$ are not displayed.

The integral in the angular variable is approximated using a
trapezoidal rule on a uniform partition of $[0,2\pi]$ in $\Theta_j$
subintervals of length $ k_j=2\pi/\Theta_j$. Therefore, we consider
the angles
\[
\theta_p^j:= p\, k_j, \qquad p=0,\ldots, \Theta_j-1
\]
and approximate \eqref{eq:singularPart2} by
\begin{equation}\label{eq:sing:01}
\frac{ k_j}2\sum_{p=0}^{\Theta_j-1} \Bigg( \int_{-\infty}^\infty
K^{ij}_{\rm sing,\delta} (\underline u,\underline z+
\rho\,\underline e(\theta_p^j) )\, |\rho|\,
 \xi (\underline z+\rho\, \underline
e(\theta_p^j) ) \mathrm d \rho\Bigg).
\end{equation}
Recall that a uniform Cartesian grid in $I_2$ with mesh length $h_j$
has been introduced in the approximation of the regular part
\eqref{eq:regularpart}. We will now use points of this grid to
approximate the radial integrals  in \eqref{eq:sing:01}.

If $|\cos \theta_p^j|\ge \sqrt2/2$, that is, modulo $2\pi$,
$\theta_p^j \in [-\pi/4,\pi/4] \cup [3\pi/4, 5\pi /4]$, we look for
the intersections
\[
\{ \underline x= \underline z+\rho\, e(\theta_p^j)\, :\, \rho \in
\mathbb R\}\cap \{ \underline x= ( q\, h_j,\xi)\, :\, \xi \in
\mathbb R\}, \qquad q=0,\ldots, N_j-1,
\]
i.e., the intersections of the double rays stemming from $\underline
z$ with angle $\pm\theta_p^j$ with the vertical lines of the uniform
grid: the points of intersections are given by
\begin{equation}
\label{eq:points01} \big( q\, h_j, z_2+( q\,h_j-z_1) \tan
\theta_p^j\big),\qquad q=0,1,\ldots, N_j-1;
\end{equation}
clearly, the distance between two consecutive points is
$h_j/|\cos\theta_p^j|\le \sqrt{2}h_j.$ Alternatively we can express
these points in the form
\[
\underline z + \rho_q^{h_j} (\underline z,\theta_p^j)\, \underline
e(\theta_p^j), \qquad \mbox{where}\qquad \rho_q^h (\underline
z,\theta ):= \frac{q\,h-z_1}{\cos\theta}.
\]
For these values of $\theta_p^j$, the corresponding integral in
(\ref{eq:sing:01}) is approximated by
\begin{equation}
 h_j\,c(\theta_p^j) \sum_{q=0}^{N_j-1}  K^{ij}_{\rm sing,\delta}
\big(\underline u,\underline z+ \rho_q^{ h_j} (\underline
z,\theta_p^j)\,\underline e(\theta_p^j)\big)  \, |\rho_q^{h_j}
(\underline z,\theta_p^j)| \, \xi \left(\underline z+\rho_q^{ h_j}
(\underline z,\theta_p^j)\, \underline e(\theta_p^j)
\right),\label{eq:sing2}
\end{equation}
where $c(\theta)=1/|\cos\theta|$. For the remaining angles,
intersections are computed with the corresponding horizontal lines:
\begin{equation}
\label{eq:points02} \{ \underline x= \underline z+\rho\,
\underline e(\theta_p^j)\, :\, \rho \in \mathbb R\}\cap \{ \underline x= (\xi,
q\, h_j)\, :\, \xi \in \mathbb R\}, \qquad q=0,\ldots, N_j-1.
\end{equation}
The quadrature points are deliberately selected to
lie on lines with one of the coordinates equal to $q \, h_j$. As
  discussed in the introduction, such a selection was introduced
  in~\cite{377360} to enable fast interpolation of the function $\xi$
  on the radial quadrature points (that is, to produce a fast
  algorithm for evaluation of the operator $R_{N_j}$
  in~\eqref{eq:nyst}).

Using the angle-dependent weight
\begin{equation}\label{eq:2.3.a}
c(\theta):= \min \Big\{
\frac1{|\cos\theta|},\frac1{|\sin\theta|}\Big\},
\end{equation}
as well as the nodal points radii
\[
\rho^h_q(\underline z,\theta) := \left\{ \begin{array}{ll}
\displaystyle \frac{qh-z_1}{|\cos\theta|}, & \mbox{if
$|\cos\theta|\ge
\sqrt2/2$},\\[1.5ex] \displaystyle \frac{qh-z_2}{|\sin\theta|}, &
\mbox{otherwise}, \end{array}\right.
\]
expression (\ref{eq:sing2}) provides approximation of all the
integrals in (\ref{eq:sing:01}).

In sum, we define our discrete singular operator (which depends on
$h_j$ and $k_j$, or, equivalently, on $N_j$ and $\Theta_j$) by
\begin{equation}\label{eq:newQuad-1}
(\mathrm L^{ij}_{\delta,h}\xi)(\underline{u}):= \left\{
\begin{array}{l} \ds \frac{h_j k_j}{2} \!\!\sum_{p=0}^{\Theta_j-1}
 \sum_{q=0}^{N_j-1}\!\! c(\theta_p^j)\bigg(K^{ij}_{\rm sing,\delta}
\Big(\underline{u},\underline r^{ji}(\underline{u})+ \rho_q^{h_j}
(\underline r^{ji}(\underline{u}),\theta_p^j)\,\underline
e(\theta_p^j) \Big)\\[1.5ex]
\ds \hspace{1.4cm} \times |\rho_q^{h_j} (\underline
r^{ji}(\underline{u}),\theta_p^j)|\,
 \xi \Big(\underline r^{ji}(\underline{u})+\rho_q^{h_j} (
\underline r^{ji}(\underline{u}),\theta_p^j)\, \underline
e(\theta_p^j)
\Big)\bigg),\\[1.5ex]
\ds \hfill \mbox{if $\underline u  \in
\Omega^{ij}_{\delta/2}$},\\[1.5ex]
\ds 0, \hfill \mbox{otherwise},
\end{array}\right.
\end{equation}
where $\Omega^{ij}_{\delta/2}=(\r^i)^{-1}(S^{ij}_{\delta/2})=
(\r^i)^{-1}(S^{i}\cap S^{j}_{\delta/2}).$
\begin{figure}[htp]
\centerline{\includegraphics[width=0.90\textwidth]{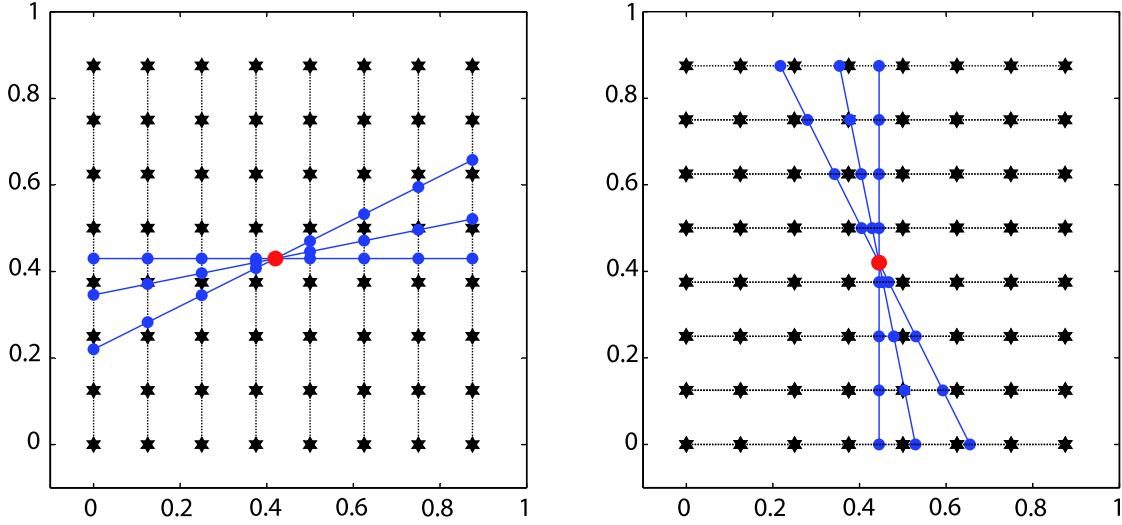}}
\caption{Points used in the integration of the singular
part.\label{fig:points03}}
\end{figure}
In what follows we use the scaling
\begin{equation}
\label{eq:alpha}
\Theta_j\approx N_j^{1+\alpha},\qquad \alpha>0,
\end{equation}
i.e., we assume that there exist $C,c>0$ such that for all $j$,
$cN_j^{1+\alpha}\le \Theta_j\le C N_j^{1+\alpha}$.  We point out that
$\rho_q^h(\underline{z},\,\cdot\,)$ is discontinuous at
$\{\pi/4,3\pi/4,5\pi/4, 7\pi/4\}$ but both side limits exist. It is
immaterial which value (the right or left limit or even the average of
both) is set as value to $\rho_q^{h}(\underline{z},\,\cdot\,)$ at
these points.

\subsection{The numerical method}

We are now in a position to lay down a fully discrete version of
equations~\eqref{eq:ExactSystem}. The unknowns are taken to be
approximate values
\begin{equation}\label{eq:values}
\varphi^j_{\underline{\ell}}, \qquad \underline\ell \in
\Omega^j_h:=\big\{\underline m\in\Z_{N_j}\::\:\r^j(\underline{x}
^j_{\underline m})\in \mathrm{int}(\mathrm{supp}\,\omega^j)\big\},
\qquad j=1,\ldots,J
\end{equation}
of the true function values $ \psi^j(\underline
x^j_{\underline\ell})$. In what follows we denote by
$\boldsymbol\varphi^j\in \mathbb C(\Omega^j_h)$ the array whose
$\underline\ell = (\ell_1,\ell_2)$ entry is
$\varphi^j_{\underline\ell}$ for all indices $\underline\ell \in
\Omega^j_h$. We consider the set of indices
\begin{equation}\label{znstar}
\mathbb Z_{N_j}^* := \big\{ \underline{m}:=(m_1,m_2)\in \Z^2\, :\,
-N_j/2\le m_1,m_2 < N_j/2\big\},
\end{equation}
(note that the cardinality of $\displaystyle \mathbb Z_{N_j}^*$ is
$\# \displaystyle\Z^*_{N_j}=\# \Z_{N_j}=N_j^2$), the space of trigonometric
polynomials
\[
\mathbb{T}_{N_j}:= \mathrm{span}\, \,\big\{ \exp( 2\pi{\rm i}\,
\underline m\cdot \underline u)\, \,:\,\, \underline m \in \mathbb
Z_{N_j}^*\big\},
\]
the interpolation operator $\mathrm R_{N_j}: \mathbb C(\Omega^j_h)
\to \mathbb T_{N_j}$ defined by the equations
\[
\mathrm R_{N_j}\boldsymbol\varphi (\underline x^j_{\underline\ell})=
\left\{ \begin{array}{ll} \varphi^j_{\underline\ell}, \quad&
\underline\ell \in \Omega^j_h,\\[1.5ex]
0, & \underline\ell \in \mathbb Z_{N_j}\setminus
\Omega^j_h,\end{array}\right.
\]
and the vectors $\boldsymbol\omega^j\in \mathbb C(\Omega^j_h)$ with
components $\omega^j_{\underline\ell}:=\omega^j(\mathbf r^j(
\underline x^j_{\underline\ell}))$. Using these notations, the
discrete Nystr\"{o}m equations for the system~\eqref{eq:ExactSystem} are
given by
\begin{eqnarray}
{\textstyle\frac12} \varphi^{i}_{\underline \ell} + \sum_{j=1}^J
 h_j^2 \sum_{\underline m\in \Omega^j_h} K^{ij}_{\rm
reg,\delta}(\underline x^i_{\underline \ell}, \underline
x^j_{\underline m}) \omega^j_{\underline m}\, \varphi^j_{\underline
m}  + \sum_{j=1}^J  \Big(\mathrm L^{ij}_{\delta,h} \mathrm
R_{N_j}(\boldsymbol\omega^j \odot \boldsymbol\varphi^j) \Big)
(\underline x_{\underline \ell}^i)&& \nonumber \\
&& \hspace{-96pt}=-U^{\rm inc}(\r^i(\underline
x^i_{\underline \ell})) \label{eq:nyst}
\end{eqnarray}
for $\underline\ell \in \mathbb C(\Omega^j_h)$ and $j=1,\ldots,N$,
where the binary  operator $\odot$ denotes
componentwise product of arrays: for example, $\boldsymbol\omega^j
\odot \boldsymbol\varphi^j$ is the array whose $\underline\ell$-th
entry is given by
$\omega^j_{\underline\ell}\varphi^j_{\underline\ell}$.  Once these
point values $\varphi^j_{\underline\ell}$ have been computed, the
Nystr\"{o}m method provides the reconstruction formula
\begin{equation}
\psi^{i}_h =-2\sum_{j=1}^J
 \bigg( h_j^2 \sum_{\underline m\in \Omega^j_h} K^{ij}_{\rm
reg,\delta}(\,\cdot\,, \underline x^j_{\underline m})
\omega^j_{\underline m}\, \varphi^j_{\underline m}  + \mathrm
L^{ij}_{\delta,h} \mathrm
R_{N_j}(\boldsymbol\omega^j\odot\boldsymbol\varphi^j) \bigg)-2U^{\rm
inc}\circ\r^i, \label{eq:nyst2}
\end{equation}
($i=1,\ldots,J$).  With these definitions, we clearly have
\[
\psi^j_h(\underline x^j_{\underline\ell})=\varphi^j_{\underline\ell}
\qquad \underline\ell \in \Omega^j_h.
\]
Finally, the $J$ parameter-space continuous functions $\psi^j_h:I_2\to
\mathbb{C}$ can be assembled into a single continuous approximate
solution $\psi_h(\r)$ defined on $S$:
\begin{equation}
\label{eq:nyst3} \psi_h(\r):=\sum_{j \in \mathcal I(\r)}
\omega^j(\r) \psi^j_h((\r^j)^{-1}(\r)), \qquad \mbox{where $\mathcal
I(\r):=\{ j\,:\, \r \in \supp \omega^j \}$.}
\end{equation}

The main convergence result of this paper can now be stated; its proof
is given in Sections~\ref{sec:3} through~\ref{sec:5}. The regularity
estimates given in the statement of the theorem are expressed in terms
of standard Sobolev norms on the surface $S$  (see for instance
\cite{adams:2003,McLean:2000,sauter2010boundary} or
any standard text on Sobolev spaces) .

\begin{remark}\label{remark_23}
The notation $\delta\approx h^\beta$ means that
\[
 c h^\beta\le \delta\le Ch^\beta,
\]
for some constants $c>0$ and $C>0$, independent of $h$ and $\beta$.
Hereafter, the parameter $\beta \in (0,1)$ will be taken to be
fixed. Dependence of constants on this parameter will not be shown
explicitly.
\end{remark}

\begin{theorem}\label{theo:MAIN} Let $\psi$ be the solution of
  \eqref{eq:intEqn}, and assume $\delta\approx h^\beta$ with
  $\beta\in(0,1)$. Then for sufficiently large values of $N$ the
  Nystr\"om equations \eqref{eq:nyst} admit a unique
  solution. Letting $\psi_h$ denote the reconstructed function on $S$
  as defined by~\eqref{eq:nyst2} and~\eqref{eq:nyst3}, we
  have the error estimate
\begin{equation}
\label{eq:MAIN01}
 \|\psi_h-\psi\|_{L^2(S)}\le C_{t}
\big(|\log h|^{\frac32}h^{t(1-\beta)}+h^{t-1}\big)\|\psi\|_{H^t(S)},
\end{equation}
for all $t\ge 2$. Finally, for all
$\varepsilon>0$ and $t\ge 2$
\begin{equation}
\label{eq:MAIN02}
 \|\psi_h-\psi\|_{L^\infty(S)}\le C_{t,\varepsilon}
\big(h^{t(1-\beta)-(1+\varepsilon)\beta}+h^{t-1}\big)\|\psi\|_{H^{t}(S)}.
\end{equation}
\end{theorem}
Theorem \ref{theo:MAIN} tells us that for a smooth surface $S$ and a
 $\mathcal
{C}^\infty$   right-hand side $U^\mathrm{inc}$ (from which it follows that
$\psi \in \mathcal C^\infty(S)$) the Nystr\"{o}m algorithm under
consideration converges super-algebraically fast.

\begin{remark}\label{remark:delta:clarified} As mentioned in 
  Section~\ref{geometry}, the parameter $\delta\approx h^\beta$ plays
  central roles in both the theory and the actual performance of the
  the algorithm under consideration. With regards to the latter we
  briefly mention here that use of the floating cut-off
  function~\eqref{eq:eta}, whose support is controlled by the
  parameter $\delta$, helps restrict the use of the costly polar
  integration scheme to a small region around the singular point (thus
  reducing the overall computing time required by the algorithm), and,
  further, it enables acceleration via a sparse, parallel-face
  FFT-based equivalent-source technique; see~\cite{377360} for
  details. (In particular we note that the value $\beta = 1/3$ is used
  in~\cite{377360,MR1875086} for optimal speed of the
  equivalent-source accelerated Nystr\"om method.) The parameter
  $\delta\approx h^\beta$ also has a significant impact on our
  theoretical treatment. Indeed, one of the most delicate points in
  our stability analysis concerns the convergence in norm of a
  discrete singular operator to a corresponding continuous singular
  operator. One of the terms in our estimate of the norm of the
  difference of these operators (equation \eqref{eq01:prop:main1-02}
  in Proposition \ref{prop:main1-02}) is bounded by $|\log h|^{3/2}
  h^{\beta/2}$. By taking $\beta>0$ we ensure that this terms also
  tends to zero, from which the desired convergence in norm results.
\end{remark} 

\section{Biperiodic framework}\label{sec:3}

\subsection{Continuous equations\label{continuous}}

The analysis of the method will be carried out by recasting the system
\eqref{eq:ExactSystem} as a system of periodic integral equations with
all unknown functions defined on the unit square $I_2$. To introduce
our periodic formulation we utilize a second family of cut-off
functions, $\widetilde{\omega}^j_\delta \in{\cal C}^\infty(I_2)$, that
depends on $\delta$ and satisfies the following assumptions:
\begin{subequations}
 \label{eq:omegatilde}
\begin{eqnarray}\label{eq:omegatildeProp1}
& & \supp \widetilde{\omega}^j_\delta \subset
\Omega^{jj}_{3\delta/2}=(\r^j)^{-1}(S^{j}_{3\delta/2}),\\
& & \widetilde{\omega}_\delta^j\ \equiv\ 1  \quad \mbox{ in }
\Omega^{jj}_{\delta}=(\r^j)^{-1}(S^j_\delta)\supset \Omega^j,
\label{eq:omegatildeProp3}
\\
& & \|\partial_{\underline\alpha}\widetilde{\omega}_\delta^j
\|_{\infty,I_2}\le
C_{\underline\alpha}\delta^{-|\underline\alpha|}\qquad
\forall\underline\alpha\ge \underline 0,\label{eq:omegatildeProp2}
\end{eqnarray}
\end{subequations}
where for a given non-negative bi-index
$\underline\alpha=(\alpha_1,\alpha_2)$,
\[
\partial_{\underline\alpha}\xi(\underline{u})=\partial^{\alpha_1}_{ {u
} _1 } \partial^ { \alpha_2 } _ { {u}_2}\xi(u_1,u_2)
\]
and $\|\cdot\|_{\infty,I_2}$ is the $L^\infty(I_2)$ norm.
We will use the characteristic function
\begin{equation}
\label{eq:defOmegaTilde0}
 \widetilde{\omega}_0^j(\underline{u}):=\left\{
\begin{array}{ll}
 1&\mbox{if }\underline{u}\in \Omega^j,\\
0&\mbox{otherwise},\\
\end{array}
\right.
\end{equation}
which can be viewed as the limit of $\widetilde{\omega}^j_{\delta}$ as
$\delta\to 0$.

Using these functions, we define the following periodic integral
operators
\begin{eqnarray} \mathrm A^{ij}_{\rm reg,\delta}\xi&:=&(\omega^i\circ
\r^i) \int_{I_2}  K^{ij}_{\rm
reg,\delta}(\,\cdot\,,\underline{u})\widetilde
{\omega}^j_\delta(\underline{u})\xi(\underline{u})
{\rm d}\underline{u},\label{eq:defKijreg}\\
\mathrm A^{ij}_{\rm sing,\delta}\xi&:=&
 (\omega^i\circ \r^i)  \int_{I_2} K^{ij}_{\rm
sing,\delta}(\,\cdot\,,\underline{u})\widetilde{\omega}^j_\delta
(\underline{u})\xi(\underline{u})
 {\rm d}\underline{u},\label{eq:defKijsing}\\
\mathrm A^{ij}_\delta \xi&:=& (\mathrm A^{ij}_{\rm reg,\delta}+
\mathrm A^{ij}_{\rm sing,\delta})\xi, \label{eq:defOpKij}
\end{eqnarray}
as well as the right-hand sides $ U^i:=-(\omega^iU^{\rm
  inc})\circ\r^i$, properly extended by zero to the full unit square
$I_2$. If $\underline u \in I_2 \setminus \Omega^{ij}_{2\delta}$, then
$\r^i(\underline u) \not\in S^{j}_{2\delta}$ and
\begin{eqnarray*}
\supp (K^{ij}_{\mathrm{sing},\delta}(\underline u,\,\cdot\,)
\,\widetilde\omega^j_\delta) &\subset& \supp
\eta_\delta(\r^i(\underline u), \r^j(\,\cdot\,) )\cap
\supp\widetilde\omega^j_\delta
\\
&\subset& \supp
\eta_\delta(\r^i(\underline u), \r^j(\,\cdot\,) ) \cap
\Omega^{jj}_{3\delta/2}=\emptyset
\end{eqnarray*}
by \eqref{eq:requirement_on_delta2} and \eqref{eq:omegatildeProp1},
which implies that
\begin{equation}
(\mathrm A^{ij}_{\rm sing,\delta}\xi)(\underline{u})=0\qquad \forall
\underline{u}\in I_2\setminus \Omega^{ij}_{2\delta}.
\label{eq:KijSingVanishes}
\end{equation}

Letting $\psi^j:=\psi\circ \r^j$ be as in
equation~\eqref{eq:ExactSystem} and since
$\widetilde{\omega}_\delta^j\:(\omega^j\circ\r^j)\equiv
\omega^j\circ\r^j$ (see \eqref{eq:omegatildeProp3} and note that
$\supp (\omega^j\circ \r^j)\subset \Omega^{jj}_\delta$), we see that
the functions
\begin{equation}\label{defphij}
  \phi^j:=(\omega^j\psi)\circ \r^j=(\omega^j\circ \r^j)\psi^j:I_2\to
  \mathbb{C}\qquad j=1,\ldots,J
\end{equation}
extended by zero outside $\Omega^j=\supp(\omega^j\circ\r^j)$
constitute a solution of the system
\begin{equation}
  \label{eq:continuous} {\textstyle\frac12} \phi^i+\sum_{j=1}^J
  \mathrm A^{ij}_{\rm reg,\delta}\phi^j+\sum_{j=1}^J \mathrm
  A^{ij}_{\rm sing,\delta}\phi^j=U^i, \qquad i=1,\ldots,J.
\end{equation}

Equation~\eqref{eq:continuous} amounts to a system of equations for
the vector $(\phi^j)\in \left(L^2(I_2)\right)^J$  for a
  given right-hand side  $(V^i)\in \left(L^2(I_2)\right)^J$. With this
understanding, Theorem \ref{th:3.6} shows that the
system~\eqref{eq:continuous} has a unique solution for any right-hand
side $(U^i)\in \left(L^2(I_2)\right)^J$. (It follows that for the
particular right-hand side \eqref{eq:continuous} the solution
is~\eqref{defphij}---which, clearly, is independent of $\delta$ in
spite of the $\delta$-dependence of the system of
equations~\eqref{eq:continuous}.)  Theorem \ref{prop:2per2cont} shows
that for the case $ U^i:=-(\omega^iU^{\rm inc})\circ\r^i$, the
solution of \eqref{eq:intEqn} can be reconstructed from the solution
of~\eqref{eq:continuous}.

\subsection{Analysis of the continuous system}

\begin{theorem}\label{prop:2per2cont}
Let $(\phi^1,\ldots,\phi^J)$ be a solution of \eqref{eq:continuous}
with $U^i:=-(\omega^i U^{\mathrm{inc}})\circ \r^i$. Then the
solution of \eqref{eq:intEqn}  can be expressed in the form
\[
\psi(\r):= \sum_{j \in \mathcal I(\r)}\phi^j((\r^j)^{-1}(\r)),
\]
see equation~\eqref{eq:nyst3}.
\end{theorem}
\begin{proof}
  Because of the particular form of the right-hand side as well as
  the presence of the factor $\omega^i \circ \r^i$ in the operator
  $\mathrm A^{ij}_\delta$ (see
  equations~\eqref{eq:defKijreg},~\eqref{eq:defKijsing} and
  ~\eqref{eq:continuous}), it follows that $\supp \phi^i\subseteq
  \supp (\omega^i \circ \r^i) =\Omega^i$. Therefore, by
  \eqref{eq:omegatildeProp3}, $\widetilde\omega^i_\delta \phi^i \equiv
  \phi^i$ for all $i$.

Consider now the functions $\psi^i: D^i \to \mathbb C$ given by
\[
\psi^i:= -2 \sum_{j=1}^J \int_{I_2} K^{ij} (\,\cdot\,, \underline v)
\phi^j(\underline v) \mathrm d \underline v-2\,U^{\mathrm{inc}}\circ
\r^i.
\]
These functions are constructed so that $\phi^i= (\omega^i \circ
\r^i)\, \psi^i$ for all $i$. Although the functions $\psi^i$ are
 infinitely differentiable  only up to the boundary of
$D^i$ (the domain of the chart $\r^i$), the function $\phi^i$ can be
smoothly extended by zero to $\mathbb R^2$. Similarly, we can consider
the functions $\Phi^i:S \to \mathbb C$ such that $\Phi^i(\r)=
\omega^i(\r) \psi^i ((\r^i)^{-1}(\r))=\phi^i((\r^i)^{-1}(\r))$ if $\r
\in S^i$ and are zero otherwise. These are $\mathcal
  {C}^\infty$ functions on the surface $S$ and $\psi= \sum_{j=1}^J
\Phi^j$.

Then, if $\r \in S$ and $i \in \mathcal I(\r)$, we can write
$\r=\r^i(\underline u^i)$ with $\underline u^i:=(\r^i)^{-1}(\r)$,
and
\begin{eqnarray*}
({\textstyle\frac12}\psi+A\psi)(\r)&=& {\textstyle\frac12}\sum_{i
\in \mathcal I(\r)} \phi^i(\underline u^i)+
\sum_{i=1}^J\omega^i(\r)(A\psi)(\r)\\
&=&\sum_{i\in \mathcal
I(\r)}\Big({\textstyle\frac12}\phi^i(\underline u^i)+ \omega^i(\r)
\int_S K(\r,\r') \Big(\sum_{j=1}^J \Phi^j(\r')\Big) \mathrm d
\mathbf r'\Big)\\
&=& \sum_{i\in \mathcal
I(\r)}\Big({\textstyle\frac12}\phi^i(\underline u^i) + \sum_{j=1}^J
\omega^i(\r) \int_{\supp \omega^j} K(\r,\r') \Phi^j(\r') \mathrm d
\r'\Big)\\
&=& \sum_{i\in \mathcal
I(\r)}\Big({\textstyle\frac12}\phi^i(\underline u^i) + \sum_{j=1}^J
\omega^i(\r) \int_{I_2} K(\r,\r^j(\underline v))\,a^j(\underline
v)\,\phi^j(\underline v)\,\mathrm d \underline v\Big)\\
&=& \sum_{i\in \mathcal
I(\r)}\Big({\textstyle\frac12}\phi^i(\underline u^i) + \sum_{j=1}^J
\omega^i(\r^i(\underline u^i)) \int_{I_2} K^{ij}(\underline
u^i,\underline v) \widetilde\omega^j_\delta(\underline v)
\phi^j(\underline v) \mathrm d \underline v\Big)\\
&=& \sum_{i\in \mathcal
I(\r)}\Big({\textstyle\frac12}\phi^i(\underline u^i)+ \sum_{j=1}^J
\mathrm A^{ij}_\delta \phi^j (\underline u^i)\Big)=\sum_{i\in
\mathcal I(\r)}U^i(\underline u^i) =-U^{\mathrm{inc}}(\r).
\end{eqnarray*}
Note that we have used the fact that $\widetilde\omega^j_\delta
\phi^j \equiv \phi^j$. This finishes the proof.
\end{proof}

The product space
\[
 {\cal H}^s:=\underbrace{\H^s\times \H^s\times\cdots\times \H^s}_{\mbox{$J$
times}},
\]
will be endowed with the product norm, also denoted by
$\|\cdot\|_s$. We can then consider the matrices of operators
\[
 {\mathcal A}^{\rm
reg}_\delta:=\left(\mathrm A^{ij}_{\rm
reg,\delta}\right)_{i,j=1}^J,\quad {\mathcal A}^{\rm
sing}_\delta:=\left(\mathrm A^{ij}_{\rm
sing,\delta}\right)_{i,j=1}^J,\qquad {\mathcal A}_\delta:= {\mathcal
A}^{\rm reg}_\delta+ {\mathcal A}_\delta^{\rm sing},
\]
as well as the identity operator  ${\cal I}:{\cal H}^s\to {\cal
H}^s$. With this notation,  \eqref{eq:continuous} can be written in
operator form as
\begin{equation}
 \label{eq02:continuous}
 {\mathcal B}_\delta\bm{\phi}:=({\textstyle\frac12}{\cal I}+  {\mathcal
A}_\delta) \boldsymbol\phi={\bf U}
\end{equation}
where $\bm{\phi} :=(\phi ^1,\phi ^2,\ldots,\phi ^J)$ and
\begin{equation}\label{defU}
\mathbf U:=\left(U^1,\ldots, U^J\right) \qquad
U^j=-(\omega^jU^{\mathrm{inc}})\circ \r^j.
\end{equation}
\begin{remark}\label{rem_32}
  Note that, while the operator ${\mathcal A}_\delta$ depends on
  $\delta$, equations~\eqref{eq:defKijreg}-\eqref{eq:defOpKij} show
  that, for elements $\boldsymbol\phi =(\phi^1,\ldots,\phi^J)$ such
  that $\mathrm{supp}\,\phi^i\subseteq \supp(\omega^i\circ \r^i)$,
  $\mathcal A_\delta\boldsymbol\phi$ is independent of $\delta$. The
  following theorem shows that the operator in
  equation~\eqref{eq02:continuous} is invertible, and, thus, in view
  of Theorem~\ref{prop:2per2cont}, for right-hand sides of the
  form~\eqref{defU}, the solution of equation~\eqref{eq02:continuous}
  is independent of $\delta$ as well.
\end{remark}

\begin{theorem}\label{th:3.6}
For all\, $0\le \delta\le\delta_0$, the operators $\mathcal
B_\delta=\frac12\mathcal I+\mathcal A_\delta:\mathcal H^0\to\mathcal
H^0$ are invertible. Moreover
\[
\|\mathcal B_\delta\|_{\mathcal H^0\to \mathcal H^0}+\|\mathcal
B_\delta^{-1}\|_{\mathcal H^0\to\mathcal H^0}\le C.
\]
\end{theorem}

\begin{proof}  We will first show that  ${\mathcal B}_\delta:{\cal H}^0\to{\cal
H}^0$ is invertible. Propositions
\ref{prop:propertiesK} and \ref{prop:3.4} and the compact injection
$H^{1}\subset H^{0}$ prove that ${\mathcal A}_\delta:{\cal H}^0\to
{\cal H}^0$ is compact (and uniformly bounded in $\delta$).
Therefore, ${\mathcal B}_\delta:{\cal H}^0\to{\cal H}^0$ is bounded
and Fredholm of index zero. Thus, it suffices to prove the
injectivity of the operator.

Let then $\bm{\phi}\in{\rm Ker}\,{\mathcal B}_\delta$, that is,
\[
{\textstyle\frac12}\phi^i +\sum_{j=1}^J \mathrm
A^{ij}_\delta\phi^j=0,\qquad i=1,\ldots,J.
\]
Arguing as in the proof of Theorem \ref{prop:2per2cont}, it is clear
that $\widetilde\omega_\delta^j\phi^j\equiv \phi^j$ for all $j$,
which means that $\mathrm{Ker}\,\mathcal B_\delta$ is independent of
$\delta$. Since $\mathrm A^{ij}_{\delta}:H^s\to H^{s+1}$ are
continuous for all $s$, then $\phi^i\in{\cal C}^\infty(\overline I_2)$.
Following the argumentation in the proof of Theorem
\ref{prop:2per2cont}, we define
\begin{equation}
\label{eq:th36a} \psi^i:=-2\sum_{j=1}^J \int_{I_2}
K^{ij}(\,\cdot\,,\underline v) \widetilde{\omega}_\delta^j
(\underline v)\phi^j(\underline v)\,{\rm d}\underline
v=-2\sum_{j=1}^J \int_{I_2} K^{ij}(\,\cdot\,,\underline v)
\phi^j(\underline v)\mathrm d \underline v,
\end{equation}
so that $\phi^i=(\omega^i\circ \r^i)\psi^i$. We now set
\[
\psi(\r):=\sum_{j \in\mathcal I(\r)}\phi^j((\r^j)^{-1}(\r))=
\sum_{j=1}^J\Phi^j(\r),
\]
where the functions $\Phi^j:S\to \mathbb C$ are defined by extending
$\omega^j \, \psi^j \circ (\r^j)^{-1}=\phi^j \circ (\r^j)^{-1}$ by
zero to $S \setminus S^j$. Proceeding as before, we prove that
$\frac12\psi+ A\psi=0$ and therefore $\psi=0$.

From \eqref{eq:th36a} we deduce that for $\underline u\in D^i$,
\begin{eqnarray*}
\psi^i(\underline u)&=& -2\sum_{j=1}^J \int_{I_2} K^{ij}(\underline
u,\underline v)\,\phi^j(\underline v)\,{ \rm d}\underline
v=-2\sum_{j=1}^J \int_{S^j} K(\r^i(\underline
u),\r')\Phi^j(\r')\,\mathrm d\r'\\
&=&-2 \int_{S} K(\r^i(\underline u),\r')\psi(\r')\,{ \rm d}\r'=0
\end{eqnarray*}
since $\psi$ itself is zero and the functions $\Phi^j$ are each
supported in $S^j$. Therefore $\phi^i=(\omega^i\circ \r^i)\psi^i=0$
and the injectivity of $\mathcal B_\delta$ is proven.

To prove the uniform boundedness of ${\mathcal B}_\delta^{-1}:{\cal
H}^0\to{\cal H}^0$ for $\delta\ge 0$ we proceed as in the proof of
Theorem 10.9 in \cite{kress98}. Recall that ${\mathcal A}_\delta:
{\cal H}^0\to {\cal H}^1$ is uniformly bounded in $\delta$ by
Proposition \ref{prop:propertiesK} with $s=0$. Since the injection
${\cal H}^1\subset{\cal H}^0$ is compact, the set $\{{\mathcal
A}_\delta\}_\delta$ turns out to be collectively compact. Moreover,
it is easy to verify that $\lim_{\delta \to 0} {\mathcal
A}_\delta\bm{\xi}=\mathcal A_0\boldsymbol\xi$ in $\mathcal H^0$ for
all $\boldsymbol\xi \in \mathcal H^0$. Applying that pointwise
convergence is uniform on compact sets and the collective
compactness of the set of operators $\{ \mathcal A_\delta\}$ it
follows (see \cite[Corollaries 10.5 and 10.8]{kress98}) that
\begin{equation}\label{collcomp}
\| {\mathcal B}_0^{-1}({\mathcal A}_0-{\mathcal
A}_{\delta}){\mathcal A}_{\delta}\|_{{\cal H}^0\to{\cal H}^0}\to 0.
\end{equation}
Consider now ${\cal C}_\delta:=2(\mathcal I-{\mathcal
B}_0^{-1}{\mathcal A}_\delta)$, which is uniformly bounded.
Straightforward computations show that
\[
 {\cal C}_{\delta}{\mathcal B}_\delta=
{\cal I}+2{\mathcal B}_0^{-1}({\mathcal B}_0{\mathcal
A}_\delta-{\mathcal B}_\delta{\mathcal A}_{\delta}) ={\cal
I}+2{\mathcal B}_0^{-1}({\mathcal A}_0-{\mathcal
A}_{\delta}){\mathcal A}_{\delta}.
\]
Therefore \eqref{collcomp} proves that for $\delta$ small enough
$\mathcal C_\delta \mathcal B_\delta$ is invertible with uniformly
bounded inverse and therefore, since $\mathcal C_\delta$ is uniformly
bounded, so is $\mathcal B_\delta^{-1}.$
\end{proof}

\subsection{Discrete system\label{discrete}}

In order to obtain a continuous system of equations from the fully
discrete system~\eqref{eq:nyst} we introduce an interpolation
operator $\mathrm Q_{N_j}:{\cal C}^0(\overline I_2)\to
\mathbb{T}_{N_j}$ given by
\begin{equation}\label{defQNj}
(\mathrm Q_{N_j}\xi)(\underline x^j_{\underline{\ell}})=\xi(\underline
x^j_{\underline{\ell}})\qquad \forall \underline{\ell}\in \Z_{N_j}.
\end{equation}
The discrete counterparts of the functions $\phi^j$ are
\begin{equation}
 \phi^j_{h}:=\mathrm Q_{N_j}((\omega^j\circ{\bf r}^j)\psi^j_h),
\label{eq:defphij}\\
\end{equation}
where $\psi^j_h$ is defined by \eqref{eq:nyst2} with
$\varphi^j_{\underline m}$ obtained as the solution of the system
\eqref{eq:nyst}. Note that
\begin{equation}\label{eq:relate}
\phi^j_h(\underline x^j_{\underline\ell})=\mathrm
Q_{N_j}((\omega^j\circ{\bf r}^j)\: \psi^j_h)(\underline
x_{\underline\ell}^j)= \mathrm R_{N_j}(\boldsymbol\omega^j\odot
\boldsymbol\varphi^j) (\underline x_{\underline\ell}^j)= \left\{
\begin{array}{ll}
\omega^j_{\underline \ell}\, \varphi^j_{\underline
\ell},\qquad&\mbox{if }\underline {
\ell }\in \Omega^j_h,\\[1ex]
0,&{\rm otherwise}.
\end{array}
\right.
\end{equation}

On the other hand, $\big(\widetilde{\omega}^j_{\delta}
\phi^j_{h}\big)(\underline{x}_{\underline{\ell}}^j)=\phi^j_{h}(\underline{x}_{
  \underline{\ell}}^j)$ for all $\underline\ell$, since
$\phi^j_{h}(\underline{x}_{\underline{\ell}}^j)=0,$ for all
$\underline{\ell}\not\in\Omega^j_h$ and inside $\Omega^j$ the
functions $\widetilde\omega^j$ do not have any influence by
\eqref{eq:omegatildeProp3}. Therefore \eqref{eq:nyst} can be
equivalently re-expressed as an equation for continuous biperiodic
functions $(\phi^1_h ,\ldots, \phi^J_h )$ such that for all $i\in
\{1,\ldots, J\}$
\begin{eqnarray}
 \nonumber
{\textstyle\frac12} \phi^i_h   + \sum_{j=1}^J \mathrm Q_{N_i} \bigg(
h_j^2 \sum_{\underline m \in \mathbb{Z}_{N_j}}
\left(\omega^i\circ\r^i\right) K^{ij}_{\rm reg,\delta} (\,\cdot\,,
\underline x_{\underline m}^j) \widetilde\omega^j_\delta(\underline
x^j_{\underline m})
\phi_{h}^j (\underline x_{\underline m}^j) \bigg)&&\\
+ \sum_{j=1}^J \mathrm Q_{N_i}\Big( (\omega^i\circ\r^i) \mathrm
L^{ij}_{\delta,h} \phi^j_{h} \Big) &=& \mathrm Q_{N_i} U^i.\qquad \
\ \label{nyst3}
\end{eqnarray}
Note that if the system \eqref{nyst3} is uniquely solvable, the
solution belongs necessarily to $\mathbb{T}_{N_1}\times
\cdots\times\mathbb{T}_{N_J}$. The solutions of \eqref{eq:nyst} and
\eqref{nyst3} are related by the formula \eqref{eq:relate}.

For the sake of our analysis we recast the system \eqref{nyst3} in
an equivalent operator form. To do this, we introduce the discrete
operators
\begin{equation}\label{newKijsing}
\mathrm A^{ij}_{{\rm sing},\delta,h}:= (\omega^i\circ \r^i) \mathrm
L^{ij}_{\delta,h},
\end{equation}
and, for $\xi\in\mathcal C^0(\overline I_2)$, we define
(cf.~\cite{CeDoSa:2002} and~\cite{DoRaSa:2008})
\begin{equation}
\label{eq:def:DN-1} \mathrm D_{N_j} \xi := h_j^2 \sum_{\underline
m\in \mathbb Z_{N_j}} \xi(\underline x^j_{\underline m})
\delta_{\underline x_{\underline m}^j},
\end{equation}
where $\delta_{\underline u}$ denotes here the Dirac delta
distribution at the point $\underline u \in \overline I_2$:
$\delta_{\underline
  u}\psi:=\psi(\underline{u})$. Further, we note that since the
operators $\mathrm A^{ij}_{\rm reg,\delta}$ have continuous kernels,
they may be applied to delta distributions:
\[
\mathrm A^{ij}_{\rm reg,\delta} \delta_{\underline u}=
\left(\omega^i\circ\r^i\right) K^{ij}_{\rm reg,\delta}
(\hspace{2pt}\cdot\hspace{2pt}, {\underline u})
 \widetilde{\omega}^j_\delta  ({\underline u}).
\]
In view of these definitions, we clearly have
\[
\mathrm A^{ij}_{\rm reg,\delta} \mathrm D_{N_j}\xi= h_j^2
\sum_{\underline m \in \mathbb{Z}_{N_j}}
\left(\omega^i\circ\r^i\right) K^{ij}_{\rm reg,\delta}
(\hspace{2pt}\cdot\hspace{2pt}, \underline x_{\underline m}^j)
\widetilde\omega^j_\delta(\underline x^j_{\underline m})\xi
(\underline x_{\underline m}^j),
\]
and~\eqref{nyst3} is equivalent to the equation
\begin{equation}
\label{nyst4} {\textstyle\frac12} \phi^i_{h} + \sum_{j=1}^J \mathrm
Q_{N_i} \Big(\mathrm A^{ij}_{\rm reg,\delta} \mathrm D_{N_j}
\phi^j_{h} + \mathrm A^{ij}_{{\rm sing},\delta,h}  \phi^j_{h}\Big) =
\mathrm Q_{N_i} U^i, \qquad i=1,\ldots, J,
\end{equation}
for the unknowns $(\phi^1_{h},\ldots,\phi^J_{h})\in
\mathbb{T}_{N_1}\times\cdots\times \mathbb{T}_{N_J}$.

In order to recast this system of operator equations as a system
defined in $L^2$ spaces, we insert the orthogonal projections
$\mathrm P_{N_j}:L^2(I_2)\to \mathbb{T}_{N_j}$ and write
\eqref{nyst3} in the form
\begin{equation}
{\textstyle\frac12} \phi^i_h + \sum_{j=1}^J \mathrm Q_{N_i}
\Big(\mathrm A^{ij}_{\rm reg,\delta} \mathrm D_{N_j} \mathrm P_{N_j}
\phi^j_h+ \mathrm A^{ij}_{{\rm sing},\delta,h} \mathrm
P_{N_j}\phi^j_h\Big) = \mathrm Q_{N_i} U^i, \qquad i=1,\ldots,
J.\label{nyst5}
\end{equation}
for the unknowns $(\phi^1_h ,\ldots,\phi^J_h )\in
L^2(I_2)\times\cdots\times L^2(I_2)$.
\begin{remark}\label{remark_31}
  Note that the projection operator $\mathrm P_{N_j}$, which maps
  $L^2$ to the space of trigonometric polynomials,
   makes  it possible to recast the fully discrete
  equation~\eqref{nyst3} (whose unknowns are approximate point values
  of the continuous solution of equation~\eqref{eq:ExactSystem}) via
  its version~\eqref{nyst4}, posed in the space of trigonometric
  polynomials, as an equation~\eqref{nyst5} in the complete space
  $L^2$.
\end{remark}

\subsection{Mapping properties of the operators introduced in
  Section~\ref{continuous}}

As a result of the work in Sections~\ref{continuous}
and~\ref{discrete}, our overall integral equation and its Nystr\"om
discretization have been re-expressed as equations
\eqref{eq:continuous} and~\eqref{nyst5}, respectively, which are
posed in terms of unknowns defined (and compactly supported) on the
open unit square $I_2$.  Such functions can naturally be extended to
period-1 biperiodic functions defined in $\R^2$, on which our
analysis is based. In this section, in particular, we study the
mapping properties of the operators introduced in
Section~\ref{continuous}, when viewed as operators defined on spaces
of periodic functions.

To do this we first consider the biperiodic Sobolev spaces, see
e.g.~\cite{SaVa:2002}, and we study their connections (that result
through our use of local charts and partitions of unity) with the
classical Sobolev spaces on the surface $S$. The Fourier coefficients
of any locally integrable biperiodic function $\xi:\mathbb R^2\to
\mathbb C$ are given by
\[
 \widehat{\xi}(\underline{m}):=\int_{I_2} \xi(\underline u)
 e_{\underline m}(\underline u)\mathrm d\underline u, \qquad
e_{\underline{m}}(\underline{u}):=\exp(2\pi{\rm i}\underline{m}
\cdot\underline{u}),   \quad \underline{m}\in\Z^2.
\]
For arbitrary $s\in \mathbb R$, the Sobolev norm
\begin{equation}
\label{eq:defSobolevNorm}
 \|\xi\|_{s}:=\Big(|\widehat{\xi}(\underline{0})|^2+\sum_{\underline{m}
\ne\underline{0}
}(|m_1|^2+|m_2|^2)^s|\widehat{\xi}(\underline{m})|^2\Big)^{1/2}
\end{equation}
is well defined for all $\xi$ in the space $\mathbb
T:=\mathrm{span}\{ e_{\underline m}\,:\, \underline m \in \mathbb
Z^2\}=\cup_N \mathbb T_N$ of trigonometric polynomials. The Sobolev
space $H^s$ is defined as the completion of $\mathbb T$ under the
norm $\|\,\cdot\,\|_s$. Note that $H^0$ is the space of biperiodic
extensions of functions in $L^2(I_2)$. By a simple density argument,
we can define the Fourier coefficients for any element of $H^s$ and
any $s<0$. It is also possible to define partial derivatives of any
order: for a given non-negative bi-index
$\underline\alpha=(\alpha_1,\alpha_2)$, using the notation
$|\underline\alpha|=\alpha_1+\alpha_2$, we see that
$\partial_{\underline\alpha}$ is a bounded operator from $H^s$ to
$H^{s-|\underline\alpha|}$ for all $s$.

The atlas introduced in Section~\ref{sec:2} for representation of the
surface $S$ together with the $H^s$ norm just defined gives rise to a
definition of the Sobolev norm
\begin{equation}
\label{Sobolev1}
 \|U\|_{H^s(S)}:=\bigg(\sum_{i=1}^J\|(\omega^i U)\circ
\r^i\|^2_{s}\bigg)^{1/2}
\end{equation}
for any $U\in \mathcal{C}^\infty(S)$ and any $s \in \mathbb R$.  (In
this formula it has been implicitly assumed that, even though $\r^i$
is only defined on $D^i\subset I_2$, the function $(\omega^i U)\circ
\r^i$ can be extended by zero to a function in $\mathcal{C}^\infty
(I_2)$---since the support of $\omega^i\circ \r^i$ is contained in
$D^i$---and the result can be extended as a  $\mathcal
{C}^\infty$   biperiodic function
to all of $\mathbb{R}^2$.)  The space $H^s(S)$ can then be then
defined, for instance, as the completion of ${\cal C}^\infty(S)$ in
the above norm: this definition is equivalent to the classical
definitions given in e.g. \cite{adams:2003}, \cite[\S 1.3.3]{Gri:1985},
 \cite[\S 2.4]{sauter2010boundary} 
and~\cite[Chapter 3]{McLean:2000}.

\begin{lemma}\label{lemma:sobolev} For all $s\in \mathbb R$ and $j\in
\{1,\ldots,J\}$
\begin{equation}
\label{Sobolev2}
 \|(\widetilde{\omega}_\delta^j\xi)\circ(\r^j)^{-1}\|_{H^s(S)} \le
C_s\delta^{-|s|}\|\xi\|_{s} \qquad \forall \xi \in H^s, \qquad
\forall \delta,
\end{equation}
where $C_s$ is a constant independent of $\delta$.
\end{lemma}

\begin{proof}
Because of \eqref{Sobolev1} we only need to prove that the maps
\begin{equation}\label{Tij}
\xi \longmapsto \mathrm T^{ji}_\delta\xi:=
\underbrace{(\omega^i\circ
\r^i)(\widetilde\omega^j_\delta\circ\underline
r^{ji})}_{=:\varrho_\delta^{ji}}\, (\xi\circ \underline r^{ji})
\end{equation}
map $H^s$ to $H^s$ for every $s$ and that their norms are bounded by a
multiple of $\delta^{-|s|}$. To do this we first note that
\begin{equation}\label{supprhoji}
\supp \varrho^{ji}_\delta \subset (\r^i)^{-1} (S^j\cap S^i),
\end{equation}
(which is the domain where $\underline r^{ji}$ is defined), and that,
in view of equation~\eqref{eq:omegatildeProp2}, we have
\begin{equation}\label{boundrhoji}
\|\partial_{\underline\alpha}\varrho^{ji}_\delta\|_{\infty,I_2}\le
C_{\underline\alpha}\delta^{-|\alpha|}.
\end{equation}

Since the operator $\mathrm T^{ji}_\delta$ is given by multiplication
by the $\mathcal C^\infty$ function $\varrho^{ji}_\delta$ preceded by
application of the smooth ($\delta$-independent) diffeomorphism
$\underline r^{ji}:(\r^i)^{-1} (S^j\cap S^i) \to (\r^j)^{-1}(S^j\cap
S^i)$, using the differential form of the Sobolev norms $H^s$ for
positive integers $s$ together with equations~\eqref{supprhoji}
and~\eqref{boundrhoji}, we see that
\[
\| \mathrm T^{ji}_\delta \xi\|_s \le C_s \delta^{-s}\|\xi\|_s \qquad
\forall \xi \in H^s \qquad s=0,1,2,\ldots
\]
Together with the interpolation properties~\cite{SaVa:2002} of the
Sobolev spaces $H^s$, this bound establishes the result for all $s\ge
0$. The result for $s<0$ follows from a duality argument.
\end{proof}

Here and in the sequel, $\|A\|_{X\to Y}$ denotes the operator norm of
the bounded operator $A:X\to Y$ between the Hilbert spaces $X$ and
$Y$.

\begin{proposition}\label{prop:propertiesK}
  For all $s\in\R$ and all indices $i,j$, the integral operators
  $\mathrm A^{ij}_\delta:H^s\to H^{s+1}$ given by
  equation~\eqref{eq:defOpKij} are continuous and we have
\[
 \|\mathrm A^{ij}_{\delta}\|_{H^s\to H^{s+1}}\le C_s\delta^{-|s|}.
\]
\end{proposition}

\begin{proof}
  We can write $\mathrm A^{ij}_\delta$ as the composition (from right
  to left) of the maps $\xi \mapsto
  (\widetilde\omega^j_\delta\xi)\circ(\r^j)^{-1}$ with $A$ (the
  combined integral operator given in \eqref{eq:defA}) and then
  $\varphi \mapsto (\omega^i \varphi)\circ
  \r^i$. Lemma~\ref{lemma:sobolev} provides a bound for the norm of
  first of these maps, whereas the norm of the last function as a map from
  $H^s(S)$ to the periodic Sobolev space $H^s$ is clearly bounded in
  view of the definition~\eqref{Sobolev1} of the
  surface Sobolev norms. The result therefore follows from the fact
  that $A$ is a bounded operator from $H^s(S)$ to $H^{s+1}(S)$ for all
  $s \in \mathbb R$---as it follows from standard results concerning
  pseudodifferential operators on smooth surfaces (see
  e.g.~\cite[Chapters 6--9]{HsWe:2008}).
\end{proof}

The next result studies the mapping properties of the regular and
singular part of $\mathrm A^{ij}_\delta$ in the frame of the periodic
Sobolev spaces and gives some estimates for the continuity constants
in terms of $\delta$. In some of the arguments, it is convenient to
use the space $\mathcal D:=\cap_{s\in \mathbb R}H^s$, consisting of
 $\mathcal
{C}^\infty$ biperiodic functions.

\begin{proposition}\label{prop:3.3}
For all $i,j=1,\ldots,J$, and $s,t\in\R$ the operators
\[
\mathrm A_{\rm reg,\delta}^{ij}:\H^s\to \H^t,\qquad  \mathrm A_{\rm
sing,\delta}^{ij}: \H^s\to \H^{s+1}
\]
are continuous. Moreover, for all $t\ge 0\ge s$ and all $r\in\R$ there
exist positive constants $C_{s,t}$ and $C_r$ such that
\begin{eqnarray}
\|\mathrm A_{\rm reg,\delta}^{ij}\|_{\H^{s}\to \H^t} \label{eq:p33a}
&\le&
C_{s,t}\delta^{s-t},\\
 \|\mathrm A_{\rm sing,\delta}^{ij}\|_{\H^{r}\to \H^{r+1}}&\le&
C_r\delta^{-\max\{-r,r+1,1\}}.\label{eq:p33b}
\end{eqnarray}
Finally
\begin{equation}
 \|\mathrm A_{\rm sing,\delta}^{ij}\|_{\H^{0}\to \H^{0}} \le
C_0. \label{eq:p33c}
\end{equation}
\end{proposition}

\begin{proof}Let
\[
F_{\rm reg,\delta}^{ij}(\underline{u},\underline{v}):=(\omega^i\circ
\r^i)(\underline{u}) K^{ij}_{\rm
reg,\delta}(\underline{u},\underline{v})
 \widetilde{\omega}^j_\delta (\underline{v})
\]
be the integral kernel of the operator $\mathrm A_{\rm
  reg,\delta}^{ij}$. This function is well defined on $D^i\times D^j$
and can be extended by zero to the rest of $I_2\times I_2$ thanks to
the cut-off functions that appear in its definition. In particular,
$F_{\rm reg,\delta}^{ij}$ admits a  $\mathcal
{C}^\infty$  biperiodic extension to
$\R^2\times \R^2$, and, thus $\mathrm A_{\rm reg,\delta}^{ij}$ is a
continuous operator from $H^s$ to $\H^t$ for all $s,t\in\R$ (see
\cite[Theorem 6.1.1]{SaVa:2002}). Since $\mathrm A_{\rm
  sing,\delta}^{ij}=\mathrm A^{ij}_\delta-\mathrm A_{\rm
  reg,\delta}^{ij}$, using Proposition \ref{prop:propertiesK}, the
claimed mapping properties of $\mathrm A_{\rm sing,\delta}^{ij}$
follow directly.

To establish~\eqref{eq:p33a} we consider the operators
\begin{equation}
 \label{eq:def:Lambdam}
D_{\underline{u}}^{m}\xi:= \partial_{u_1}^{2m}\xi+
\partial_{u_2}^{2m}\xi,\qquad
\Lambda_m
\xi:=D_{\underline{u}}^{m}\xi+\widehat{\xi}(\underline{0}).
\end{equation}
(Note that $\Lambda_m:H^s\to H^{s-2m}$ is a bounded isomorphism for
all $s$.) Letting
\[
G(\underline{u},\underline{v}):=D_{\underline{u}}^{m} F_{\rm
reg,\delta}^{ij}(\underline{u},\underline{v})+\int_{I_2}
 F_{\rm
reg,\delta}^{ij}(\underline{w},\underline{v})\,{\rm
d}{\underline{w}},
\]
we note that $G\in{\cal C}^\infty(\R^2\times \R^2)$, that $G$ is
periodic in each of its variables, and that, when restricted to
$I_2\times I_2$, $G$ has compact support. Now, for all $\eta\in
\mathcal D$ we have
\begin{eqnarray*}
 \Lambda_m \mathrm A_{\rm reg,\delta}^{ij}\Lambda_n \eta &=&\int_{I_2}
G(\cdot ,\underline{v})D_{\underline{v}}^{n}\eta(\underline{v})
\,{\rm d}\underline{v}+\widehat{\eta}(\underline{0})  \int_{I_2}
G(\cdot,\underline{v})\,{\rm
d}\underline{v} \\
&=&\int_{I_2}
D_{\underline{v}}^{n} G(\cdot ,\underline{v})\eta(\underline{v})
\,{\rm d}\underline{v}+\widehat{\eta}(\underline{0}) \int_{I_2}
G(\cdot,\underline{v})\,{\rm
d}\underline{v}.
\end{eqnarray*}
Further, bounds of the form
\begin{equation}
\label{eq:p33d} \sup_{\underline{u},\underline{v}\in I_2}
|D_{\underline{u}}^{m}D_{\underline{v}}^{n} F_{\rm
reg,\delta}^{ij}(\underline{u},\underline{v})|\le C_{m,n}
 \delta^{-2m-2n},
\end{equation}
follow from the definition the kernel functions $K^{ij}_{{\rm
reg},\delta}$ (see \eqref{eq:Kreg} and subsequent lines) and from
the assumptions \eqref{eq:omegatildeProp2} on
$\widetilde{\omega}^{j}_\delta$. It follows that, for $\xi\in{\cal
D}$ we have
\begin{eqnarray*}
\|\mathrm A^{ij}_{\rm reg,\delta}\xi\|_{2m}&\le&C_{2m}\|\Lambda_m
\mathrm A^{ij}_{\rm reg,\delta}\xi\|_{0}=C_{2m}\|\Lambda_m \mathrm
A^{ij}_{\rm
reg,\delta}\Lambda_n \Lambda_n^{-1}\xi\|_{0}\\
&\le& C_{2m,2n}\delta^{-2m-2n}\| \Lambda_n^{-1}\xi\|_{0}\le
C_{2m,2n}\delta^{-2m-2n}\|  \xi\|_{-2n},
\end{eqnarray*}
where we have used the integro-differential form of $\Lambda_m
\mathrm A^{ij}_{\rm reg,\delta}\Lambda_n $ and \eqref{eq:p33d}. This
inequality establishes \eqref{eq:p33a} in the particular case
$t=2m$, $s=-2n$; the result for general $t\geq 0\geq s$ then follows
by interpolation~\cite{SaVa:2002}.

To establish~\eqref{eq:p33b}, in turn, we first note
that~\eqref{eq:p33a} implies
\[
\| \mathrm A^{ij}_{{\rm reg},\delta}\|_{H^r\to H^{r+1}}\le \left\{
\begin{array}{ll}
  \| \mathrm A^{ij}_{{\rm reg},\delta}\|_{H^r\to H^{0}}\le C_{r,0} \delta^{r},&
  \mbox{if $r<-1$},\\
  \| \mathrm A^{ij}_{{\rm reg},\delta}\|_{H^r\to H^{r+1}}\le C_{r,r+1}
\delta^{-1},&
  \mbox{if
    $-1\le r<0$},\\
  \| \mathrm A^{ij}_{{\rm reg},\delta}\|_{H^{0}\to H^{r+1}}\le C_{0,r+1}
  \delta^{-(r+1)},& \mbox{if $r\ge 0$},
\end{array}
\right.
\]
or, equivalently
\[
\| \mathrm A^{ij}_{{\rm reg},\delta}\|_{H^r\to H^{r+1}} \le C'_{r}
\delta^{-\max\{-r,r+1,1\}}.
\]
Since $\mathrm A^{ij}_{{\rm sing},\delta}=\mathrm
A^{ij}_\delta-\mathrm A^{ij}_{{\rm reg},\delta}$, and in view of
Proposition \ref{prop:propertiesK}, we have
\begin{eqnarray*}
\|\mathrm A^{ij}_{{\rm sing},\delta}\|_{H^r\to H^{r+1}}&\le&
\|\mathrm A^{ij}_\delta\|_{H^r\to H^{r+1}}+\|\mathrm A^{ij}_{{\rm
reg},\delta}\|_{H^r\to H^{r+1}}\\
&\le& C_r \delta^{-|r|}
+C_r'\delta^{-\max\{-r,r+1,1\}} \le  C_r''\delta^{-\max\{-r,r+1,1\}}.
\end{eqnarray*}
which establishes~\eqref{eq:p33b}.

Inequality~\eqref{eq:p33c}, finally, follows from~\eqref{eq:p33a}
with $s=t=0$ and Proposition~\ref{prop:propertiesK}:
\[
\|\mathrm A^{ij}_\delta\|_{H^{0}\to H^{0}}\le \|\mathrm
A^{ij}_\delta\|_{H^{0}\to H^{1}} \le C,
\]
with $C$ independent of $\delta.$
\end{proof}

The limiting case $\delta=0$ is studied next.

\begin{proposition}\label{prop:3.4}
For all $i,j$, the operators $\mathrm A^{ij}_0:H^0\to H^1$ are
bounded.
\end{proposition}

\begin{proof}
  Clearly, $\mathrm A^{ij}_{0}\xi =\mathrm
  A^{ij}_{\delta_0}\:(\widetilde{\omega}_0^j\xi) $, where
  $\widetilde{\omega}_0^j $ coincides with the characteristic function
  of set $\Omega^j$ (see \eqref{eq:defOmegaTilde0}). Since the
  operator of multiplication by $\widetilde{\omega}_0^j $ defines a
  continuous operator in $H^0$, the result follows from Proposition
  \ref{prop:propertiesK}.
\end{proof}

\subsection{Convergence estimates in the biperiodic
framework}\label{sec:3.5}

We now state a convergence theorem (Theorem~\ref{theo:main2}) in the
biperiodic framework; as shown in Section~\ref{sec:5}, this theorem is
equivalent to our main convergence result, Theorem~\ref{theo:MAIN}.
The proofs of the two results presented in this section (Theorems
\ref{theo:main1} and \ref{theo:main2}), which require certain
analytical tools that are developed in the next section, are given in
section \ref{section:proofs}.

Let us introduce the operators
\begin{eqnarray*}
\mathcal Q_h \boldsymbol\xi&:=&\left(
\mathrm Q_{N_1}\xi_1,\ldots,\mathrm Q_{N_J}\xi_J\right),\\
\mathcal P_h \boldsymbol\xi&:=&\left(
\mathrm P_{N_1}\xi_1,\ldots,\mathrm P_{N_J}\xi_J\right),\\
\mathcal D_h \boldsymbol\xi&:=&\left( \mathrm
D_{N_1}\xi_1,\ldots,\mathrm D_{N_J}\xi_J\right),
\end{eqnarray*}
see Section~\ref{discrete}, and let
\[
 {\mathcal A}^{\rm sing}_{\delta, h}:=\big(\mathrm A_{{\rm
sing},\delta,h}^{ij}\big)_{i,j=1}^J.
 \]
 Clearly, equation \eqref{nyst5} (and, thus, in view of
 Remark~\ref{remark_31}, the main Nystr\"om system presented in
 equation~\eqref{eq:nyst}), can be re-expressed in the form
\begin{equation}
\label{eq:NystromEnd}
 {\mathcal B}_{\delta,h}\bm{\phi}_h:=\left({\textstyle\frac12}{\cal I}+\mathcal
Q_h {\mathcal A}_\delta^{\rm reg}\mathcal D_h\mathcal P_h+
 \mathcal Q_h {\mathcal A}^{\rm sing}_{\delta,h}
 \mathcal P_h\right)
 \bm{\phi}_h=\mathcal Q_h {\bf U}
\end{equation}
where $\bm{\phi}_h:=(\phi_{h}^1,\phi_{h}^2,\ldots,\phi_{h}^J)$.
(Note that, although not explicit in the notation, the unique solution
$\bm{\phi}_h$ of equation~\eqref{eq:NystromEnd} does depend on the
parameter $\delta$.)

\begin{theorem}\label{theo:main1} Let $\delta\approx h^{\beta}$ with
  $\beta\in(0,1)$. Then
\[
\lim_{h \to 0}\|{\mathcal B}_{\delta,h}-{\mathcal B}_\delta\|_{{\cal
H}^0\to {\cal H}^0}= 0.
\]
Thus, for all $h$ small enough, ${\mathcal B}_{\delta,h}:{\cal H}^0\to
{\cal H}^0$ is invertible, with $h$-uniformly bounded inverse.
\end{theorem}

Hence, for $h$ small enough the numerical scheme admits a unique
solution which depends continuously on the right-hand side.

\begin{theorem}\label{theo:main2} Let $\delta\approx h^{\beta}$ with
  $\beta\in(0,1)$ and let $\mathbf U$ be given by \eqref{defU}. Let
  $\boldsymbol\phi$ and $\boldsymbol\phi_h$ be the respective
  solutions of ${\mathcal B}_\delta\bm{\phi}={\bf U}$ and ${\mathcal
    B}_{\delta,h}\bm{\phi}_h=\mathcal Q_h{\bf U}$. Then, there exist
  constants $C_{s,t}>0$ for all $t\ge s\ge 0$ and $t>1$ such that
\[
 \|\bm{\phi}-\bm{\phi}_h\|_{s}\le
C_{s,t} {h^{t(1-\beta)-s}}\|\bm{\phi}\|_{t}.
\]
\end{theorem}

Note that $\boldsymbol\phi$ does not depend on $\delta$ (see
Remark~\ref{rem_32}), although $\boldsymbol\phi_{h}$ is a $\delta$
dependent quantity.

\section{Estimates for the singular kernel and associated singular
  operator}\label{sec:4}

 In this section we discuss the regularity properties
  of the singular kernel, and we present a non-standard estimate
  involving $L^\infty$ and $L^2$ norms over spaces of trigonometric
  polynomials (Proposition~\ref{Prop:BoundOnSingOperatorLinfty}) for
  the continuity constants of the associated singular operator. 

\begin{lemma}\label{K1polar}
The functions
\[
(\underline u,\rho,\theta) \longmapsto |\rho|K_1(\mathbf
r^i(\underline u),\mathbf r^j(\underline r^{ji}(\underline
u)+\rho\,\underline e(\theta))\,)
\]
are $\mathcal C^\infty$ in their domain of definition.
\end{lemma}

\begin{proof}
The kernel function $K_1$ can be decomposed (see \eqref{K1}) as
\[
K_1(\r,\r')=\frac{F_1(\r,\r')}{|\r-\r'|}+
\frac{F_2(\r,\r')}{|\r-\r'|^3}\,(\r-\r')\cdot\boldsymbol\nu(\r'),
\]
where $F_1, F_2\in \mathcal C^\infty(S\times S)$. Let
\begin{eqnarray*}
f_\ell(\underline u,\rho,\theta) &:=& F_\ell (\r^i(\underline
u),\r^j(\underline r^{ji}(\underline u)+\rho\,\underline
e(\theta)))\qquad \ell =1,2,\\
g_1(\underline u,\rho,\theta) &:=& \rho^{-2} |\r^i(\underline u)-
\r^j (\underline r^{ji}(\underline u)+\rho \underline
e(\theta))|^2\\
g_2(\underline u,\rho,\theta) &:=& \rho^{-2}\Big(\r^i(\underline u)-
\r^j (\underline r^{ji}(\underline u)+\rho \underline
e(\theta))\Big) \cdot \boldsymbol\nu(\underline r^{ji}(\underline
u)+\rho\,\underline e(\theta)).
\end{eqnarray*}
It is clear that $f_1$ and $f_2$ are $\mathcal C^\infty$. Noticing
that the functions
\begin{eqnarray}\label{G1}
\underline z &\longmapsto & |\r^i(\underline
u)-\r^j(\underline r^{ji}(\underline u)+\underline z)|^2\\
\label{G2} \underline z &\longmapsto & \Big(\r^i(\underline
u)-\r^j(\underline r^{ji}(\underline u)+\underline z)\Big)\cdot
\boldsymbol\nu(\underline r^{ji}(\underline u)+\underline z)
\end{eqnarray}
satisfy the hypotheses of Lemma \ref{K1polarLemma}, it follows that
$g_1$ and $g_2$ are $\mathcal C^\infty$. It is also clear that $g_1$
is positive for $\rho\neq 0$. The Hessian matrix at $\underline
z=\underline 0$ of the function defined in \eqref{G1} is the matrix
with elements
\[
2 \partial_{z_{n_1}} \r^j(\underline u)\cdot \partial_{z_{n_2}}
\r^j(\underline u) \qquad 1\le n_1, n_2\le 2.
\]
Clearly this matrix is positive semidefinite, and since its
determinant is
\[
4 |\partial_{z_1}\r^j(\underline u)\times
\partial_{z_2}\r^j(\underline u)|^2 = 4a^j(\underline u)^2,
\]
(see \eqref{Jac}), it is positive definite. Using the equality
\eqref{K1polarLemmaA}, it follows that $g_1(\underline u,0,\theta)\neq
0$. The mapping mentioned in the statement of the present lemma can be
formulated in terms of the expression
\[
g_1^{-1/2} f_1+ g_1^{-3/2} f_2\,g_2;
\]
the previous arguments show this mapping is  infinitely
  differentiable, and the lemma thus follows .
\end{proof}

\begin{lemma}\label{K1polarLemma}
Let $F$ be a $\mathcal C^\infty$ function in a neighborhood of the
origin in $\mathbb R^2$. If $F(\underline 0)=0$ and $\nabla
F(\underline 0)=\underline 0$, then the function
\[
f(\rho,\theta):=\rho^{-2}F(\rho\,\underline e(\theta))
\]
is $\mathcal C^\infty$. Moreover,
\begin{equation}\label{K1polarLemmaA}
f(0,\theta)={\textstyle\frac12}\, \underline e(\theta)\cdot \mathrm
HF(\underline 0)\underline e(\theta),
\end{equation}
where $\mathrm HF(\underline 0)$ is the Hessian matrix of $F$ at the
origin.
\end{lemma}

\begin{proof}
  The result follows from an application of the Taylor formula
\[
h(\rho,\theta)= h(0,\theta) + \rho \partial_\rho h(0,\theta) + \rho^2
\int_0^1 \partial_\rho^2 h(\rho\,u,\theta)\,(1-u)\mathrm d u.
\]
to the function $h(\rho,\theta):= F(\rho\,\underline e(\theta))$.
\end{proof}

The forthcoming analysis relies heavily on use of families $\{
\chi_\delta\,:\,0<\delta\le \delta_0\}$ of functions for which there
exist constants $c_0>0$ and $C_{\underline\alpha}$ for
$\underline\alpha\ge \underline 0$ such that for all $\delta$
($0<\delta\le \delta_0$)
\begin{subequations}
\label{eq:4.4}
\begin{eqnarray}
& & \chi_\delta \in \mathcal C^\infty(\mathbb
R^2),  \\
\label{eq:4.4.b} & &
\chi_\delta(\rho,\theta)=\chi_\delta(\rho,\theta+2\pi)
\qquad \forall (\rho,\theta) \in \mathbb R^2,\\
& & \supp \chi_\delta \subset(-c_0\,\delta,c_0\,\delta)\times
\mathbb R,\label{eq:4.4.c}\\
& & \label{eq:4.4.d} \|\partial_{\underline{\alpha}} \chi_\delta
\|_{L^\infty(\mathbb R^2)}\le C_{\underline{\alpha}}
\delta^{-|\underline{\alpha}|}\ \quad \forall\underline\alpha\ge
\underline 0.
\end{eqnarray}
\end{subequations}

\begin{proposition}\label{lemma:Kijsing}
For all $\underline{u}\in {\Omega^{ij}_{2\delta_0}}={({\bf
\r}^i)^{-1}(S^i\cap S_{2\delta_0}^j)}$ and $\delta\le \delta_0$, let
us define:
\begin{equation}
\chi_{\underline u,\delta}(\rho,\theta):={\textstyle\frac12}
\omega^i(\r^i(\underline u))|\rho| \, K^{ij}_{\rm
sing,\delta}(\underline{u},\underline r^{ji}(\underline{u})+\rho
\underline{e}(\theta))\: \widetilde{\omega}_\delta^j (\underline
r^{ji}(\underline{u})+\rho \underline{e}(\theta))\label{eq:cor:5.2}
\end{equation}
Then the sequence $\{\chi_{\underline u,\delta}\}$ satisfies
conditions \eqref{eq:4.4} with $c_0$ and $C_{\underline\alpha}$
independent of $\underline u$.
\end{proposition}

\begin{proof} We start by considering the simpler functions
\begin{equation}\label{eq:4.4.e}
\psi_{\underline{u} ,\delta}(\rho,\theta):=|\rho|\, K^{ij}_{\rm
sing,\delta}(\underline{u},\underline r^{ij}(\underline{u})+\rho
\underline{e}(\theta)).
\end{equation}
Since $\Omega^{ij}_{2\delta_0}\subset(\r^{i})^{-1}(S^i\cap S^j)$,
the mapping $\underline
r^{ji}:=(\r^j)^{-1}\circ\r^i:\Omega^{ij}_{2\delta_0} \subset I_2\to
I_2$ is well defined. Therefore, the functions
\begin{eqnarray*}
 \eta_{\underline{u},\delta}(\underline{z})&:=&\eta_\delta({\bf
r}^i(\underline{u}), {\bf r}^j(\underline r^{ji}(\underline{u})+
\underline{z})),\\
f_{\underline{u}}(\underline{z})&:=&|\underline z|\,
K_{1}(\r^i(\underline{u}),\r^j(\underline r^{ji}
(\underline{u})+\underline{z}))
\end{eqnarray*}
are well defined in $\Omega_{\underline{u}}:=\{\underline{z}\,:\,
\underline r^{ji}(\underline u)+\underline{z}\in D^j\}$ (that is
$\Omega_{\underline{u}} = D^j-\underline r^{ji}(\underline u)$).

As a simple consequence of \eqref{eq:requirement_on_delta}, if $\r
\in S^j_{2\delta_0}$, then $\overline{B(\r,\epsilon_1\delta)}
\subset S^j$. Applying \eqref{eq:suppetadelta} and the fact that
$\r^i(\underline u)\in S^{j}_{2\delta_0}$, it follows that
\begin{equation}\label{riju}
\supp \eta_\delta(\r^i (\underline u),\,\cdot\,) \subset
\overline{B(\r^i(\underline u),\epsilon_1\delta)} \subset S^j.
\end{equation}
Consequently
\[
\supp \eta_{\underline u,\delta} \subset
(\r^j)^{-1}( 
\overline{B(\r^i(\underline u)
,\epsilon_1\delta) 
)}
-\underline r^{ji}(\underline u) \subset (\r^j)^{-1}(S^j)-\underline
r^{ji}(\underline u)=\Omega_{\underline u}.
\]
Therefore $\supp \eta_{\underline u,\delta}$ is strictly contained
in $\Omega_{\underline u}$ and $\eta_{\underline u,\delta}$ can be
extended by zero to a function in $\mathcal C^\infty(\mathbb R^2)$.

On the other hand, for $\underline z \in \supp \eta_{\underline
  u,\delta}$ we have
\begin{eqnarray*}
|\underline z| &=& |\underline r^{ji}(\underline u)+\underline
z-\underline r^{ji}(\underline u)| \le C_j |\r^j(\underline
r^{ji}(\underline u)+\underline z)- \r^j (\underline
r^{ji}(\underline u))|\\
&=& C_j |\r^j(\underline r^{ji}(\underline u)+\underline
z)-\r^i(\underline u)|\le C_j \epsilon_1\delta
\end{eqnarray*}
(we have used \eqref{riju} in the final inequality). Therefore we can
fix $c>0$ independent of the particular chart number $j$ so that
$\supp \eta_{\underline u,\delta}\subset B(\underline 0,c\delta)$ for
all $\underline u \in \Omega^{ij}_{2\delta_0}$. Moreover, from
\eqref{eq:eta} we see that
\[
 \|\partial_{\underline{\alpha}}\eta_{\underline{u},\delta}\|_{L^\infty(\mathbb
R^2)}\le C_{\underline{\alpha}}\delta^{-|\underline{\alpha}|} \qquad
\forall \underline u\in \Omega^{ij}_{2\delta_0}.
\]
Thus, $\eta_{\underline{u} ,\delta}(\rho \underline{e}(\theta))$
satisfies the conditions \eqref{eq:4.4}.

Finally,  $\psi_{\underline{u} ,\delta}(\rho,\theta)=
(f_{\underline{u}}\eta_{{\underline{u}},\delta})(\rho\cos\theta,\rho
\sin\theta)$ and since, by Lemma \ref{K1polar}, $f_{\underline u}(\rho
\underline{e}(\theta))$ is $\mathcal
{C}^\infty$  on the support of
$\eta_{\underline{u} ,\delta}(\rho \underline{e}(\theta))$, it follows
that $\{\psi_{\underline u,\delta}\}$ satisfies \eqref{eq:4.4} with
$c_0$ and $C_{\underline\alpha}$ independent of $\underline u$. In
view of the inequalities \eqref{eq:omegatilde} the same is true of the
family $\{\chi_{\underline u,\delta}\}$, and the result follows.
\end{proof}

\begin{proposition}
  \label{Prop:BoundOnSingOperatorLinfty} There exists $C>0$ such that
  for all $N\in\mathbb{N}$ and all $\delta$, $0 < \delta\le \delta_0$,
\[
\|\mathrm A^{ij}_{\rm sing,\delta}\xi_N\|_{\infty,I_2}\le C|\log
N|^{3/2} \|\xi_N\|_0\qquad \forall \xi_N\in\mathbb{T}_N.
\]
\end{proposition}

\begin{proof} By \eqref{eq:KijSingVanishes}, it suffices to bound the
  values of $(\mathrm A^{ij}_{\mathrm{sing},\delta}\xi_N)(\underline
  u)$ for $\underline u\in \Omega^{ij}_{2\delta}$.  The polar
  coordinate form \eqref{eq:singularPart2} gives, for $\underline u
  \in \Omega^{ij}_{2\delta}$,
\begin{eqnarray*}
(\mathrm A^{ij}_{\mathrm{sing},\delta}\xi_N)(\underline u) &=&
\int_{-c_0\delta}^{c_0\delta} \int_0^{2\pi} \chi_{\underline
u,\delta}(\rho,\theta) \xi_N(\underline r^{ji}(\underline
u)+\rho\,\underline e(\theta))\,\mathrm d\rho\mathrm d \theta\\
&=& \sum_{\underline m \in \mathbb Z_N^* } \widehat\xi_N(\underline
m) e_{\underline m} (\underline r^{ji}(\underline u)) \,
\int_{-c_0\delta}^{c_0 \delta} \int_0^{2\pi} \chi_{\underline
u,\delta}(\rho,\theta) e_{\underline m}(\rho\,\underline
e(\theta))\mathrm d \rho \mathrm d \theta,
\end{eqnarray*}
where $\chi_{\underline u,\delta}$ is given by \eqref{eq:cor:5.2}.
Therefore, by the Cauchy-Schwarz inequality we have
\begin{equation}\label{KijxiN}
\| \mathrm A^{ij}_{\mathrm{sing},\delta}\xi_N\|_{\infty,I_2} \le \|
\xi_N\|_0 \Big( \sum_{\underline m\in \Z_N^*} D_{\underline
m,\delta}^2\Big)^{1/2},
\end{equation}
where
\[
D_{\underline m,\delta}:=\max_{\underline u \in
\Omega^{ij}_{2\delta}}
\int_0^{2\pi}\left|\int_{-c_0\delta}^{c_0\delta} \chi_{\underline
u,\delta}(\rho,\theta) \exp( {2\pi}{\rm i}\rho \underline m
\cdot\underline
e(\theta))\, \mathrm d\rho\right|\mathrm d\theta.
\]
As a direct consequence of Proposition \ref{lemma:Kijsing} it
follows that
\[
\left| \int_{-c_0\delta}^{c_0\delta} \chi_{\underline{u},
\delta}(\rho,\theta) \exp( {2\pi}{\rm i}\rho\,\underline m\cdot
\underline
e(\theta) )\,{\rm d}\rho\right|\le C \delta,
\]
and, integrating by parts, that
\begin{eqnarray*}
\left|\int_{-c_0\delta}^{c_0\delta} \chi_{\underline{u},
\delta}(\rho,\theta) \exp( {2\pi}{\rm i}\rho\,\underline m\cdot
\underline
e(\theta) )\,{\rm d}\rho \right|&=& \left|{\rm i}
\int_{-c_0\delta}^{c_0\delta}
\partial_\rho \chi_{\underline{u}, \delta}(\rho,\theta)\,
\frac{\exp({2\pi}{\rm
i}\rho\,\underline m \cdot\underline e(\theta) )}{\underline m
\cdot\underline e(\theta)}\mathrm d\rho\right|\\
&\le& C
\frac1{|\underline m\cdot \underline e(\theta)|},
\end{eqnarray*}
where the constant in both inequalities is independent of
$\underline u$. Therefore
\[
D_{\underline m,\delta} \le C
\int_0^{2\pi}\!\!\min\Big\{\delta,\frac{1}{|m_1\cos\theta+m_2\sin\theta|}\Big\}
\mathrm d \theta\le  C' \frac{\log N}{\sqrt{ m_1^2+m_2^2}},\quad
\forall \underline m \in \Z_N \setminus \{ \underline 0\},
\]
where we have applied Lemma \ref{lemma:auxP} for the last bound.
Inserting this bound in \eqref{KijxiN} and using the fact that
\[
\sum_{\underline m\in \mathbb Z_N^*\setminus\{ \underline 0\}}
\frac1{m_1^2+m_2^2} \le C \log N,
\]
(this can be proved by comparison with the integral of
$1/(x^2+y^2)$) the result follows readily.
\end{proof}

\begin{lemma}\label{lemma:auxP} For all $\underline m\in \Z^2\setminus\{
\underline
0\}$
\[
\int_0^{2\pi}\min \Big\{ 1,
\frac1{|m_1\cos\theta+m_2\sin\theta|}\Big\} \mathrm d\theta \le
\frac{2\pi\big(1+\log \sqrt{m_1^2+m_2^2}\big)}{\sqrt{m_1^2+m_2^2}}.
\]
\end{lemma}

\begin{proof} With some simple trigonometric arguments we prove that
\begin{eqnarray*}
\int_0^{2\pi} \!\!\!\!\min\{ 1,
|m_1\cos\theta+m_2\sin\theta|^{-1}\}\mathrm d \theta &=&\!\!
\int_0^{2\pi}\!\!\min \{ 1,|\mathrm{Re}( (m_1-{\rm i}m_2)
\exp({\rm i}\theta)) |^{-1}\}\mathrm d \theta\\
&=& 4 \int_0^{\pi/2}\!\!\!\! \min\{ 1,
(\sqrt{m_1^2+m_2^2}\,\sin\theta)^{-1}\} \mathrm d\theta.
\end{eqnarray*}
If we now shorten $c:=\sqrt{m_1^2+m_2^2}$, we can estimate
\begin{eqnarray*}
\int_0^{\pi/2}\min\{ 1,(c\,\sin\theta)^{-1}\} \mathrm d \theta &=&
c^{-1} \int_0^{\pi/2}\min \{ c,(\sin\theta)^{-1}\}\mathrm d \theta\\
 & \le &
c^{-1}\int_0^{\pi/2}\min\{ c,\frac\pi{2\theta}\} \mathrm
d\theta=\frac{\pi}{2c}(1+\log c),
\end{eqnarray*}
which finishes the proof.
\end{proof}

\section{Proofs of the main results}\label{sec:5}

\subsection{Inverse inequalities. Auxiliary approximation properties}

 In this subsection we collect some properties concerning
  the bivariate trigonometric polynomials $\xi_N\in\mathbb T_N$, which
  are needed for the analysis of the Nystr\"om method under
  consideration.

  From the definition of Sobolev norms \eqref{eq:defSobolevNorm} it is
  easy to establish the inverse inequalities
\begin{equation}\label{eq:invIneqSobolev}
\| \xi_N\|_t  \le  (\sqrt2\, h)^{s-t}\|\xi_N\|_s \qquad \forall\xi_N
\in \mathbb T_N, \qquad s \le t
\end{equation}
and
\begin{equation} \label{eq:invIneqLinfty}
\|\xi_N\|_{\infty,I_2}\le
\sum_{\underline{m}\in\mathbb{Z}_N^*}|\widehat{\xi}_N(\underline{m})|\le   N
\bigg(\sum_{\underline{m}\in\Z_N^*}|\widehat{\xi}_N(\underline{m})|^2\bigg)^{
1/2}=  h^{-1 }\|\xi_N\|_{0} \quad \forall\xi_N \in \mathbb T_N.
\end{equation}

The operator $\mathrm P_N:\cup_s H^s\to\mathbb{T}_N$ that cuts off
the tail of the Fourier series and at the same time gives the best
$H^s$ approximation in $\mathbb T_N$ for all $s$ is given by
\[
\mathrm P_N\xi:=\sum_{\underline{m}\in \Z_{N}^*}
\widehat{\xi}(\underline{m}) e_{\underline{m}}.
\]
Recalling that $h=1/N$, it is easy to check that
\begin{equation}
\label{eq01:lemma:aux01}
 \|\mathrm P_N\xi-\xi\|_{s}\le (2\,h)^{t-s}\|\xi\|_t \qquad \forall
 \xi \in H^t, \qquad s\le t.
\end{equation}
The interpolation operators introduced in \eqref{defQNj} satisfy
(cf. \cite[Theorem 8.5.3]{SaVa:2002})
\begin{equation}
\label{eq02:lemma:aux01}
 \|\mathrm Q_N\xi-\xi\|_{s}\le C_{s,t} h^{t-s}\|\xi\|_{t}\qquad \forall \xi\in
 H^t, \qquad 0 \le s \le t,\qquad  t> 1.
\end{equation}
The following lemma studies the uniform boundedness of $\mathrm Q_N$
as an operator from the space of continuous bivariate periodic
functions to $H^0$. (See \cite[Lemma 11.5]{Kr:1999} for a different
proof of this result in the univariate case.)

\begin{lemma}\label{lemma:aux01}
For all $N$
\begin{equation}
\label{eq03:lemma_aux01} \|\mathrm Q_N\xi\|_{0} =
\frac{1}{N}\bigg(\sum_{\underline{n}\in\mathbb{Z}_N}
\left|\xi(\underline{x}_{\underline{n}})\right|^2\bigg)^{1/2}\le\|\xi\|_{\infty,
I_2 } \qquad \forall \xi \in \mathcal C^0(\overline I_2).
\end{equation}

\end{lemma}

\begin{proof} Given $\xi_N \in\mathbb T_N$ we can write
\[
\xi_N\left( {\textstyle\frac{2\pi{\rm i}}{N}}\,\underline n\right)=
\sum_{\underline m \in \Z_N^*} \widehat\xi_N(\underline m)
\exp({\textstyle\frac{2\pi{\rm i}}{N}}\,(\underline n\cdot\underline
m)), \qquad \underline n \in \Z_N,
\]
and therefore, using the Parseval Theorem for the 
2-dimensional discrete finite Fourier transform, we obtain
\[
\|\xi_N\|_0^2=\sum_{\underline m\in
\Z_N^*}|\widehat\xi_N(\underline m)|^2 = \frac1{N^2}
\sum_{\underline n\in \Z_N}|\xi_N\left( {\textstyle\frac{2\pi{\rm
i}}{N}\,\underline n}\right)|^2.
\]
When $\xi_N=\mathrm Q_N\xi$, this gives the equality in
\eqref{eq03:lemma_aux01}, whereas the inequality is a simple
consequence of the fact that $\# \Z_N=N^2$.
\end{proof}

We finally give an approximation result for the operator $\mathrm
D_N$ defined in \eqref{eq:def:DN-1}:

\begin{lemma} \label{lemma:aux02}
There exists $C_{s,t}>0$ independent of $h$ such that
\[
  \|\mathrm D_N \mathrm P_N \xi-\xi\|_{s}\le C_{s,t} h^{t-s}\|\xi\|_t \qquad
  \forall \xi \in H^t, \qquad s\le t \le 0, \qquad s<-1.
\]
\end{lemma}

\begin{proof}
 The result is proven in \cite[Lemma 6]{CeDoSa:2002} (see also
\cite{DoRaSa:2008}) for univariate 1-periodic functions (with the
restriction $s<-1/2$ instead). The proofs can be easily adapted to
the current case.
\end{proof}

\subsection{Error analysis of the polar coordinate quadrature rules}
 
The results presented in this subsection, which lie at the heart of the main
convergence proof presented in this paper, provides estimates on the
quadrature error
\begin{equation}\label{eq:ErrorEstimate}
\mathcal E_{h,k,\gamma}(\psi):=\left|
\int_{-\infty}^\infty\int_0^{2\pi} \psi(\rho,\theta) \,\mathrm
d\rho\,\mathrm d\theta- Q_{h,k,\gamma}\psi\right|,
\end{equation}
(which are used in Section~\ref{section:proofs}) for the family of
polar-integration quadrature rules
\begin{equation}
  \label{eq:NewQuadRule} Q_{h,k,\gamma} \psi := k\sum_{p=0}^{\Theta-1}
  \bigg[c(\theta_p)h \sum_{q=-\infty}^{\infty} \psi (\gamma(\theta_p)+q
  c(\theta_p)h,\theta_p)\bigg]\approx
  \int_{-\infty}^{\infty}\int_{0}^{2\pi}
  \psi(\rho,\theta)\,{\rm d}\rho\,{\rm d}\theta
\end{equation}
with integrands {\em arising from certain trigonometric
  polynomials}. The parameters in this quadrature formula are the
positive integers
$N$ and $\Theta$ and the mesh-sizes $h=1/N$ and $k=2\pi/\Theta$; the
angular and radial quadrature nodes, in turn, are given by
$\theta_p=pk$ for $p= 0,\dots ,\Theta-1$ and $\gamma(\theta_p)+q
c(\theta_p)h$ for $q=-\infty,\dots, \infty$ with   
  $c(\theta):=\min\{1/|\cos \theta |$, $1/|
\sin\theta|\}$ respectively  (see~\eqref{eq:2.3.a}). Finally, the function 
 $\gamma$ is a 2$\pi-$periodic piecewise continuous function 
%  The integrand
% $\psi$, in turn, is a function
% with compact support in the variable $\rho$ which is given by
% $\psi=\chi_\delta\widetilde{\xi}_N$ where, for a given trigonometric
% polynomial $\xi_N\in\mathbb{T}_N$,
% $\widetilde{\xi}_N(\rho,\theta):=\xi_N(\rho\cos\theta,\rho\sin\theta)$.
% 
% 
Equation~\eqref{eq:NewQuadRule} embodies  trapezoidal quadrature rules in the
variables
$\theta$ and $\rho$, where the grid points used for integrating in
$\rho$ depend on $\theta$. The error estimate~\eqref{eq:ErrorEstimate}
is applied in Propositions \ref{prop:main1-015} and \ref{prop:5.5}  
on the $j$-th coordinate patch for
each $j=1,\dots, J$---with $N=N_j$, $\Theta=\Theta_j$ and $h=h_j$.

% \begin{remark}\label{remark:constants}
%   In view of Proposition~\ref{lemma:Kijsing}, throughout this appendix
%   the sequence $\{ \chi_\delta\}_{\delta >0}$ is assumed to satisfy
%   the conditions listed in~\eqref{eq:4.4}. Accordingly, the constants
%   $C$ and $C_r$ in Theorem~\ref{th:A.1} depend on the corresponding
%   constants in the bounds~\eqref{eq:4.4.d} and on the constant $c_0$
%   in equation~\eqref{eq:4.4.c}.  The dependence of the error bounds on
%   the parameter $\delta$, in turn, is made explicit in
%   equations~\eqref{eq:A1.1} and~\eqref{eq:A1.2}.
% \end{remark}

Our main results concerning this family of rules are  collected in the 
next theorem. For brevity our proof of this result is omitted here; all details
in these regards can be found in Appendix \ref{sec:A}.

\begin{theorem}\label{th:A.1}
Let $\Theta\approx N^{1+\alpha}$ with $\alpha>0$ and assume the sequence $\{
\chi_\delta\}_{\delta >0}$ satisfies the conditions listed in~\eqref{eq:4.4}.
% (see Remark~\ref{remark:constants}). 
Then, there exists $C>0$, independent of
$\delta$, $h$ and $\gamma$ such that, for all $\delta$ and $h$ satisfying
$c_0\delta < 1/2$ and $h<\delta$, we have
\begin{equation}
 \mathcal E_{h,k,\gamma}(\chi_\delta\,\widetilde\xi_N) \le   C \Big(
h\delta^{-1}|\log h|+\delta\Big)\|\xi_N\|_{0}\qquad \forall \xi_N
\in\mathbb T_N. \label{eq:A1.1}
\end{equation}
In addition, for each $r \ge 2$ there exists $C_r>0$ such that
\begin{equation}
 \mathcal E_{h,k,\gamma}(\chi_\delta\,\widetilde\xi_N    ) \le C_r
h^r\delta^{1-r} \|\xi_N\|_r \qquad \forall \xi_N\in \mathbb
T_N.\label{eq:A1.2}
\end{equation}

\end{theorem}

\subsection{Proofs of the results of Section \ref{sec:3.5}}

\label{section:proofs}

Recall that $N$ has been taken to be the largest of the discrete
parameters $N_j$ and therefore $h_j\le h$ for all $j=1,\ldots,J$.
Also, by \eqref{eq:quasiuniform}, we can bound $h_j^{-1}\le C
h^{-1}$ and $|\log h_j|\le C |\log h|$ whenever needed.

\begin{proposition}
 \label{prop:main1-01} For all  $t> 2$, there exists $C_t>0$ such
that
 \[
  \left\|{\mathcal A}^{\rm reg}_\delta-\mathcal Q_h{\mathcal A}^{\rm
reg}_{\delta}\mathcal D_h\mathcal P_h\right\|_{{\cal H}^0\to {\cal H}^0}\le C_t
\delta^{-t} h ^{t}.
 \]
\end{proposition}

\begin{proof} Let us first consider the decomposition
\begin{eqnarray*}
\mathrm A^{ij}_{\rm reg,\delta}-\mathrm Q_{N_i} \mathrm A^{ij}_{\rm
reg,\delta}\mathrm D_{N_j} \mathrm P_{N_j}&=& \mathrm A^{ij}_{\rm
reg,\delta}(\mathrm I-\mathrm D_{N_j} \mathrm P_{N_j})+(\mathrm
I-\mathrm Q_{N_i})\mathrm A^{ij}_{\rm reg,\delta}(\mathrm D_{N_j}
\mathrm P_{N_j}-\mathrm I)\\&&+ (\mathrm I-\mathrm Q_{N_i})\mathrm
A^{ij}_{\rm reg,\delta}.
\end{eqnarray*}
By Proposition  \ref{prop:3.3} and Lemma \ref{lemma:aux02}, for all
$t>2$,
\[
\|\mathrm A^{ij}_{\rm reg,\delta}(\mathrm I-\mathrm D_{N_j}\mathrm
P_{N_j})\xi\|_{0}\le C_t \delta^{-t}\|(\mathrm I-\mathrm
D_{N_j}\mathrm P_{N_j})\xi\|_{-t}\le C'_t h^t\delta^{-t}\|\xi\|_0.
\]
The second term is bounded using \eqref{eq02:lemma:aux01},
Proposition \ref{prop:3.3} and Lemma \ref{lemma:aux02}:
\begin{eqnarray*}
\|(\mathrm I-\mathrm Q_{N_i}) \mathrm A^{ij}_{\rm
reg,\delta}(\mathrm D_{N_j} \mathrm P_{N_j}-\mathrm I)
\xi\|_{0}&\le& C_t h^{t/2} \|\mathrm A^{ij}_{\rm reg,\delta}(\mathrm
D_{N_j}\mathrm
P_{N_j}-\mathrm I) \xi\|_{t/2}\\
&& \hspace{-24pt}\le C_t' h^{t/2} \delta^{-t}\|(\mathrm D_{N_j}\mathrm
P_{N_j}-\mathrm I) \xi\|_{-t/2}\le C_t''\delta^{-t} h^t\|\xi\|_{0}.
\end{eqnarray*}
Finally, by \eqref{eq02:lemma:aux01}, Proposition \ref{prop:3.3} and
Lemma \ref{lemma:aux02},
\begin{eqnarray*}
 \|(\mathrm I-\mathrm Q_{N_i})\mathrm A^{ij}_{\rm reg,\delta}\xi \|_{0}&\le& C_t
h^t\|\mathrm A^{ij}_{\rm
reg,\delta}\xi \|_{t}\le C'_t h^t\delta^{-t}\|\xi \|_{0},
\end{eqnarray*}
and the proof is finished.
\end{proof}

\begin{proposition}
\label{prop:main1-015} There exists $C>0$ such that for all
$\xi_{N_j}\in \mathbb{T}_{N_j}$ and $\delta>h$
\begin{eqnarray}
\sup_{\underline u\in \Omega^{ij}_{\delta/2}}\Big| (\mathrm
A^{ij}_{\rm sing,\delta} -\mathrm A^{ij}_{{\rm
sing},\delta,h})\xi_{N_j}(\underline u) \Big|&\le& C\Big(\delta^{-1}
h_j|\log h_j|+\delta
\Big)\|\xi_{N_j}\|_{0},\label{eq01a:prop:main1-015}\\
\sup_{\underline u\in \Omega^{ij}_{2\delta}\setminus
\Omega^{ij}_{\delta/2}}\Big| (\mathrm A^{ij}_{\rm sing,\delta}
-\mathrm A^{ij}_{{\rm sing},\delta,h})\xi_{N_j}(\underline u)
\Big|&\le&
C|\log h_j|^{3/2}\|\xi_{N_j}\|_{0},\label{eq01b:prop:main1-015}\\
\sup_{\underline u\in I_2\setminus \Omega^{ij}_{2\delta}}\Big|(
\mathrm A^{ij}_{\rm sing,\delta}-\mathrm A^{ij}_{{\rm
sing},\delta,h})\xi_{N_j} (\underline
u)\Big|&=&0.\label{eq01c:prop:main1-015}
\end{eqnarray}
\end{proposition}

\begin{proof} Recalling the definition of the discrete operator
$\mathrm A^{ij}_{\mathrm{sing},\delta,h}=(\omega^i\circ\r^i) \mathrm
L^{ij}_{\delta,h}$ (the operator $\mathrm L^{ij}_{\delta,h}$ is
defined in \eqref{eq:newQuad-1}), it is clear that
\begin{equation}\label{eq:Kijsingh=0}
(\mathrm A^{ij}_{{\rm sing},\delta,h}\xi_{N_j})(\underline u)=0
\qquad \underline u\in I_2\setminus\Omega^{ij}_{\delta/2}.
\end{equation}
Using also \eqref{eq:KijSingVanishes}, it is obvious that both
operators in \eqref{eq01c:prop:main1-015} vanish for $\underline u
\in I_2\setminus\Omega^{ij}_{2\delta}$. Also
\eqref{eq01b:prop:main1-015} is a simple consequence of
\eqref{eq:Kijsingh=0} and Proposition
\ref{Prop:BoundOnSingOperatorLinfty}.

For $\underline u\in \Omega^{ij}_{\delta/2}$, recalling that
$\widetilde\xi_{N_j}(\rho,\theta)=\xi_{N_j}(\underline
r^{ji}(\underline u)+ \rho \underline e(\theta))$,
\[
(\mathrm A^{ij}_{\rm sing,\delta} -\mathrm A^{ij}_{{\rm
sing},\delta,h}) \xi_{N_j} (\underline u)=
\int_{-\infty}^{\infty}\int_0^{2\pi}\! (\chi_{\underline u,\delta}
\widetilde{\xi}_{N_j})(\rho,\theta){\rm d}\rho\,{\rm d}\theta
-Q_{h,k,\gamma}(\chi_{\underline u,\delta}\widetilde{\xi}_{N_j}),
\]
where $\chi_{\underline u,\delta}$ is defined by \eqref{eq:cor:5.2}
and the quadrature rule $Q_{h,k,\gamma}$ is given in
\eqref{eq:NewQuadRule}. By Proposition \ref{lemma:Kijsing} we can
apply Theorem \ref{th:A.1}, which estimates the above quadrature
error in terms of the constants that appear in \eqref{eq:4.4}. Since
these constants can be taken to be independent of $\underline u \in
\Omega^{ij}_{\delta/2}$ (this is part of the assertion of
Proposition \ref{lemma:Kijsing}), then \eqref{eq01a:prop:main1-015}
follows readily.
\end{proof}

In the following sequence of results we will use the geometric
cut-off operator
\begin{equation}\label{eq:def:G}
\mathcal G\boldsymbol\xi= (\mathrm G_1\xi_1,\ldots,\mathrm G_J
\xi_J):= \big( (1-\widetilde\omega_0^1)
\xi_1,\ldots,(1-\widetilde\omega_0^J)\xi_J\big),
\end{equation}
where, as a reminder, $\widetilde\omega_0^j$ is the characteristic
function of the domain $\Omega^j$.

\begin{proposition}\label{prop:5.5}
For all $t\ge 2$, there exists $C_t>0$ such that for all $
\xi_{N_j}\in \mathbb T_{N_j}$ and $\delta> h$
\begin{equation}
\label{eq02:prop:main1-015}
  \| (\mathrm A^{ij}_{\rm sing,\delta}
-\mathrm A^{ij}_{{\rm sing},\delta,h})\xi_{N_j}\|_{\infty,I_2}\le
C_t\delta^{1-t}h_j^{t}\|\xi_{N_j}\|_{t}+C \delta \|\mathrm G_j
\xi_{N_j}\|_{\infty,I_2}.
\end{equation}
\end{proposition}

\begin{proof} Following the argument of the preceding proof, but
using \eqref{eq:A1.2} of Theorem \ref{th:A.1}, it follows that for
any integer $t\ge 2$,
\begin{equation}
 \sup_{\underline u\in \Omega^{ij}_{\delta/2}}\Big| (\mathrm A^{ij}_{\rm
sing,\delta} -\mathrm A^{ij}_{{\rm
sing},\delta,h})\xi_{N_j}(\underline u) \Big|\le C_t
\delta^{1-t}h^t\|\xi_{N_j}\|_t.\label{eq08:prop:main1-15}
\end{equation}
From the definition \eqref{eq:2.8}, it follows that if $\r\not\in
S^j_{\delta/2}$, then $\overline{B(\r,\epsilon_1\delta)}\cap\supp
\omega^j =\emptyset$. Therefore, if $\underline
u\not\in\Omega^{ij}_{\delta/2}$,
\begin{eqnarray*}
\supp K^{ij}_{\mathrm{sing},\delta}(\underline u,\,\cdot\,)\cap
\Omega^j &\subset & \supp \eta_{\delta}(\r^i(\underline
u),\r^j(\,\cdot\,)) \cap\Omega^j\\
&\subset& (\r^j)^{-1}\big(\overline{B(\r^i(\underline
u),\epsilon_1\delta)}\cap \supp \omega^j)=\emptyset,
\end{eqnarray*}
where we have used \eqref{eq:suppetadelta}. This means that if
$\underline u\in
\Omega^{ij}_{2\delta}\setminus\Omega^{ij}_{\delta/2}$, only  the
value of $\xi_{N_j}$ on $I_2\setminus\Omega^j$ (where
$\xi_{N_j}=\mathrm G_j \xi_{N_j}$) is relevant. We then change to
polar coordinates as in the proof of Proposition
\ref{Prop:BoundOnSingOperatorLinfty} to obtain:
\begin{eqnarray}\nonumber
|(\mathrm A^{ij}_{\mathrm{sing},\delta}-\mathrm
A^{ij}_{\mathrm{sing},\delta,h})\xi_{N_j}(\underline u)| &=&
|(\mathrm A^{ij}_{\mathrm{sing},\delta}\xi_{N_j})(\underline u)|
\\
\nonumber & \le & \|\xi_{N_j}\|_{\infty,I_2\setminus\Omega^j}
\int_{-c_0\delta}^{c_0\delta} \int_0^{2\pi}|\chi_{\underline
u,\delta}(\rho,\theta)|\,\mathrm d\rho\,\mathrm d\theta \\
&\le& C \delta \|\mathrm G_j\xi_{N_j}\|_{\infty,I_2}.
\label{eq09:prop:main1-15}
\end{eqnarray}
(In the last inequality we have applied Proposition
\ref{lemma:Kijsing}, according to which the constants $c_0$ and $C$ do
not depend on $\underline u$.) Equation~\eqref{eq02:prop:main1-015}
now follows from equations~\eqref{eq01c:prop:main1-015},
\eqref{eq08:prop:main1-15} and~\eqref{eq09:prop:main1-15}.
\end{proof}

\begin{proposition}
 \label{prop:main1-02} For all $\varepsilon>0$ there exists $C_\varepsilon>0 $
independent of $h$
and
$\delta>h$
such that
\begin{equation}
\label{eq01:prop:main1-02}
 \left\|{\mathcal A}^{\rm sing}_\delta -\mathcal Q_h{\mathcal A}^{{\rm
sing}}_{\delta,h}\mathcal P_h\right\|_{{\cal H}^0\to {\cal H}^0}\le
C_\varepsilon \Big(\delta^{-1-\varepsilon}h+ \delta^{-1}h\,|\log h|
+\delta^{1/2}|\log h|^{3/2} \Big).
\end{equation}
Furthermore, for any integer $r\ge 2$,
\begin{equation}
\label{eq02:prop:main1-02}
 \big\|\big({\mathcal A}^{\rm sing}_\delta -\mathcal Q_h{\mathcal A}^{{\rm
sing}}_{\delta,h}\mathcal P_h\big)\bm{\xi}\big\|_{0}\le C_r
\big(\delta^{-r} h^r\|\bm{\xi}\|_{r} +\delta\|\mathcal
P_h\boldsymbol\xi-\boldsymbol\xi\|_{\infty,I_2}+\delta\|\mathcal
G\boldsymbol\xi\|_{\infty,I_2}\Big) \quad \forall\boldsymbol\xi\in\mathcal H^0.
\end{equation}
\end{proposition}

\begin{proof}
We start with the decomposition
\begin{eqnarray}
 \mathrm A_{\rm sing,\delta}^{ij}-\mathrm Q_{N_i} \mathrm A^{ij}_{{\rm
sing},\delta,h}\mathrm P_{N_j}&=&\mathrm A^{ij}_{\rm
sing,\delta}(\mathrm I-\mathrm P_{N_j})+(\mathrm I-\mathrm
Q_{N_i})\mathrm A^{ij}_{\rm
sing,\delta}\mathrm P_{N_j}\nonumber\\
&&+\mathrm Q_{N_i}(\mathrm A^{ij}_{\rm sing,\delta} -\mathrm
A^{ij}_{{\rm sing},\delta,h})\mathrm
P_{N_j}.\label{eq03:prop:main1-02}
\end{eqnarray}
For all $t\ge 1$,  Proposition \ref{prop:3.3} with $r=-1$ and
(\ref{eq01:lemma:aux01}) imply
\begin{equation}\label{eq04:prop:main1-02}
\|\mathrm A^{ij}_{\rm sing,\delta}(\mathrm I-\mathrm
P_{N_j})\xi\|_{0}\le  C \delta^{-1} \|\xi-\mathrm
P_{N_j}\xi\|_{-1}\le C_t \delta^{-1}h^{t}\|\xi\|_{t-1}.
\end{equation}
For $t>1$, by \eqref{eq02:lemma:aux01}, Proposition \ref{prop:3.3}
with $r=t-1$ and the inverse inequality \eqref{eq:invIneqSobolev},
we can bound
\begin{eqnarray}
 \|(\mathrm I-\mathrm Q_{N_i})\mathrm A^{ij}_{\rm sing,\delta}\mathrm
P_{N_j}\xi\|_{0}&\le& C_t h^t\|\mathrm A^{ij}_{\rm sing,\delta
}\mathrm P_{N_j}\xi\|_{t}\le C_t' \delta^{-t}\,h^t
 \|\mathrm P_{N_j}\xi\|_{t-1}\label{eq05a:prop:main1-02}\\
&\le& C'_t   \delta^{-t}h\|\mathrm P_{N_j}\xi\|_{0}\le C'_t
\delta^{-t}h\| \xi\|_0.\label{eq05b:prop:main1-02}
\end{eqnarray}
For fixed $\xi$ we introduce the auxiliary function $ g:=(\mathrm
A^{ij}_{\rm sing,\delta} -\mathrm A^{ij}_{{\rm
sing},\delta,h})\mathrm P_{N_j}\xi$, so that by Lemma
\ref{lemma:aux01},
\begin{equation}
\label{eq05:prop:main1-02}
 \|\mathrm Q_{N_i}(\mathrm A^{ij}_{\rm sing,\delta}
-\mathrm A^{ij}_{{\rm sing},\delta,h})\mathrm P_{N_j}\xi\|^2_{0} =
\|\mathrm Q_{N_i}g\|_0^2= h_i^{2}\sum_{\underline
m\in\Z_{N_i}}|g(\underline{x} ^ i_{\underline{m}})|^2.
\end{equation}
We can then split $\Z_{N_i}$ as follows:
\begin{eqnarray*}
A_{N_i}&:=&\big\{\underline m\in\mathbb{Z}_{N_i}\, :\,
\underline{x}_{\underline m}^i\not\in \Omega^{ij}_{2\delta}
\big\},\\
B_{N_i}&:=&\big\{\underline m\in\mathbb{Z}_{N_i}\, :\,
\underline{x}_{\underline m}^i\in  \Omega^{ij}_{2\delta} \setminus
\Omega^{ij}_{\delta/2} \big\},\\
C_{N_i}&:=&\big\{\underline m\in\mathbb{Z}_{N_i}\, :\,
\underline{x}_{\underline m}^i\in \Omega^{ij}_{\delta/2} \big\}.
\end{eqnarray*}
A key point is the fact that $\# B_{N_i}\le C \delta N_i^2$, which
is proved in Lemma \ref{lemmacardB}. From
\eqref{eq01c:prop:main1-015} it follows that $g(\underline
x_{\underline m}^i)=0$ for every $\underline m\in A_{N_i}$. Also, by
\eqref{eq01a:prop:main1-015}
\[
\Big( h_i^2\sum_{\underline m\in
C_{N_i}}|g(\underline{x}^i_{\underline{m}})|^2\Big)^{1/2}\le
C\Big(\delta^{-1} h_j|\log h_j|+\delta \Big)\|\mathrm
P_{N_j}\xi\|_{0}.
\]
Finally, from \eqref{eq01b:prop:main1-015}, we obtain
\[
\Big( h_i^2\sum_{\underline m\in B
_{N_i}}|g(\underline{x}^i_{\underline{m}})|^2\Big)^{1/2}\le C
\delta^{1/2} |\log h_j|^{3/2} \|\mathrm P_{N_j}\xi\|_{0}.
\]
Going back to \eqref{eq05:prop:main1-02}, we have proved that
\begin{equation}\label{eq09:prop:main1-02}
 \|\mathrm Q_{N_i}(\mathrm A^{ij}_{\rm sing,\delta}
-\mathrm A^{ij}_{{\rm sing},\delta,h})\mathrm P_{N_j}\xi\|_{0}\le
C\Big(\delta^{-1} h|\log h|+\delta^{1/2}|\log h|^{3/2}
\Big)\|\mathrm P_{N_j}\xi\|_{0}.
\end{equation}
A direct estimate from \eqref{eq05:prop:main1-02}, using
\eqref{eq02:prop:main1-015} now, gives
\begin{equation}\label{eq10:prop:main1-02}
 \|\mathrm Q_{N_i}(\mathrm A^{ij}_{\rm sing,\delta}
-\mathrm A^{ij}_{{\rm sing},\delta,h})\mathrm P_{N_j}\xi\|_{0}\le
C_r\delta^{1-r}h^{r}\|\xi\|_{r}+C \delta \|\mathrm G_j \mathrm
P_{N_j}\xi\|_{\infty,I_2}.
\end{equation}

Using the decomposition \eqref{eq03:prop:main1-02} and the bounds
\eqref{eq04:prop:main1-02} with $t=1$, \eqref{eq05b:prop:main1-02}
with $t=1+\varepsilon$ and \eqref{eq09:prop:main1-02} we can easily
prove \eqref{eq01:prop:main1-02}. At the same time,
\eqref{eq02:prop:main1-02} follows from the same decomposition,
using now \eqref{eq04:prop:main1-02} with $t=r+1$,
\eqref{eq05a:prop:main1-02} with $t=r$ and
\eqref{eq10:prop:main1-02}. In a last step we use the fact that the
cut-off operators $\mathrm G^j$ are bounded in $L^\infty(I_2)$.
This finishes the proof of the Proposition.
\end{proof}

\begin{lemma}\label{lemmacardB} There exists $C$ independent of $\delta>h$
such that
\[
 \# \big\{\underline m \in \mathbb Z_{N_i}\, :\, \underline x_{\underline
m}^i \in \Omega^{ij}_{2\delta}
 \setminus\Omega^{ij}_{\delta/2}\big\} \le C\delta N_i^2.
\]
\end{lemma}

\begin{proof}
We start by setting $\Gamma^j:=\partial (\supp\omega^j)$, which has
been assumed to be the finite union of Lipschitz arcs. Therefore,
for every $\delta>0$, we can pick up a finite set $\Gamma_\delta^j
\subset \Gamma^j$, with the following properties
\[
\Gamma \subset \bigcup_{\r \in \Gamma_\delta^j}
\overline{B(\r,\epsilon_1\delta)} \qquad \mbox{with
$\#\Gamma_\delta^j \le C \delta^{-1}$}.
\]
The constant $C$ that controls the size of the discrete set
$\Gamma_\delta^j$ depends on the Lipschitz constants related to
$\Gamma^j$. On the other hand, for $\delta$ small enough
\[
S^{j}_{2\delta}\setminus (\supp \omega^j)\subset\bigcup_{{\bf
r}\in\Gamma} \overline{B({\bf r},4\epsilon_1\delta)}
\]
and therefore
\begin{equation}\label{eq:lemmaB}
S^j_{2\delta}\setminus S^j_{\delta/2}\subset S^j_{2\delta}\setminus
(\supp\omega^j) \subset \bigcup_{\r \in \Gamma^j_\delta}
\overline{B(\r, 5\epsilon_1\delta)}.
\end{equation}
We now go back to the parametric domain using $(\r^i)^{-1}$ to
define the sets
\[
\mathcal O_{\r}:= (\r^i)^{-1} (S^i\cap \overline{B(\r,
5\epsilon_1\delta)}) \qquad \r \in \Gamma_\delta^j.
\]
Since $|\underline u-\underline v|\le C^i|\r^i(\underline
u)-\r^i(\underline v)|$ for all $\underline u,\underline v \in D^i$,
each set $\mathcal O_{\r}$ is either empty or can be surrounded by a
closed ball of radius $5 C^i\epsilon_1\delta$. This gives a
collection of points $\{ \underline u_1^\delta,\ldots,\underline
u_{k_\delta}^\delta\}\subset \mathbb R^2$ with $k_\delta \le C
\delta^{-1}$ such that
\begin{eqnarray*}
\Omega^{ij}_{2\delta}\setminus \Omega^{ij}_{\delta/2} &=&
(\r^i)^{-1} (S^{ij}_{2\delta}\setminus S^{ij}_{\delta/2}) =
(\r^i)^{-1}
(S^i \cap (S^j_{2\delta}\setminus S^j_{\delta/2}))\\
&\subset & \bigcup_{\r \in \Gamma^j_\delta} \mathcal O_\r \subset
\bigcup_{\ell=1}^{k_\delta} \big\{ \underline v \in \mathbb R^2\,
:\, |\underline v -\underline u_\ell^\delta|\le 5
C^i\epsilon_1\delta\big\}.
\end{eqnarray*}
The number of points of the uniform grid $\{ \underline
x_{\underline m}^i\,:\, \underline m \in \mathbb Z_{N_i}\}$ that fit
in a closed disk of radius $r$ is bounded by $(2 r N_i+1)^2$.
Therefore, the cardinal of the intersection of the uniform grid with
$\Omega^{ij}_{2\delta}\setminus \Omega^{ij}_{\delta/2}$ can be
bounded by
\[
k_\delta (10 C^i \epsilon_1\delta\,N_i+1)^2 \le C \delta^{-1} (10
C^i \delta\,N_i+1)^2 \le C' \delta N_i^2,
\]
where we have applied that $1\le \delta N_i$ (which is implied by
$\delta > h$), and the result follows.
\end{proof}

\paragraph{Proof of Theorem \ref{theo:main1}.} This result follows now
from an adequate choice of parameters in Propositions
\ref{prop:main1-01} and \ref{prop:main1-02}. Since $\delta \approx
h^\beta$, then $\delta^{-1} h \approx h^{1-\beta}$. We choose
$\varepsilon=(1-\beta)/(2\beta)$, so that $\delta^{-1-\varepsilon} h
\approx h^{(1-\beta)/2}$. We then apply Proposition
\ref{prop:main1-01} with $t=3$ and the first bound of Proposition
\ref{prop:main1-02} with the above $\varepsilon$. This yields the
bound
\[
 \|{\mathcal B}_{\delta,h}-{\mathcal B}_\delta\|_{{\cal H}^0\to {\cal
H}^0}\le C_\varepsilon (h^{3(1-\beta)}+ h^{(1-\beta)/2}+
h^{1-\beta}|\log h| +h^{\beta/2}|\log h|^{3/2}),
\]
which ensures convergence as $h \to 0$. The uniform bound for the
inverses of $\mathcal B_\delta$ (Theorem \ref{th:3.6}) ensures then
a uniform bound for the inverses of the operators $\mathcal
B_{\delta,h}$ when $h$ is small enough and $\delta \approx h^\beta$.
\hfill$\Box$

\paragraph{Proof of Theorem \ref{theo:main2}.}
For $\psi \in H^r(S)$, let us define $\boldsymbol\phi:=
((\omega^1\psi)\circ\mathbf r^1,\ldots,(\omega^J\psi)\circ\mathbf r^J)
\in \mathcal H^r$. By construction
 (cf. \eqref{eq:def:G}), $\mathcal G
\boldsymbol\phi=0$ . Applying Proposition \ref{prop:main1-01} and the
second bound of Proposition \ref{prop:main1-02} it follows that for
any integer $r\ge 3$
\begin{equation}
\|({\mathcal B}_{\delta,h}-{\mathcal B}_{\delta} )\bm{\phi}\|_0 \le
C_r\big( h^{r(1-\beta)} \|\bm{\phi}\|_r+ h^\beta\|\mathcal
P_h\boldsymbol\phi-\boldsymbol\phi \|_ {\infty,I_2}
\big).\label{eq03:proof:theo:main1}
\end{equation}
Taking now
\begin{equation}\label{eq:restr}
r\ge \max\{1/\beta,\beta/(1-\beta)\},
\end{equation}
so that $1+\beta\le 1+r(1-\beta)\le r$, from Sobolev's embedding
theorem and the approximation estimate \eqref{eq01:lemma:aux01} we
obtain
\begin{eqnarray*}
h^\beta \| \mathcal
P_h\boldsymbol\phi-\boldsymbol\phi\|_{\infty,I_2} &\le& C_\beta
h^\beta \| \mathcal
P_h\boldsymbol\phi-\boldsymbol\phi\|_{1+\beta}\le C_{r,\beta}
h^{r(1-\beta)}\| \boldsymbol\phi\|_{1+r(1-\beta)}\\
&\le& C_{r,\beta}
h^{r(1-\beta)}\|\boldsymbol\phi\|_r.
\end{eqnarray*}
Applying this estimate in
\eqref{eq03:proof:theo:main1} we deduce that for every integer $r\ge
3$ satisfying \eqref{eq:restr}, we have
\begin{equation}\label{eq05:proof:theo:main1}
\|({\mathcal B}_{\delta,h}-{\mathcal B}_{\delta} )\bm{\phi}\|_0 \le
C_r h^{r(1-\beta)}\|\bm{\phi}\|_r.
\end{equation}
Since $\beta$ is fixed, the dependence of the various bounding
constants on the parameter $\beta$ is dropped from the notation in
what follows (recall Remark \ref{remark_23}).

For $\psi\in H^r(S)$ we now define $\mathcal
M_{\delta,h}\psi:=(\mathcal B_{\delta,h}-\mathcal
B_\delta)\boldsymbol\phi$ with $\boldsymbol\phi$ as above; clearly
$\mathcal M_{\delta,h}: H^r(S)\to \mathcal H^0$ is a continuous
map. In view of~\eqref{Sobolev1}
equation~\eqref{eq05:proof:theo:main1} can be re-expressed in the form
\begin{equation}\label{Mdh}
  \|\mathcal M_{\delta,h}\psi\|_0 \le C_r
  h^{r(1-\beta)}\|\psi\|_{H^r(S)} \qquad \forall \psi \in H^r(S).
\end{equation}
Since $\|\mathcal B_{\delta,h}-\mathcal B_\delta\|_{\mathcal
  H^0\to\mathcal H^0}\to 0$, it follows that the sequence $\mathcal
M_{\delta,h}$ is uniformly bounded,
\[
\|\mathcal M_{\delta,h}\psi\|_0 \le C \|\psi\|_{H^0(S)},
\]
and therefore, using Sobolev interpolation theory ~\cite[Appendix
B]{McLean:2000} for the operator $\mathcal M_{\delta,h}$, it follows
that~\eqref{Mdh}, and therefore~\eqref{eq05:proof:theo:main1}, hold
for all $r\ge 0$.

Let now $\boldsymbol\phi$ and $\boldsymbol\phi_h$ be the respective
solutions of $\mathcal B_\delta\boldsymbol\phi=\mathbf U$ and
$\mathcal B_{\delta,h}\boldsymbol\phi_h=\mathcal Q_h\mathbf U$. From
Theorem~\ref{theo:main1} it follows that for $h$ small enough, the
inverse of ${\mathcal B}_{\delta,h}:{\cal H}^0\to {\cal H}^0$ is
uniformly bounded with respect to $h$.  Hence
\begin{eqnarray}
 \|\bm{\phi}-\bm{\phi}_h\|_0&\le&
C \|{\mathcal B}_{\delta,h}(\bm{\phi}-\bm{\phi}_h)\|_0\le C
\|{\mathcal B}_{\delta}  \bm{\phi}-{\mathcal
B}_{\delta,h}\bm{\phi}_h\|_0+
C\|({\mathcal B}_{\delta,h}-{\mathcal B}_{\delta} )\bm{\phi}\|_0\nonumber\\
&=&C \|{\bf U}  -\mathcal Q_h{\bf U}\|_0 +C \|({\mathcal
B}_{\delta,h}-{\mathcal B}_{\delta}
)\bm{\phi}\|_0\label{eq:01:proofTheoMain02}.
\end{eqnarray}
But, from~\eqref{eq02:lemma:aux01}, it follows that for $t>1$
\begin{equation}
 \label{eq02:proof:theo:main1}
\|{\bf U}  -\mathcal Q_h{\bf U}\|_0\le C_t h^t\|{\bf U}\|_{t}=C_t
h^t\|{\mathcal B}_{\delta_0} \bm{\phi}\|_{t}\le C'_t
h^t\|\bm{\phi}\|_t.
\end{equation}
Thus, substituting~\eqref{eq05:proof:theo:main1}
and~\eqref{eq02:proof:theo:main1} in~\eqref{eq:01:proofTheoMain02}, we
conclude that for all $t>1$ there exists $C_t>0$ such that
\begin{equation}\label{eq04:proof:theo:main1}
 \|\bm{\phi}-\bm{\phi}_h\|_0\le C_t h^{t(1-\beta)}\|\boldsymbol\phi\|_{t}.
\end{equation}
Error estimates in stronger Sobolev norms can now be obtained by means
of the inverse inequalities~\eqref{eq:invIneqSobolev} together with
the fact that $\mathcal P_h$ provides the best approximation, for all
periodic Sobolev norms, in the discrete subspaces of trigonometric
polynomials. Indeed, for $s>0$ we have
\begin{equation}\label{eq:stronger}
\|\boldsymbol\phi-\boldsymbol\phi_h\|_s  \le  \|\boldsymbol\phi-
\mathcal P_h\boldsymbol\phi\|_s + C h^{-s} \|\mathcal
P_h\boldsymbol\phi-\boldsymbol\phi_h\|_0 \le \|\boldsymbol\phi-
\mathcal P_h\boldsymbol\phi\|_s + C h^{-s} \|
\boldsymbol\phi-\boldsymbol\phi_h\|_0.
\end{equation}
The proof now follows by substituting \eqref{eq01:lemma:aux01} and
\eqref{eq04:proof:theo:main1} in~\eqref{eq:stronger}. \hfill$\Box$

\subsection{Proof of Theorem \ref{theo:MAIN}}

In view of the reconstruction formula of Theorem \ref{prop:2per2cont}
and since $\mathcal B_\delta\boldsymbol\phi=\mathbf U$, the exact
solution of \eqref{eq:intEqn} can be expressed in the form
\begin{eqnarray*}
 \psi(\r)&=&\sum_{i\in\mathcal I(\r)} (\omega^i
\psi^i)\left((\r^i)^{-1}(\r)\right)=\sum_{i\in\mathcal I(\r)}
\phi^i\left((\r^i)^{-1}(\r)\right)\\
&=&\sum_{i\in\mathcal I(\r)} 2\bigg(U^i-\sum_{j=1}^J\mathrm
A^{ij}_{\rm reg,\delta}\phi^j-\sum_{j=1}^J \mathrm A^{ij}_{\rm {\rm
sing},\delta}\phi^j\bigg)\left((\r^i)^{-1}(\r)\right).
\end{eqnarray*}
Similarly, for the discrete Nystr\"{o}m solution we have the relation
\begin{eqnarray*}
 \psi_h(\r)&=&\sum_{i\in\mathcal I(\r)} (\omega^i
\psi^i_h)\left((\r^i)^{-1}(\r)\right)\\
&=&\sum_{i\in\mathcal I(\r)} 2\Big(U^i-\sum_{j=1}^J\mathrm
A^{ij}_{\rm reg,\delta}\mathrm D_{N_j}\phi_h^j-\sum_{j=1}^J \mathrm
A^{ij}_{ {\rm
sing},\delta,h}\phi_h^j\Big)\left((\r^i)^{-1}(\r)\right),
\end{eqnarray*}
which follows by re-expressing equations~\eqref{eq:nyst2}
and~\eqref{eq:nyst3} in terms of $\phi_h^j=\mathrm
Q_{N_j}((\omega^j\circ\r^j)\psi^j_h)$.

Since $\|\psi\|_{H^s(S)}=\|\boldsymbol\phi\|_s$, by Propositions \ref{prop:5.8}
and \ref{prop:5.9} below we thus have
\begin{eqnarray*}
\|\psi-\psi_h\|_{L^2(S)} &\le & C\max_{i,j} \Big( \|\mathrm
A^{ij}_{\mathrm{reg},\delta}\mathrm D_{N_j}\phi^j_h-\mathrm
A^{ij}_{\mathrm{reg},\delta}\phi^j\|_0+
\|\mathrm A^{ij}_{\mathrm{sing},\delta}\phi^j_h-\mathrm
A^{ij}_{\mathrm{sing},\delta}\phi^j\|_0\Big)\\
& \le & C_t \Big( h^{t-1}+|\log h|^{3/2} h^{t(1-\beta)}\Big)
\|\boldsymbol\phi\|_t\\
&=&
C_t \big( h^{t-1}+|\log h|^{3/2} h^{t(1-\beta)}\big) \|\psi\|_{H^t(S)},
\end{eqnarray*}
and \eqref{eq:MAIN01} follows. Note that the constants $C_t$ depend
on $\beta$ (see Remark \ref{remark_23}). The $L^\infty(S)$ bound
\eqref{eq:MAIN02} follows from similar arguments, and the proof of
Theorem \ref{theo:MAIN} is thus complete.\hfill$\Box$

\begin{proposition} \label{prop:5.8}For all $t>1$ there exists
$C_t>0$ such that
\begin{equation}
\label{eq:01:lemma:01:theo:MAIN01}
 \|\mathrm A^{ij}_{\rm
reg,\delta}\mathrm D_{N_j}\phi_h^j-\mathrm A^{ij}_{\rm
reg,\delta}\phi^j\|_0 \le C_t h^{t(1-\beta)}\|\boldsymbol\phi\|_{t}.
\end{equation}
In addition, for all $\varepsilon>0$ and $t\ge 2$ there exists
$C_{\varepsilon,t}>0$ such that
\begin{equation}
\label{eq:02:lemma:01:theo:MAIN01}
 \|\mathrm A^{ij}_{\rm
reg,\delta}\mathrm D_{N_j}\phi_h^j-\mathrm A^{ij}_{\rm
reg,\delta}\phi^j\|_{\infty,I_2}\le
C_{\varepsilon,t}h^{t(1-\beta)-(1+\varepsilon)\beta}\|\boldsymbol\phi\|_{t}.
\end{equation}
\end{proposition}

\begin{proof}
Using successively Proposition \ref{prop:3.3} (recall that
$\delta\approx h^{\beta}$), Lemma \ref{lemma:aux02} and Theorem
\ref{theo:main2} we derive the  first estimate:
\begin{eqnarray}
 \|\mathrm A^{ij}_{\rm
reg,\delta}\mathrm D_{N_j}\phi_h^j-\mathrm A^{ij}_{\rm
reg,\delta}\phi^j\|_0&\le&\|\mathrm A^{ij}_{\rm reg,\delta}(\mathrm
D_{N_j}\phi_h^j-\phi_h^j)\|_{0}+\|\mathrm A^{ij}_{\rm
reg,\delta}(\phi_h^j-\phi^j)\|_{0}\nonumber\\
&\le& C_t h^{-t\beta}\|\mathrm D_{N_j}\phi_h^j-\phi_h^j\|_{-t}+C_0
\|\phi_h^j-\phi^j\|_{0}\nonumber\\
&\le& C_t' h^{t(1-\beta)}(\|\phi_h^j\|_0 +\|\phi^j\|_{t})\le
C_t''h^{t(1-\beta)}\|\boldsymbol\phi\|_t. \label{eq04:theo:MAIN}
\end{eqnarray}
To derive an estimate in the $L^\infty(I_2)$ norm, we proceed
similarly:
\begin{eqnarray}
\|\mathrm A^{ij}_{\rm reg,\delta}\mathrm D_{N_j}\phi_h^j-\mathrm
A^{ij}_{\rm
reg,\delta}\phi^j\|_{\infty,I_2}&&\\
&&\hspace{-40pt}\le C_{\varepsilon}\left(\|\mathrm
A^{ij}_{\rm reg,\delta}(\mathrm
D_{N_j}\phi_h^j-\phi_h^j)\|_{1+\varepsilon}+\|\mathrm A^{ij}_{\rm
reg,\delta}(\phi_h^j-\phi^j)\|_{1+\varepsilon}\right)\nonumber\\
&&\hspace{-40pt}\le C_{\varepsilon,t}\Big( h^{-(1+\varepsilon+t)\beta}\|\mathrm
D_{N_j}\phi_h^j-\phi_h^j\|_{-t}+
h^{-(1+\varepsilon)\beta} \|\phi_h^j-\phi^j\|_{0}\Big)\nonumber\\
&&\hspace{-40pt}\le C_{\varepsilon,t}'\Big( h^{t(1-\beta)-(1+\varepsilon)\beta}
\|\phi_h^j\|_0 +
h^{t(1-\beta)-(1+\varepsilon)\beta} \|\phi^j\|_{t}\Big)\nonumber\\
&&\hspace{-40pt}\le
C_{\varepsilon,t}''h^{t(1-\beta)-(1+\varepsilon)\beta}\|\boldsymbol\phi\|_t,
\label{eq05:theo:MAIN}
\end{eqnarray}
where we applied sequentially Proposition \ref{prop:3.3}, Lemma
\ref{lemma:aux02} and Theorem \ref{theo:main2}.
\end{proof}

\begin{proposition}\label{prop:5.9}
For all $t\ge 2$, there exists $C_t>0$ such that
\begin{equation}\label{eq:p59a}
 \| \mathrm A^{ij}_{{\rm sing},\delta,h}\phi_h^j-\mathrm A^{ij}_{
{\rm sing},\delta}\phi^j\|_{0}\le C_t\big(|\log h|^{3/2}\
h^{t(1-\beta)}+  h^{t-1}\big)\|\boldsymbol\phi\|_{t}.
\end{equation}
Further, for all $\varepsilon>0$ and $t\ge 2$, there exists
$C_{\varepsilon,t}>0$ such that
\begin{equation}\label{eq:p59b}
 \|\mathrm A^{ij}_{{\rm sing},\delta,h}\phi_h^j-\mathrm A^{ij}_{
{\rm sing},\delta}\phi^j\|_{\infty,I_2}\le C_{\varepsilon,t}\big(
h^{t(1-\beta)-(1+\varepsilon)\beta}+h^{t-1}\big)\|\boldsymbol\phi\|_{t}.
\end{equation}
\end{proposition}

\begin{proof}
We start by considering the decomposition
\begin{eqnarray}
\mathrm A^{ij}_{{\rm sing},\delta,h}\phi_h^j-\mathrm A^{ij}_{{\rm
sing},\delta}\phi^j&=& (\mathrm A^{ij}_{{\rm sing },\delta,h}
-\mathrm A^{ij}_{{\rm sing},\delta})(\phi_h^j-\mathrm
P_{N_j}\phi^j)\nonumber\\
&&+  (\mathrm A^{ij}_{{\rm sing },\delta,h}-\mathrm
A^{ij}_{{\rm
sing},\delta}) \mathrm P_{N_j}\phi^j+\mathrm A^{ij}_{{\rm
sing},\delta}(\phi_h^j-\phi^j).\ \qquad
\label{eq:p59c}
\end{eqnarray}
In order to establish equations~\eqref{eq:p59a} and~\eqref{eq:p59b} we
first estimate the $L^\infty(I_2)$ norm (and, thus, the $H^0$ norm)
for the first two terms of the right hand side of~\eqref{eq:p59c}. For
the first term, from Proposition \ref{prop:main1-015} and
\eqref{eq01:lemma:aux01} we have
\begin{eqnarray}
 \|( \mathrm A^{ij}_{{\rm sing},\delta,h}-\mathrm A^{ij}_{{\rm
sing},\delta})(\phi_h^j-\mathrm
P_{N_j}\phi^j)\|_{\infty,I_2}\!\!&\le&\!\!C|\log h|^{3/2}\|
 \phi_h^j-\mathrm P_{N_j}\phi^j\|_0\nonumber \\
&\le& \!\! C|\log h|^{3/2}\big(\|\phi_h^j-\phi^j\|_0+
\|\mathrm P_{N_j}\phi^j-\phi^j\|_0\big)\nonumber\\
&\le&\!\! C_t|\log h|^{3/2}\big( h^{t(1-\beta)}
+h^{t}\big)\|\boldsymbol\phi\|_t.\label{eq:p59d}
\end{eqnarray}
For the second term, on the other hand, Proposition \ref{prop:5.5} and
the fact that $\mathrm G_j\phi^j=0$ show that for $t\ge 2$,
\begin{eqnarray}
\|( \mathrm A^{ij}_{{\rm sing},\delta,h}-\mathrm A^{ij}_{{\rm
sing},\delta}) \mathrm P_{N_j}\phi^j \|_{\infty,I_2} &\le& C_t
h^{t(1-\beta)+\beta}\|\mathrm P_{N_j}\phi^j\|_{t}+C
h^\beta\|\mathrm G_j\mathrm P_{N_j}\phi^j\|_{\infty,I_2} \nonumber\\
&\le & C_t h^{t(1-\beta)+\beta}\|\phi^j\|_t+ C h^\beta\|\mathrm
P_{N_j}\phi^j-\phi^j\|_{\infty,I_2} \nonumber \\
& \le & C_t h^{t(1-\beta)+\beta}\|\phi^j\|_t + C'_t h^{t-1}
\|\phi^j\|_t,\label{eq:p59e}
\end{eqnarray}
where in the last step we have applied the Sobolev embedding theorem
$H^{\beta+1}\subset L^\infty(I_2)$ and \eqref{eq01:lemma:aux01}.  It
thus remains to estimate the $L^\infty(I_2)$ and $H^0$ norms of the
last term on the right hand side of equation~\eqref{eq:p59c}.

We establish the $H^0$ bound first: in view of Proposition
\ref{prop:3.3} and Theorem \ref{theo:main2} we have
\begin{equation}
  \label{eq:p59f} \|\mathrm A^{ij}_{\rm
    sing,\delta}(\phi_h^j-\phi^j)\|_0\le C \|\phi_h^j-\phi^j\|_{0}\le
  C_t h^{t(1-\beta)}\|\boldsymbol\phi\|_{t};
\end{equation}
substituting this bound together with \eqref{eq:p59d} and
\eqref{eq:p59e} in \eqref{eq:p59c} yields~\eqref{eq:p59a}.  To obtain
the $L^\infty(I_2)$ bound we proceed similarly, but using the bound
\begin{eqnarray}
\|\mathrm A^{ij}_{\rm sing,\delta}(\phi_h^j-\phi^j)\|_{\infty,
I_2}&\le& C_r \|\mathrm A^{ij}_{\rm
sing,\delta}(\phi_h^j-\phi^j)\|_{1+r}\le C_{r}'
h^{-(1+r)\beta}
\|\phi_h^j-\phi^j\|_{r}\nonumber\\
&\le& C_{r,t}
h^{t(1-\beta)-(1+\varepsilon)\beta}\|\boldsymbol\phi\|_{t},
 \label{eq:p59g}
\end{eqnarray}
instead of~\eqref{eq:p59f}. The first inequality in
Equation~\eqref{eq:p59g} follows from the Sobolev embedding theorem. In view of
the relation $\delta\approx h^\beta$, in turn, the second
inequality results from~\eqref{eq:p33b} with
$r:=\beta\,\varepsilon/(1+\beta)>0$. The third inequality, finally,
follows from Theorem \ref{theo:main2} and the fact that
$r+(1+r)\beta=(1+\varepsilon)\beta$. The proof is now complete.
\end{proof}

\appendix

\section{Error analysis of the polar coordinate quadrature
  rules}\label{sec:A}
The discussion in this Appendix, which lies at the heart of the main
convergence proof presented in this paper, provides a proof of Theorem
\ref{th:A.1}

\subsection{Error estimates for radial integration with trapezoidal rules}

This section provides bounds on the error
\begin{equation}\label{eq:A.1.1}
{\cal E}(\widetilde\xi_N;\theta):=\bigg|\int_{-\infty}^{\infty} (\chi_\delta
\widetilde\xi_N)(\rho,\theta)\,{\rm d}\rho-c(\theta)h\! \sum_{q=-\infty}^\infty
(\chi_\delta \widetilde\xi_N)(\gamma(\theta)+qc(\theta)h,\theta)\bigg|
\end{equation}
that arises from the radial trapezoidal quadrature rule, where
$\gamma$ is a given $2\pi$-periodic function in $L^\infty(\mathbb R)$
and $\{\chi_\delta\}$ satisfies conditions~\eqref{eq:4.4}, and where
$\widetilde\xi_N(\rho,\theta)=\xi_N(\rho\,\underline e(\theta)\,)$ for
a given trigonometric polynomial $\xi_N \in \mathbb T_N$. Note that
the mesh-size $h\,c(\theta)$ and the location of the node
corresponding to $q=0$ are allowed to depend on $\theta$. As a first
step we will limit our analysis to the case of trigonometric
monomials, and we thus study the quantity
\begin{equation}\label{eq:A.1.2}
E_{\underline{m}}(\theta):={\cal E}(\widetilde e_{\underline
m};\theta)= \bigg|\int_{-\infty}^{\infty}
(\chi_\delta\widetilde{e}_{\underline{m}})(\rho,\theta)\,{\rm
d}\rho-c(\theta)h \sum_{q=-\infty}^\infty
(\chi_\delta\widetilde{e}_{\underline{m}})(\gamma(\theta)+qc(\theta)h,
\theta)\bigg|,
\end{equation}
where $\widetilde e_{\underline m}(\rho,\theta)=e_{\underline
  m}(\rho\,\underline e(\theta)) = \exp(2\pi{\rm i}\rho\, \underline
m\cdot\underline e(\theta))$

For univariate functions defined on $(-1/2,1/2)$, we denote the
Fourier coefficients by
\[
\widehat{g}(m):=\int_{-1/2}^{1/2}g(\rho)\,{\exp}(-2\pi{\rm
i}m\rho)\,{\rm d}\rho.
\]
The Wiener algebra is the vector space of all functions whose Fourier
series is absolutely (and therefore uniformly) convergent:
\[
\mathcal R:= \big\{ g: (-{\textstyle\frac12},{\textstyle\frac12})
\to\mathbb C\, :\, \sum_{m=-\infty}^\infty |\widehat g(m)|<
\infty\big\}.
\]
The Fourier series of $g\in \mathcal R$ gives a 1-periodic extension of the
function to the entire real line.

\begin{lemma}\label{lemma:A.1}
Let $f \in \mathcal R$ be $1$-periodically extended to the real
line. Then the bound
\[
\left| \widehat f(m)-h \sum_{q=0}^{N-1} f(\gamma + q\,h)
\exp(-2\pi{\rm i} m (\gamma+q\,h))\right| \le \sum_{\ell\neq
0}|\widehat f(m+\ell N)|
\]
holds for all $m$, $\gamma \in \R$ and $N$.
\end{lemma}

\begin{proof} Since the Fourier series of $f$ converges absolute
and uniformly, we can justify the following equality:
\begin{eqnarray*}
  h\sum_{q=0}^{N-1} (f\exp(-2\pi{\rm i}m\,\cdot\,))(\gamma+qh)&=& \sum_{n\in\Z}
\widehat{f}(n)\bigg( h \sum _{q=0}^{N-1} \exp(2\pi{\rm
i}(n-m)(\gamma+qh))\bigg)\\
&& \ \hspace{-1.5cm} = \sum_{n\in\Z} \widehat{f}(n)\exp(2\pi{\rm
i}(n-m)\gamma)  \bigg( h \sum _{q=0}^{N-1} \exp(2\pi{\rm i}\, q
(n-m)h)\bigg).\qquad
\end{eqnarray*}
However,
\[
 \frac{1}{N}\sum _{q=0}^{N-1}
\exp(2\pi{\rm i}kq/N)=\left\{\begin{array}{ll}
1,\ &\mbox{if $k/N\in \Z$ },\\
0,\ &\mbox{otherwise},
\end{array}
\right.
\]
and therefore
\[
 h\sum_{q=0}^{N-1} (f\exp(-2\pi{\rm i}m\,\cdot\,))(\gamma+qh)=
\sum_{\ell=-\infty}^\infty  \widehat{f}(m+\ell N) \exp(2\pi{\rm
i}\ell N\gamma),
\]
which implies the result by separating the term corresponding to
$\ell=0$ in the right-hand side.
\end{proof}

\begin{lemma}\label{lemma:A.3} Let $E_{\underline m}$ be given by
  \eqref{eq:A.1.2}. Then there exists a constant $C$ such that
\begin{equation}\label{eq:A.3.a}
E_{\underline m}(\theta) \le C (\delta+h), \qquad \forall
\theta\in[0,2\pi]\quad\mbox{and} \qquad \forall \underline m \in \Z^2.
\end{equation}
For all $r\ge 2$ there exists $C_r>0$ such that
\begin{equation}
\label{eq:A.3.b}
  E_{\underline{m}}(\theta) \le C_r \delta^{1-r}\left\{\begin{array}{ll}
\displaystyle \frac{1}{(N-|m_1\cot \theta |-|m_2|)^r}, &
\mbox{if $|\tan \theta |\ge|\cot \theta |$},\\[3ex]
\displaystyle  \frac{1}{(N-|m_1|-|m_2\tan \theta |)^r},&\mbox{
otherwise},
\end{array}
\right.
\end{equation}
for all $m_1$, $m_2$ satisfying $|m_1|+|m_2|\le N-1$ and all $\theta
\in [0,2\pi]$.
\end{lemma}

\begin{proof} Let $Q$ denote the cardinality $Q:=\# \{ q\,:\,
  \gamma(\theta) + q c(\theta) h \in (-c_0\delta,c_0\delta)\}$. Then,
  using~\eqref{eq:4.4} and since, clearly, $Q \leq
  \left(\frac{2c_0\delta}{c(\theta) h}+1\right)$, we have
\begin{equation}
E_{\underline m}(\theta)  \le  \int_{-c_0\delta}^{c_0\delta}
|\chi_\delta(\rho,\theta)|\mathrm d\rho + c(\theta) h
\|\chi_\delta\|_{L^\infty(\mathbb R^2)} Q
 \le  C_0\Big(2c_0\delta + c(\theta)
h\Big(\frac{2c_0\delta}{c(\theta) h}+1\Big)\Big),
\end{equation}
and~\eqref{eq:A.3.a} follows.

To establish~\eqref{eq:A.3.b}, in turn, we first consider the case
$\theta \in [-\pi/4,\pi/4]$, so that $c(\theta)=1/\cos\theta$ and
$|\tan \theta|\le 1\le |\cot\theta|$.  Clearly
\begin{equation}
E_{\underline m}(\theta) =\bigg|\int_{-\infty}^\infty
f\left(\frac\rho{c(\theta)}\right)\exp\left(2\pi {\rm
i}m_1\frac{\rho}{c(\theta)}\right)\,{\rm d}\rho- c(\theta)
h\sum_{q=-\infty}^\infty \big(f\exp(2\pi
{\rm i}m_1\,\cdot\,)\big)\Big(\frac{\gamma(\theta)}{c(\theta)}+qh\Big)\bigg|
\end{equation}
where
\[
f(r):= \chi_\delta(c(\theta)r,\theta)\, \exp\left( 2\pi{\rm
i}m_2 c(\theta) r \sin
\theta\right)=\chi_\delta(c(\theta)r,\theta) \,\exp\left(
2\pi{\rm i}m_2 r \tan \theta\right)
\]
 Since $c_0\delta
<1/2$ and $c(\theta)\ge 1$,
\[
\supp \chi_\delta(c(\theta)\hspace{2pt}\cdot\hspace{2pt},\theta)
\subset (-c_0\delta/c(\theta),c_0\delta/c(\theta))\subset
(-1/2,1/2),
\]
it follows that the $\mathcal C^\infty$ function $f$ is compactly
supported in $(-1/2,1/2)$, and therefore $f\in \mathcal R$. Therefore,
by Lemma \ref{lemma:A.1} we have
\begin{eqnarray}\nonumber
E_{\underline m}(\theta) &=&
c(\theta)\bigg|\int_{-\frac12}^{\frac12}
f(r)\exp(2\pi {\rm i}m_1\,r)\,{\rm d}r -  h
\sum_{q=-\infty}^\infty \big(f\exp(2\pi
{\rm i}m_1\,\cdot\,)\big)\Big(\frac{\gamma(\theta)}{c(\theta)}+qh\Big)\bigg|\\
&\le & c(\theta) \sum_{\ell\neq 0}|\widehat f(-m_1+\ell N)|.
\label{proofofLemma56}
\end{eqnarray}
Since for $|m_1|+|m_2|\le N-1$ we have $|-m_1+\ell \,N|\ge |m_2|+1$
for all $\ell \neq 0$, it is possible to establish~\eqref{eq:A.3.b}
from~\eqref{proofofLemma56} provided sufficiently restrictive bounds
on $|\widehat f(m)|$ for $|m|>|m_2|$ are obtained.

To do this, letting $b(r):= \chi_\delta(r,\theta)$ and integrating by
parts $r$ times, we obtain
\begin{eqnarray*}
\widehat f(m) &=& \int_{-\frac12}^{\frac12} b(c(\theta)\rho)
\exp(2\pi{\rm i}(-m+m_2\tan\theta)\rho)\, \mathrm d \rho\\
&=&\left( \frac{-c(\theta)}{2\pi{\rm
i}(-m+m_2\tan\theta)}\right)^r\int_{-\frac12}^{\frac12}
b^{(r)}(c(\theta) \rho)\exp(2\pi{\rm i}(-m+m_2\tan\theta)\rho)\,
\mathrm d \rho.
\end{eqnarray*}
Using \eqref{eq:4.4} it therefore follows that
\begin{eqnarray*}
|\widehat f(m)| &\le& \left(\frac1{\sqrt 2\pi}\right)^r
\frac1{|-m+m_2\tan \theta|^r}\int_{-\infty}^{\infty}|
b^{(r)}(c(\theta)\rho)| \mathrm d \rho \\ & = &
\left(\frac1{\sqrt2\pi}\right)^r \frac1{|-m+m_2\tan \theta|^r}
\frac1{c(\theta)}\int_{-c_0\delta}^{c_0\delta} |b^{(r)}(\rho)|
\mathrm d\rho\\
& \le & \frac{C_r} {|-m+m_2\tan \theta|^r} \delta^{1-r}.
\end{eqnarray*}
Inserting this inequality on the right-hand side of
\eqref{proofofLemma56} we obtain
\begin{equation}\label{eq:4.4bis}
E_{\underline m}(\theta) \le C_r  \delta^{1-r} \sum_{\ell\neq
0}\frac1{|m_1+m_2\tan\theta-\ell N|^r}.
\end{equation}
Since $|\tan\theta|\le 1$, the desired inequality~\eqref{eq:A.3.b}
follows from~\eqref{eq:4.4bis} and Lemma \ref{lemma:tech1}. This
completes the proof for the case $\theta \in [-\pi/4,\pi/4]$. For
other values of $\theta$ the proof follows similarly.
\end{proof}

\begin{lemma}\label{lemma:A.4}
  For all $r\ge 2$ there exists $C_r>0$ such that, for all $m_1$,
  $m_2$ satisfying $|m_1|+|m_2|\le N/2$ we have
\[
E_{\underline m}(\theta) \le C_r \delta^{1-r} h^r , \qquad \forall
\theta\in [0,2\pi].
\]
\end{lemma}

\begin{proof}
  From Lemma \ref{lemma:A.3} we obtain
\[
E_{\underline m}(\theta) \le C_r \delta^{1-r}
\frac1{(N-|m_1|-|m_2|)^r}\qquad \forall \theta\in[0,2\pi], \qquad
|m_1|+|m_2|\le N-1.
\]
The proof is completed by noting that for $|m_1|+|m_2|\le N/2$ we have
$(N-|m_1|-|m_2|)^{-r}\le (N/2)^{-r}= 2^r h^r$.
\end{proof}

In what follows we consider the $\pi/2$-periodic continuous even
functions
\[
d(\theta):= \min\{|\tan\theta|,|\cot\theta|\},
\]
and
\begin{equation}\label{eq:defPsiN}
\Psi_N (\theta):= \Big( \sum_{(n,m)\in \Xi_N^+}
(N-n-m\,d(\theta))^{-4}\Big)^{1/2},
\end{equation}
where
\begin{equation}\label{eq:a.13}
\Xi_N^+ := \{ (n,m)\,:\, 0 < n,m\le N/2 < n+m<N\}.
\end{equation}

\begin{lemma}
\label{lemma:auxQuad2}
There exists $C_2>0$ such that the quantity $\mathcal E$ defined
in~\eqref{eq:A.1.1} satisfies
\begin{eqnarray*}
  \mathcal E(\widetilde\xi_N; \theta)&\le&
  C_2\Big(\delta^{-1}h+\delta^{-1} \Psi_N(\theta)+   \delta\Big)\,
  \|\xi_N\|_0, \qquad \forall \xi_N \in \mathbb T_N, \quad
  \forall\theta\in [0,2\pi],
\end{eqnarray*}
\end{lemma}
\begin{proof}
Consider the mutually disjoint sets of indices
\begin{eqnarray*}
\Lambda_{N}&:=&\Big\{\underline{m}=(m_1,m_2)\in\Z^2\::\:
|m_1|+|m_2|\le
N/2\Big\},\\
\Xi_{N}&:=&\Big\{\underline{m}=(m_1,m_2)\in\Z^2\::\:
|m_1|,|m_2|\le N/2<|m_1|+|m_2|< N\Big\}\\
&=& \Big\{ \underline m \in \Z^2\,:\, (|m_1|,|m_2|)\in
\Xi_N^+\Big\},\\
\Delta_N &:=& \Big\{ \underline m \in \Z^2\,:\, m_1, m_2 = \pm
N/2\Big\}.
\end{eqnarray*}
We note that $\Delta_N=\emptyset$ for odd values of $N$ and $\Delta_N$
contains four different elements for even $N$.  Finally the set
$\Z_N^*$ defined in equation~\eqref{znstar} satisfies $\Z_N^* \subset
\Lambda_N \cup \Xi_N \cup \Delta_N$, with strict inclusion for even
values of $N$---since for $N$ even the set $\Z_N^*$ is not symmetric
with respect to the coordinate axes in $\mathbb{Z}^2$ while all the
index sets above are.

Given $\xi_N \in \mathbb T_N$ we have
\begin{equation}\label{eq025:lemma:auxQuad2}
  \mathcal E(\widetilde\xi_N; \theta)\le {\cal E}_1(\xi_N; \theta)+{\cal
E}_2(\xi_N; \theta)+{\cal
    E}_3(\xi_N; \theta),
\end{equation}
where
\[
{\cal E}_1(\xi_N; \theta) = \sum_{\underline{m}\in
\Lambda_N}\!\!\!
|\widehat{\xi}_N(\underline{m})|E_{\underline{m}}(\theta)
\, ,\, {\cal E}_2(\xi_N; \theta) =
\sum_{\underline{m}\in\Xi_{N}}|\widehat{\xi}_N(\underline{m})|E_{\underline
{m}}(\theta)\, ,\,  {\cal
    E}_3(\xi_N; \theta) = \sum_{\underline m \in
\Delta_N}|\widehat{\xi}_N(\underline{m})|E_{\underline{m}}
(\theta)
\]
% \begin{eqnarray}
% \mathcal E(\widetilde\xi_N; \theta)&\le& \!\!\sum_{\underline{m}\in
% \Lambda_N}\!\!\!
% |\widehat{\xi}_N(\underline{m})|E_{\underline{m}}(\theta)
% +\sum_{\underline{m}\in\Xi_{N}}|\widehat{\xi}_N(\underline{m})|E_{\underline
% {m}}(\theta) +\sum_{\underline m \in
% \Delta_N}|\widehat{\xi}_N(\underline{m})|E_{\underline{m}}
% (\theta)\nonumber\\
% &=:&{\cal E}_1(\xi_N; \theta)+{\cal E}_2(\xi_N; \theta)+{\cal
% E}_3(\xi_N; \theta),\label{eq025:lemma:auxQuad2}
% \end{eqnarray}
with the understanding that the last sum vanishes when
$\Delta_N=\emptyset$.

Using Lemma \ref{lemma:A.4} with $r=2$ and the fact that the
cardinality $\# \Lambda_N$ of $\Lambda_N$ satisfies $\# \Lambda_N \le
C N^2= C h^{-2}$ we see that that
\begin{equation}
\mathcal E_1(\xi_N;\theta) \le \max_{\underline m\in
\Lambda_N}E_{\underline m}(\theta) \sum_{\underline m\in
\Lambda_N}|\widehat \xi_N(\underline m)|\le  C \delta^{-1} h \|
\xi_N\|_0. \label{eq03:lemma:auxQuad2}
\end{equation}
From Lemma \ref{lemma:A.3} with $r=2$, in turn, we have
\begin{eqnarray}\nonumber \sum_{\underline m\in \Xi_N} E_{\underline
m}(\theta)^2 & \le & C \delta^{-2} \sum_{\underline m \in \Xi_N}
(N-|m_1|-|m_2d(\theta)|)^{-4}\\
&= &C \delta^{-2} 4\sum_{\underline m \in \Xi_N^+}
(N-|m_1|-|m_2d(\theta)|)^{-4}. \label{eqA}
\end{eqnarray}
(To establish the first inequality in~\eqref{eqA} use the
transformation $(m_1,m_2) \leftrightarrow (m_2,m_1)$, which leaves
$\Xi_N$ invariant, in the cases for which $d(\theta)=|\cot\theta|$.)
Therefore
\begin{equation}\label{eq04:lemma:auxQuad2}
\mathcal E_2(\xi_N;\theta) \le C \delta^{-1}
\Psi_N(\theta)\|\xi_N\|_0.
\end{equation}
Finally, for the remaining term we use Lemma \ref{lemma:A.3} and the
fact that $\# \Delta_N\le 4$, and we thus obtain
\begin{equation}
  {\cal E}_3(\xi_N; \theta)  \le  C(\delta+h)
\Big(\sum_{\underline m\in \Delta_N}
\!\!|\widehat{\xi}_{N}(\underline{m})|^2\Big)^{1/2} \le
C(\delta+h)\|\xi_N\|_{0}. \label{eq05:lemma:auxQuad2}
\end{equation}
The lemma now follows from \eqref{eq025:lemma:auxQuad2} together with
\eqref{eq03:lemma:auxQuad2}, \eqref{eq04:lemma:auxQuad2} and
\eqref{eq05:lemma:auxQuad2}.
\end{proof}

\begin{lemma}\label{lemma:auxQuad3}
For all $r\ge 2$ there exists $C_r>0$ such that
\begin{eqnarray}
\mathcal E(\widetilde\xi_N; \theta)\le  C_r h^{r} \delta^
{1-r}\|\xi_N\|_r, \qquad \forall \xi_N \in \mathbb T_N, \qquad
\forall\theta\in[0,2\pi]. \label{eq01:cor:auxQuad}
\end{eqnarray}
\end{lemma}

\begin{proof} The proof of this result proceeds through consideration
  of equation~\eqref{eq025:lemma:auxQuad2}.  In view of
  Lemma~\ref{lemma:A.4} we obtain
\[
\mathcal E_1(\xi_N;\theta) \le C_r \delta^{1-r} h^r \sum_{\underline
m\in \Lambda_N} |\widehat\xi_N(\underline m)| \le C'_r\delta^{1-r}
h^r \| \xi_N\|_2,
\]
since for all $\xi \in H^2$ we clearly have
\begin{equation}\label{embedding}
\sum_{\underline m \in \Z^2} |\widehat\xi(\underline m)| \le C \|
\xi\|_2, \qquad C^2 =1+\sum_{\underline{0}\ne
\underline{m}\in\Z^2}\frac{1}{(|m_1|^2+|m_2|^2)^{2}}.
\end{equation}
From \eqref{eq04:lemma:auxQuad2} together with the fact that
$\Psi_N(\theta)\le 2$ (as established in Lemma \ref{lemma:propPsi}
below) and the inequality
\begin{equation}\label{tail}
\Big( \sum_{\underline m \in \Z^2\setminus\Lambda_N} |\widehat\xi
(\underline m)|^2\Big)^{1/2} \le\;\; h^{r} (2\sqrt2)^r\|\xi\|_r\quad
\forall \xi \in H^r\quad\mbox{with}\quad r\ge 0,
\end{equation}
on the other hand, we obtain
\[
\mathcal E_2(\xi_N;\theta) \le C_r \delta^{-1}  \Big(
\sum_{\underline m\in \Xi_N} |\widehat\xi_N(\underline
m)|^2\Big)^{1/2} \le C'_r \delta^{-1} h^r \| \xi_N\|_r.
\]
Finally, using the first inequality in \eqref{eq05:lemma:auxQuad2} and
\eqref{tail} we see that
\[
 {\cal  E}_3(\xi_N; \theta)\le  C\delta \Big( \sum_{\underline m\in \Delta_N}
|\widehat\xi_N(\underline m)|^2\Big)^{1/2} \le C'\delta
\,h^r\|\xi_{N}\|_r.
\]
This finishes the proof of the lemma.
\end{proof}

\subsection{Technical lemmas}
% 
% \begin{lemma}
% In the space of bivariate trigonometric polynomials $\mathbb T_N$ the
following inverse inequalities hold:
% \begin{equation}\label{eq:invIneqSobolev}
% \| \xi_N\|_t  \le  (\sqrt2\, h)^{s-t}\|\xi_N\|_s \qquad \forall\xi_N
% \in \mathbb T_N, \qquad s \le t,
% \end{equation}
% \begin{equation} \label{eq:invIneqLinfty}
% \|\xi_N\|_{\infty,I_2}\le
% \sum_{\underline{m}\in\mathbb{Z}_N^*}|\widehat{\xi}_N(\underline{m})|\le
%   N
% \bigg(\sum_{\underline{m}\in\Z_N^*}|\widehat{\xi}_N(\underline{m})|^2\bigg)^{
% 1/2}=  h^{-1 }\|\xi_N\|_{0} \quad \forall\xi_N \in \mathbb T_N.
% \end{equation}
% \end{lemma}
% \begin{proof} The proof follows directly
%   from~\eqref{eq:defSobolevNorm}.
% \end{proof}

\begin{lemma}\label{lemma:tech1}
Let $m,n \in\Z$ be such that $|m|+|n| \le N-1$ and let $r\ge 2$.
Then
\[
\sum_{\ell \neq 0}\frac1{|m+n \xi -\ell N|^r} \le
\frac{C_r}{(N-|m|-|n|\xi)^r} \qquad 0\le \xi \le 1.
\]
\end{lemma}

\begin{proof} Noting that $N\ge N-|m|-|n\xi|\ge N-|n|-|m|\ge 1$ we
  obtain the chain of inequalities
\begin{eqnarray*}
\sum_{\ell \neq 0}\frac1{|m+n \xi -\ell N|^r} & \le &
\frac2{(N-|m|-|n|\xi)^r}+2 \sum_{\ell =2}^\infty \frac1{(\ell
N-|m|-|n|)^r}\\
& \le & \frac2{(N-|m|-|n|\xi)^r}+ 2 \sum_{\ell=2}^\infty
\frac1{(\ell-1)^r
N^r}\\
& \le & \frac2{(N-|m|-|n|\xi)^r} \Big( 1+\sum_{\ell=1}^\infty
\frac1{\ell^r}\Big),
\end{eqnarray*}
and the result follows.
\end{proof}

\begin{lemma}\label{lemma:tech2} Let $\Xi_N^+$ be given by \eqref{eq:a.13}. Then
\[
\Phi_N(\xi):=\sum_{(n,m)\in \Xi_N^+}  \frac1{(N-n-m\xi)^4} \le
\min\{ 4, \frac{8}{N^2(1-\xi)^2}\} \qquad \forall N\ge 1, \quad \forall \xi\in
[0,1].
\]
\end{lemma}

%The index set $\Xi_N^+$ can be changed to

\begin{proof}
  Using the change of variables $(j,m)=(n+m,m)$ the summation limits
  become
\[
(N+1)/2 \le j \le N-1,\qquad \qquad j-N/2\le m \le N/2.
\]
We thus have
\begin{eqnarray*}
\Phi_N(\xi) &=& \sum_{\frac{N+1}2 \le j \le N-1} \Big(
\sum_{j-\frac{N}2\le m\le \frac{N}2} \frac1{(N-j+m(1-\xi))^4}\Big)\\
& \le & \sum_{\frac{N+1}2 \le j \le N-1} \Big( \sum_{j-\frac{N}2\le
m\le
\frac{N}2} \frac1{\left(N-j+(j-\frac{N}2)(1-\xi)\right)^4}\Big)\\
& \le & \sum_{\frac{N+1}2 \le j \le N-1} \frac1{\left(
\frac{N}2(\xi+1)-\xi j\right)^3}\frac{N-j+1}{N-j}\\
& \le & 2 \sum_{\frac{N+1}2 \le j \le N-1}
\int_{j-\frac12}^{j+\frac12} \frac{\mathrm d x}{\left(
\frac{N}2(\xi+1)-\xi\,x\right)^3}\le 2
\int_{\frac{N}2}^{N-\frac12}\frac{\mathrm d x}{\left(
\frac{N}2(\xi+1)-\xi\,x\right)^3}=: R_N(\xi),
\end{eqnarray*}
where the second inequality was obtained by using the fact that $\#\{
m:j-N/2\le m \le N/2\} \le N-j+1$ and where the bound that introduces
the integral relies on the fact that, for convex functions, the
midpoint rule underestimates the value of the integral.

To estimate $R_N(\xi)$ we proceed as follows:
\begin{eqnarray*}
R_N(\xi) &=& \frac2{N^2}\int_{\frac12}^{1-\frac1{2N}} \frac{\mathrm
d x}{\left( \frac12 (\xi+1)-\xi \,x\right)^3} = \frac1{N^2\xi}
\left(
\frac12(\xi+1)-\xi\,x\right)^{-2}\bigg|_{x=\frac12}^{x=1-\frac1{2N}}\\
&=&
\frac4{N^2\xi}\bigg(\Big(1-\xi+\frac{\xi}{N}\Big)^{-2}-1\bigg)\le
\frac4{N^2\xi}\big( (1-\xi )^{-2}-1\big)=
\frac4{N^2\xi}\frac{2\xi-\xi^2}{(1-\xi)^2}\\
& \le & \frac{8}{N^2(1-\xi)^2}.
\end{eqnarray*}
This bound is valid even when $\xi=0$ since $R_N(0)=8(N-1)/N^3
<8/N^2$.

Finally, noting that $\Phi_N(\xi)$ is an increasing function of
$\xi\in [0,1]$, we obtain
\[
\Phi_N(\xi)  \le  \Phi_N(1) \le R_N(1) =2
\int_{\frac{N}2}^{N-\frac1{2}}\frac{\mathrm d x}{(N-x)^3}< 2
\int_0^{N-\frac12}\frac{\mathrm d x}{(N-x)^3} < 4
\]
and the result follows.
\end{proof}

\begin{lemma}\label{lemma:propPsi} Let $\Psi_N$ be given by
\eqref{eq:defPsiN}. Then, there exists $C>0$ independent of $N$ and $\Theta$
such that
\[
\Psi_N(\theta) \le 2 \qquad \mbox{and}\qquad k \sum_{p=0}^{\Theta-1}
\Psi_N(p\,k) \le C(k+ h\log (1/k)),
\]
where $h=1/N$ and $k=2\pi/\Theta$.
\end{lemma}

\begin{proof}
Note that $\Psi_N=\sqrt{\Phi_N \circ d}$, which easily gives the
first inequality as a consequence of Lemma \ref{lemma:tech2}. A
simple computation shows that for every integer $\ell$,
\begin{equation}\label{eq:tech3:1}
1-d(\theta) \ge \frac4{\pi}\big| \theta - c_\ell\big| \qquad \theta
\in \left[c_\ell-\frac\pi4,c_\ell+\frac\pi4\right], \qquad c_\ell:=
\frac{(2\ell+1)\pi}4.
\end{equation}
Consider the sets
\[
\mathcal I_{\ell,k}:=\left[
c_\ell-\frac\pi4,c_\ell+\frac\pi4\right]\setminus
[c_\ell-k,c_\ell+k].
\]
Then, letting $\theta_p=p\,k$, it follows that
\[
\sum_{p=0}^{\Theta -1} \Psi_N(\theta_p) = \sum_{\ell=0}^3
\left(\sum_{\frac{\pi\ell}2 \le \theta_p < \frac{\pi(\ell+1)}2}
\Psi_N(\theta_p)\right)\\
\le  \sum_{\ell=0}^3 \left( 6+\sum_{\theta_p \in \mathcal
I_{\ell,k}} \Psi_N(\theta_p)\right),
\]
since the interval $[c_\ell-k,c_\ell+k]$ contains at most three of
the grid points $\theta_p$ and $\Psi_N \le 2$. Using Lemma
\ref{lemma:tech2} and \eqref{eq:tech3:1} we obtain
\[
\sum_{\theta_p \in \mathcal I_{\ell,k}} \Psi_N(\theta_p)\le
\sum_{\theta_p \in \mathcal I_{\ell,k}} \frac{\sqrt 8}{N
(1-d(\theta_p))} \le \frac{\pi}{N\sqrt2} \sum_{\theta_p \in \mathcal
I_{\ell,k}}\frac1{|\theta_p-c_\ell|}.
\]
By considering all possible arrangements of the uniform grid $\{
\theta_p\ :\ p=0,\dots,\Theta-1\}$ with respect to the double interval
$\mathcal I_{\ell,k}$, it is easy to check
 that for all $\ell$
\[
\sum_{\theta_p \in \mathcal I_{\ell,k}}\frac1{|\theta_p-c_\ell|}\le
2 \sum_{j=1}^{\lceil\Theta/8\rceil} \frac1{j\,k} < \frac2{k} \left(
1+\log\left\lceil \frac\Theta8\right\rceil\right)\le \frac2{k}
\left( 1+\log \frac{\Theta + 7}8\right) \le \frac2{k}(1+\log\Theta).
\]
The proof now results by combining the three previous inequalities
\[
k \sum_{p=0}^{\Theta-1} \Psi_N(\theta_p) \le 24 k +
\frac{4\sqrt2\pi}{N}(1+\log\Theta)
\]
and collecting constants.
\end{proof}

\subsection{Proof of Theorem \ref{th:A.1}}

The first part of the proof does not rely upon the relationships
between $N$ and $\Theta$ assumed in Theorem~\ref{th:A.1}: up to a
point in the proof that is clearly marked, $N$ and $\Theta$ are
arbitrary positive integers, and, thus, $k = 2\pi/\Theta $ and $h=2/N$
are unrelated real constants. {\em The proof does assume throughout
  that $h<\delta$, however.}

Let $\phi\in \mathcal C^\infty(\mathbb R^2)$ satisfy
$\phi(\rho,\hspace{2pt}\theta +2\pi)= \phi(\rho,\hspace{2pt}\theta)$
for all $\rho,\;\theta\in\mathbb{R}$, as well as $\supp
\phi(\cdot,\theta) \subset (-a,a)$ for all $\theta\in\mathbb{R}$. A
classical argument on the error of the trapezoidal rule for
 $\mathcal
{C}^\infty$ 
periodic functions (see for instance \cite[Theorem 9.26]{kress98})
shows that for all positive integers $r$ we have
\begin{eqnarray}\nonumber
\mathcal E_0(\phi) &:=& \left|\int_{-\infty}^\infty \Big(
\int_0^{2\pi} \phi(\rho,\theta)\,\mathrm
d\theta-k\sum_{p=0}^{\Theta-1} \phi(\rho,\theta_p)\Big) \mathrm
d\rho\right|\\
&\le& \int_{-a}^{a} \alpha_r k^r \left|\int_0^{2\pi}
\partial_\theta^r \phi(\rho,\theta)\mathrm d\theta\right| \mathrm
d\rho\le C a k^r \|\partial_\theta^r \phi\|_{L^\infty(\R^2)}.
\label{proof51:0}
\end{eqnarray}
Noting that
\begin{equation}\label{proof51:4}
\mathcal E_{h,k,\gamma}(\chi_\delta\widetilde\xi_N) \le \mathcal
E_0(\chi_\delta\widetilde\xi_N) + k \sum_{p=0}^{\Theta-1} \mathcal
E(\widetilde\xi_N;\theta_p)
\end{equation}
and that $\phi= \chi_\delta\, \widetilde\xi_N$ satisfies the
hypotheses required to apply \eqref{proof51:0} with $a=c_0\delta$, the
proof proceeds by producing bounds on the two terms on the right-hand
side of~\eqref{proof51:0}.

To obtain a bound for the second term on the right-hand side
of~\eqref{proof51:4} we first invoke Lemmas~\ref{lemma:auxQuad2}
and~\ref{lemma:propPsi} to establish that
\begin{eqnarray}\nonumber
k \sum_{p=0}^{\Theta-1} \mathcal E(\widetilde\xi_N;\theta_p) &\le& C
\Big( \delta^{-1} h +\delta^{-1} k
\sum_{p=0}^{\Theta-1}\Psi_N(\theta_p) + \delta\Big)\|\xi_N\|_0\\
& \le & C' \Big( \delta^{-1} h \log(1/k) +\delta^{-1}
k+\delta\Big)\|\xi_N\|_0, \qquad \forall\xi_N \in \mathbb T_N,
\label{proof51:5}
\end{eqnarray}
and we invoke Lemma~\ref{lemma:auxQuad3} to show that for $r\ge 2$ we
have
\begin{equation}\label{proof51:6}
k \sum_{p=0}^{\Theta-1} \mathcal E(\widetilde\xi_N;\theta_p) \le C
h^r \delta^{1-r}\|\xi_N\|_r, \qquad \forall \xi_N \in \mathbb T_N.
\end{equation}

To bound the first term on the right-hand side
of~\eqref{proof51:4}, on the other hand, we first note from the
inverse inequalities \eqref{eq:invIneqSobolev} and
\eqref{eq:invIneqLinfty} that, for all positive integers $\ell$, we
have
\begin{eqnarray*}
\| \partial_\theta^\ell \widetilde\xi_N\|_{L^\infty(\R^2)} & \le &
C_\ell \sum_{|\underline\alpha|\le \ell} \|
\partial_{\underline\alpha}\xi_N\|_{\infty,I_2}\le C_\ell h^{-1}
\sum_{|\underline\alpha|\le
\ell}\|\partial_{\underline\alpha}\xi_N\|_0 \\
&\le& C'_\ell h^{-1}\|\xi_N\|_\ell \le C''_\ell h^{-1-\ell}\|
\xi_N\|_0.
\end{eqnarray*}
(The first inequality follows from an application of the chain rule
and the definition of $\widetilde\xi_N$ in terms of $\xi_N$.)
Therefore
\begin{eqnarray}\nonumber
\| \partial_\theta^r (\chi_\delta
\widetilde\xi_N)\|_{L^\infty(\R^2)} & \le & C \sum_{\ell=0}^r \|
\partial_\theta^\ell
\chi_\delta\|_{L^\infty(\R^2)}\|\partial_\theta^{r-\ell}\widetilde\xi_N\|_{
L^\infty(\R^2)}\\
& \le & C'  \|\xi_N\|_0 \sum_{\ell=0}^{r}\delta^{-\ell} h^{\ell-r-1}
\le C'' h^{-1-r}  \|\xi_N\|_0, \label{proof51:1}
\end{eqnarray}
(the hypothesis $\delta^{-1}h < 1$ was used in the last
inequality). Additionally, using~\eqref{eq:invIneqLinfty} and
$\delta^{-1}h < 1$ we obtain
\begin{equation}\label{proof51:2}
\| \partial_\theta^r (\chi_\delta
\widetilde\xi_N)\|_{L^\infty(\R^2)} \le C h^{-1} \delta^{-r}\|
\xi_N\|_r.
\end{equation}
Inserting these inequalities into~\eqref{proof51:0} and since
$a=c_0\delta$ we obtain
\begin{equation}\label{proof51:3}
\mathcal E_0(\chi_\delta \widetilde\xi_N) \le C k^r \left\{
\begin{array}{l} \delta h^{-1-r}\|\xi_N\|_0,
\\[1.5ex] \delta^{1-r}h^{-1}\|\xi_N\|_r.\end{array}\right.
\end{equation}
{\em Using now the hypothesis $\Theta\approx N^{1+\alpha}$ with
  $\alpha>0$}, it follows that $k \le C h^{1+\alpha}\le C h$, and
therefore, for $r>1/\alpha$, $k^r\le C_\alpha h^{r+1}$. Inserting the latter
estimate in the inequalities~\eqref{proof51:3} we obtain
\begin{equation}\label{proof51:7}
\mathcal E_0(\chi_\delta\,\widetilde\xi_N) \le C_r h^r
\delta^{1-r}\|\xi_N\|_r \qquad \forall \xi_N \in \mathbb T_N,
\end{equation}
for $r=0$ and for all $r>1/\alpha$, respectively. Therefore,
equation~\eqref{proof51:7} holds for all $r\ge 0$ by
interpolation~\cite{SaVa:2002}.

Using~\eqref{proof51:4},~\eqref{proof51:5} and the $r=0$ instance
of~\eqref{proof51:7}, and noting that $\log (1/k) \le C_1 +
C_2\log(1/h) $ equation~\eqref{eq:A1.1}
follows. Equation~\eqref{eq:A1.2}, finally, follows
from~\eqref{proof51:7} and \eqref{proof51:6} (for $r\ge 2$). The proof
is now complete.

\section{A modified version of the Nystr\"om method}\label{sec:B}
\subsection{Hermite interpolation}

As it has been presented and analyzed in the previous sections, the
overall integral equation solver enjoys convergence of super-algebraic
order for smooth data. The algorithm~\cite{377360}, however,
incorporates certain additional interpolation processes for added
computational efficiency, as discussed in what follows.  A convergence
analysis for the modified algorithm is presented in
Section~\ref{sec:6.2}.

One of the most computationally expensive segment of the Nystr\"om
algorithm described in the previous sections concerns evaluation of
the discrete operator $\mathrm A^{ij}_{{\rm sing},\delta,h}\xi_{N_j}$:
the high cost of this operation stems from the evaluation it requires
of the trigonometric polynomial $\xi_{N_j}$ at a large number of
points
\begin{eqnarray}
  &&\big( q\, h_j, z_2+( q\,h_j-z_1)\tan \theta^j_p\big)_{q=0}^{N_j-1}
  ,\quad \mbox{if }
  |\cos \theta^j_p |\le \sqrt{2}/2,\label{eq:firstset}\\
  &&\big( z_1+( q\,h_j-z_2) \cot \theta^j_p,q\, h_j\big)_{q=0}^{N_j-1},\quad
  \mbox{if } |\sin \theta^j_p |<\sqrt{2}/2, \label{eq:secondset}
\end{eqnarray}
for all $p=0,\ldots,\Theta_j-1$ and all $j = 1, \ldots,J$ (see
(\ref{eq:points01}), (\ref{eq:points02}) and Figure
\ref{fig:points03}). The interpolation strategy mentioned at the
beginning of this section accelerates this evaluation process. To
describe the accelerated strategy in a manner that lends itself to
analysis, we consider once again the operator $\mathrm
L^{ij}_{\delta,h}$ in equation~\eqref{eq:newQuad-1}, we introduce the
indicator function
\[
b(\theta):=\left\{\begin{array}{ll} 1, & \mbox{if $|\cos\theta|\ge
      \sqrt2/2$}, \\ 0, & \mbox{otherwise}, \end{array}\right.
\]
and we define $\mathrm L^{ij,2}_{\delta,h}$ as the operator that is
obtained as the angle-dependent weights $c(\theta^j_p)$
in~\eqref{eq:newQuad-1} are substituted by $b(\theta^j_p) \,
c(\theta^j_p)$. Note that $\mathrm L^{ij,2}_{\delta,h}$ contains only
contributions from radial nodes lying on vertical grid lines (see
Figure \ref{fig:points03}); the remaining contributions to $
L^{ij}_{\delta,h}$ are collected in the operator $\mathrm
L^{ij,1}_{\delta,h}:= \mathrm L^{ij}_{\delta,h}-\mathrm
L^{ij,2}_{\delta,h}.$ Accordingly, the operators $\mathrm
A^{ij,1}_{\mathrm{sing},\delta,h}:=(\omega^i\circ\r^i) \mathrm
L^{ij,1}_{\delta,h}$ and $\mathrm
A^{ij,2}_{\mathrm{sing},\delta,h}:=(\omega^i\circ\r^i) \mathrm
L^{ij,2}_{\delta,h}$ provide the decomposition
\begin{equation}\label{decompo}
\mathrm  A^{ij}_{\mathrm{sing},\delta,h} = \mathrm
A^{ij,1}_{\mathrm{sing},\delta,h} +
\mathrm A^{ij,2}_{\mathrm{sing},\delta,h}
\end{equation}
of the operator in equation~\eqref{newKijsing}.

Each one of the operators on the right-hand side of~\eqref{decompo}
lends itself to rapid (high-order) evaluation through application of
univariate Hermite interpolation. To do this, given a positive integer
$d$ and a uniform mesh in the interval $[0,1]$ with mesh-size $k=1/M$,
we consider the Hermite interpolation operator, that is, the operator
$\mathrm H_{k,d}: \mathcal C^{d}[0,1]\to \{ g\in \mathcal
C^d[0,1]\,:\, g|_{[pk,(p+1)k]}\in \mathbb P_{2d+1} \quad \forall p\}$
(where $\mathbb P_{2d+1}$ is the space of polynomials of degree at
most $2d+1$), which satisfies
\[
(\mathrm H_{k,d} f)^{(r)}(p\,k)= f^{(r)}(p\,k), \qquad p=0,\ldots,M,
\qquad r=0,\ldots,d.
\]
The error estimate for this kind of piecewise polynomial
interpolation is well known:
\begin{equation}\label{eq:Hermite}
\| f-\mathrm H_{k,d}f\|_{L^\infty(0,1)}\le C_d k^{2d+2}\|
f^{(2d+2)}\|_{L^\infty(0,1)} \qquad \forall f \in \mathcal
C^{2d+2}[0,1].
\end{equation}

The Hermite interpolation method for approximation of the operator
$\mathrm L^{ij,1}_{\delta,h}\xi$ (resp. $\mathrm
L^{ij,2}_{\delta,h}\xi$) for a given $j$ proceeds by interpolation of
the trigonometric polynomial $\xi$ via an application of the operator
$\mathrm H_{k,d}$, in the variable $u_1$ (resp. $u_2$), on a refined
one-dimensional mesh of mesh-size $k = h_j/m$ ($m > 1$): the
approximated operators thus result as the needed values of $\xi$ at
points of the form \eqref{eq:secondset} (resp. \eqref{eq:firstset})
are replaced by the corresponding values produced by the Hermite
interpolant $\mathrm H_{h_j/m,d}\xi(\cdot,q\,h_j)$ (resp. $\mathrm
H_{h_j/m,d}\xi(q\,h_j,\cdot)$). Note that this interpolation process,
which requires evaluation of $\xi$ and some of its partial derivatives
at points $(p h_j/m,q h_j)$ (resp. $(qh_j,p h_j/m)$), can be
implemented efficiently by means of the Fast Fourier Transform;
see~\cite{377360}.

We express the Hermite-interpolation-based approximation of the
operators $\mathrm L^{ij,1}_{\delta,h}+\mathrm L^{ij,2}_{\delta,h}$ by
means of the operator
\[
\mathrm A^{ij,\,\mathrm{herm}}_{\mathrm{sing}, \delta,h,m,d}:=
(\omega^i\circ\r^i) \Big(\mathrm L^{ij,1}_{\delta,h} (\mathrm
H_{h_j/m,d}\otimes \mathrm I)+\mathrm L^{ij,2}_{\delta,h} (\mathrm I
\otimes \mathrm H_{h_j/m,d})\Big).
\]
The corresponding matrix of operators for given values of $h$, $m$ and
$d$ is denoted by $\mathcal A^{\mathrm{sing}}_{\delta,\mathrm{herm}}$;
note that, for notational simplicity, the dependence on the parameters
$h$, $m$ and $d$ is suppressed in the notation $\mathcal
A^{\mathrm{sing}}_{\delta,\mathrm{herm}}$.  The new
Hermite-interpolation-based discrete set of linear equations is given
by
\begin{equation}
\label{eq:numerSolution2}
 {\mathcal
B}_{\delta,\mathrm{herm}}{\bm{\phi}}_{\mathrm{herm}}:=\left({\textstyle\frac12}{
\cal I}+{\cal Q}_h {\mathcal A}_\delta^{\rm reg}{\cal D}_{h}{\cal
P}_{h}+
 {\cal Q}_h {\mathcal A}^{\rm sing}_{\delta,\mathrm{herm}}
{\cal P}_{h}\right){\bm{\phi}}_{\mathrm{herm}}={\cal Q}_h{\bf U}.
\end{equation}

In what follows we present a modified version of the analysis of
Sections~\ref{sec:3} through~\ref{sec:5}, which takes into account the
additional approximations associated with the Hermite interpolation
process described above.  This treatment thus accounts for all details
of the basic partition-of-unity/polar integration Nystr\"om
algorithm~\cite{377360}, and it establishes, with complete
mathematical rigour, its convergence and high-order accuracy.

\subsection{Modified analysis, acounting for Hermite
  interpolation~\label{sec:6.2}}

Analysis of the approximation (unique solvability, stability,
convergence) for \eqref{eq:numerSolution2} is carried out by
comparison with the discrete solution $\boldsymbol\phi_h$. As we
will show, $m$ will have to grow as $h$ decreases in order to retain
convergence properties of the method.

\begin{proposition}\label{mod:prop2}
Assume that $\delta \approx h^\beta$ with $\beta \in (0,1)$ and that
$m=m(h)$ is chosen so that $m^{-(2d+2)} h^{\beta-1}\to 0$. Then
\[
\lim_{h \to 0}\|\mathcal B_{\delta,h}-\mathcal
B_{\delta,\mathrm{herm}}\|_{\mathcal H^0\to\mathcal H^0}=0.
\]
Therefore, for $h$ small enough the discrete equations
\eqref{eq:numerSolution2} are uniquely solvable and the operators
$\mathcal B_{\delta,\mathrm{herm}}$ have uniformly bounded
inverses.\end{proposition}

\begin{proof}
Note that the operators $\mathrm L^{ij,1}_{\delta,h}\xi$ and $\mathrm
L^{ij,2}_{\delta,h}\xi$
are weighted trapezoidal rules applied to $\psi_{\underline
u,\delta}(\rho,\theta)\xi( \r^{ji} (\underline u)
+\rho\underline e(\theta))$, where $\psi_{\underline u,\delta}$ is given
by \eqref{eq:4.4.e}. Since $\psi_{\underline u,\delta}$ satisfies conditions
\eqref{eq:4.4}, it follows that
for all $\xi \in \mathcal C^0(\overline I_2)$,
\begin{equation}\label{eq:6.2.a}
\|\mathrm L^{ij,1}_{\delta,h}\xi\|_{\infty,I_2}+\|\mathrm
L^{ij,2}_{\delta,h}\xi\|_{\infty,I_2}\le C \delta
\|\xi\|_{\infty,I_2}.
\end{equation}
By \eqref{eq:6.2.a} and the error bound
for the Hermite interpolation \eqref{eq:Hermite} we can estimate
\begin{eqnarray}
& & \hspace{-2cm} \|(\mathrm A^{ij}_{\mathrm{sing},\delta,h}-\mathrm
A^{ij}_{\mathrm{sing},\delta,\mathrm{herm}}) \xi\|_{\infty,I_2} \nonumber \\
& \le & \|\mathrm L^{ij,1}_{\delta,h}(\xi-(\mathrm
H_{h_j/m,d}\otimes\mathrm I)\xi)\|_{\infty, I_2}+ \|\mathrm
L^{ij,2}_{\delta,h}(\xi-(\mathrm I \otimes \mathrm
H_{h_j/m,d})\xi)\|_{\infty, I_2}\nonumber\\
& \le & C_d \delta h_j^{2d+2} m^{-(2d+2)} \big(
\|\partial_{u_1}^{2d+2}\xi\|_{\infty,I_2}+\|\partial_{u_2}^{2d+2}\xi\|_{\infty,
I_2}\big).
\label{mod:eq1}
\end{eqnarray}
Therefore, using the inverse inequalities \eqref{eq:invIneqSobolev}
and \eqref{eq:invIneqLinfty} in the right-hand side of
\eqref{mod:eq1} it follows that
\[
\|(\mathrm A^{ij}_{\mathrm{sing},\delta,h}-\mathrm
A^{ij}_{\mathrm{sing},\delta,\mathrm{herm}}) \mathrm
P_{N_j}\xi\|_{\infty,I_2}\le
C_d m^{-(2d+2)} h_j^{-1}\delta \|\xi\|_0 \qquad \forall \xi \in H^0.
\]
This inequality and the stability estimate for the interpolation
operator $\mathcal Q_h$ (Lemma \ref{lemma:aux01}) imply that
\begin{eqnarray*}
\|(\mathcal B_{\delta,h}-\mathcal
B_{\delta,\mathrm{herm}})\boldsymbol\xi\|_0 &=& \| \mathcal Q_h
(\mathcal A^{\mathrm{sing}}_{\delta,h}-\mathcal
A^{\mathrm{sing}}_{\delta,\mathrm{herm}}) \mathcal P_h
\boldsymbol\xi\|_0\le  \|(\mathcal
A^{\mathrm{sing}}_{\delta,h}-\mathcal
A^{\mathrm{sing}}_{\delta,\mathrm{herm}}) \mathcal P_h
\boldsymbol\xi\|_{\infty,I_2}\\
& \le & C_d m^{-(2d+2)} \delta\, h^{-1}\|\boldsymbol\xi\|_0 \le
C_{d,\beta} m^{-(2d+2)} h^{\beta-1}\|\boldsymbol\xi\|_0,
\end{eqnarray*}
which proves the first assertion of the statement. Invertibility of
$\mathcal B_{\delta,\mathrm{herm}}$ as well as the uniform bound for
their inverses follow from Theorem \ref{theo:main1}.
\end{proof}

In the last results of this section the constants will depend on
$\beta$ and $d$, without this dependence being made explicit.

\begin{proposition}
Assume that $\delta \approx h^\beta$ with $\beta \in (0,1)$ and that
$m=m(h)$ is chosen so that $m^{-(2d+2)} h^{\beta-1}\to 0$. Let
$\boldsymbol\phi$ and $\boldsymbol\phi_{\mathrm{herm}}$ be the
respective solutions of $\mathcal B_\delta \boldsymbol\phi=\mathbf
U$ and $\mathcal
B_{\delta,\mathrm{herm}}\boldsymbol\phi_{\mathrm{herm}}=\mathbf U$.
Then for every $\varepsilon >0$ and $t\ge s\ge 0$ with $t>1$, there
exists a constant $C_{s,t,\varepsilon}>0$ such that
\begin{equation}\label{mod:eq3}
\|\boldsymbol\phi-\boldsymbol\phi_{\mathrm{herm}}\|_s \le C_{s,t,\varepsilon}
\big(
h^{t(1-\beta)-s}\|\boldsymbol\phi\|_t
+h^{2d+2+\beta-s}m^{-(2d+2)}\|\boldsymbol\phi\|_{2d+3+\varepsilon}\big).
\end{equation}
Consequently if $m(h)$ is chosen so that $m^{-(2d+2)} \le C
h^{1-\beta + r}$ for some $r\ge 0$,
\begin{equation}\label{mod:eq4}
\|\boldsymbol\phi-\boldsymbol\phi_{\mathrm{herm}}\|_0 \le C
h^{2d+3+r}\|\boldsymbol\phi\|_{(2d+3+r)/(1-\beta)}.
\end{equation}
\end{proposition}

\begin{proof}
Using \eqref{mod:eq1}, Lemma \ref{lemma:aux01} and the Sobolev
embedding theorem (to bound $(2d+2)$-th order partial derivatives by the
$H^{2d+3+\varepsilon}$ norm), we can easily prove that
\begin{equation}\label{mod:eq2}
\|(\mathcal B_{\delta,\mathrm{herm}}-\mathcal
B_{\delta,h})\boldsymbol\xi\|_0 \le C_\varepsilon \delta h^{2d+2}
m^{-(2d+2)} \|\boldsymbol\xi\|_{2d+3+\varepsilon}.
\end{equation}
Proceeding as in the proof of Theorem \ref{theo:main2} (see
\eqref{eq:01:proofTheoMain02} and the bounds thereafter), we can
introduce the discrete operator using Proposition \ref{mod:prop2}
and bound
\begin{eqnarray*}
\|\boldsymbol\phi-\boldsymbol\phi_{\mathrm{herm}}\|_0 &\le& C \|
\mathcal
B_{\delta,\mathrm{herm}}(\boldsymbol\phi-\boldsymbol\phi_{\mathrm{herm}})\|_0
\\
& \le & C \big( \|\mathbf U-\mathcal Q_h \mathbf U\|_0+\|(\mathcal
B_{\delta,h}-\mathcal B_\delta)\boldsymbol\phi\|_0+ \|(\mathcal
B_{\delta,\mathrm{herm}}-\mathcal
B_{\delta,h})\boldsymbol\phi\|_0\big).
\end{eqnarray*}
For the first two terms we use the same estimates as in the proof of
Theorem \ref{theo:main2}, whereas the third term is bounded using
\eqref{mod:eq2}. This gives \eqref{mod:eq3} with $s=0$. For larger
values of $s$ we apply the same argument as in Theorem
\ref{theo:main2} (use inverse inequalities and the properties of
$\mathcal P_h$). The estimate \eqref{mod:eq4} is a straightforward
consequence of \eqref{mod:eq3} and the choices made for $m$.
\end{proof}

\begin{theorem}
Assume that $\delta \approx h^\beta$ with $\beta \in (0,1)$ and that
$m=m(h)$ is chosen so that $m^{-(2d+2)}\le C h^{1-\beta+r}$ for some $r>0$. Let
$\psi$ be the solution of \eqref{eq:intEqn} and $\psi_{\mathrm{herm}}$ be given
by
\[
 \psi_\mathrm{herm}(\r):=\sum_{i\in\mathcal I(\r)} 2\Big(U^i-\sum_{j=1}^J\mathrm
A^{ij}_{\rm reg,\delta}\mathrm D_{N_j}\phi_\mathrm{herm}^j-\sum_{j=1}^J \mathrm
A^{ij,\mathrm{herm}}_{
\mathrm{sing},\delta,h,m,d}\phi_\mathrm{herm}^j\Big)\left((\r^i)^{-1}(\r)\right)
\]
where
$\boldsymbol\phi_{\mathrm{herm}}=(\phi^1_{\mathrm{herm}},\ldots,\phi^J_{\mathrm{
herm}})$is the solution of  $\mathcal
B_{\delta,\mathrm{herm}}\boldsymbol\phi_{\mathrm{herm}}=\mathbf U$.  Then,
denoting
\[
t:=(2d+3+r)/(1-\beta),
\]
we have the error estimate
\begin{equation}\label{eq:hermL2}
\|\psi- \psi_{\mathrm{herm}}\|_{L^2(S)} \le  C_r\big( |\log h|^{\frac32}
h^{t(1-\beta)}+h^{t-1}\big)\|\psi\|_{H^{t}(S)}.
\end{equation}
Finally, for all $\varepsilon>0$
\begin{equation}\label{eq:hermLinf}
\|\psi- \psi_{\mathrm{herm}}\|_{L^\infty(S)} \le  C_{r,\varepsilon}\big(
h^{t(1-\beta)-(1+\varepsilon)\beta}+h^{t-1}
\big)\|\psi\|_ { H^ { t } (S) }.
 \end{equation}
\end{theorem}

\begin{proof} This result can be established by arguments closely
  analogous to those used in the proof of Theorem \ref{theo:MAIN}.
\end{proof}

\paragraph{Acknowledgment}
OB gratefully acknowledges support from AFOSR and NSF. VD is partially
supported by Project MTM2010-21037. FJS is partially supported by the NSF-DMS
1216356 grant.  A portion of this research was carried
out when FJS was visiting professor at the University of Minnesota and
during short visits of FJS and VD to the California Institute of
Technology.

\bibliography{bibHOM}

\end{document}